\documentclass[10pt, reqno]{amsart}

\usepackage{graphicx}
\usepackage{xcolor}
\usepackage{multimedia,amssymb,amsmath}
\usepackage{tikz-cd}
\usepackage{multicol}
\usepackage{subcaption}
\usepackage[authoryear,round]{natbib}
\usepackage{stmaryrd}
\usepackage{bold-extra}
\usepackage{esint}
\usepackage{overpic}

\theoremstyle{plain}
\newtheorem{theorem}{Theorem}[section]
\newtheorem{lemma}{Lemma}[section]
\newtheorem{corollary}{Corollary}[section]
\newtheorem{proposition}{Proposition}[section]
\theoremstyle{definition}

\theoremstyle{definition}
\newtheorem{example}{Example}[section]
\theoremstyle{definition}

\theoremstyle{definition}
\newtheorem{condition}{Condition}[section]

\numberwithin{equation}{section}

\newcommand{\beginsupplement}{%
	\setcounter{table}{0}
	\renewcommand{\thetable}{S\arabic{table}}%
	\setcounter{figure}{0}
	\renewcommand{\thefigure}{S\arabic{figure}}%
}

\newcommand{\pd}{\textsf{PD}}

\newcommand{\lt}{\textsf{LT}}
\newcommand{\ltp}{\textsf{LT}_+}
\newcommand{\lts}{\textsf{LT}_{\textsf{s}}}
\newcommand{\ltu}{\textsf{LT}_{\textsf{u}}}
\newcommand{\ortho}{\textsf{O}}
\newcommand{\so}{\textsf{SO}}
\newcommand{\sym}{\textsf{Sym}}
\newcommand{\gl}{\textsf{GL}}
\newcommand{\m}{\textsf{M}}
\newcommand{\sk}{\textsf{Sk}}
\newcommand{\D}{\textsf{D}}
\newcommand{\Dp}{\textsf{D}_+}

\newcommand{\M}{\textsf{M}}
\newcommand{\V}{\textsf{V}}
\newcommand{\p}{\textsf{P}}
\newcommand{\im}{\textsf{im}}
\newcommand{\kernel}{\textsf{ker}}
\newcommand{\cone}{\textsf{Cone}}

\newcommand{\indep}{\raisebox{0.05em}{\rotatebox[origin=c]{90}{$\models$}}}

\newcommand{\RN}[1]{%
	\textup{\uppercase\expandafter{\romannumeral#1}}%
}

\DeclareMathOperator{\vvec}{\mathbf{vec}}
\DeclareMathOperator{\halfvec}{\mathbf{vech}}
\DeclareMathOperator{\rank}{\mathbf{rank}}
\DeclareMathOperator{\ssp}{\mathbf{span}}
\DeclareMathOperator{\expo}{\mathbf{expo}}

\newcommand{\RR}{\mathbb{R}}
\newcommand{\CC}{\mathbb{C}}
\newcommand{\ZZ}{\mathbb{Z}}

\newcommand{\g}{\mathfrak{g}}

\usepackage{relsize}
\newcommand{\T}{{\mathsmaller {\rm T}}}

\makeatletter
\newcommand*\bigcdot{\mathpalette\bigcdot@{.5}}
\newcommand*\bigcdot@[2]{\mathbin{\vcenter{\hbox{\scalebox{#2}{$\m@th#1\bullet$}}}}}
\makeatother

\newcommand{\ind}{\;{\rm I\hspace{-2.3mm} 1}}

\begin{document}
	
	\title[]{Regression graphs and sparsity-inducing reparametrizations}

\author{Jakub Rybak}
\address{Jakub Rybak, Heather Battey: Department of Mathematics, Imperial College London, 180 Queen's Gate, London, SW7 2AZ, UK}	\email{jakub.rybak18@imperial.ac.uk h.battey@imperial.ac.uk}

\author{Heather Battey}

\author{Karthik Bharath}
\address{Karthik Bharath: School of Mathematical Sciences, University of Nottingham, University Park, Nottingham, NG7 2RD, UK.}
\email{Karthik.Bharath@nottingham.ac.uk}

\maketitle

\begin{abstract}
	That parametrization and sparsity are inherently linked raises the possibility that relevant models, not obviously sparse in their natural formulation, exhibit a population-level sparsity after reparametrization. In covariance models, positive-definiteness enforces additional constraints on how sparsity can legitimately manifest. It is therefore natural to consider reparametrization maps in which sparsity respects positive definiteness. The main purpose of this paper is to provide insight into structures on the physically-natural scale that induce and are induced by sparsity after reparametrization. The richest of the four structures initially uncovered can be generated, under a causal ordering, by the joint-response graphs studied by \cite{WermuthCox2004}, while the most restrictive is that induced by sparsity on the scale of the matrix logarithm, studied by \citet{Battey2017}. The Iwasawa decomposition of the general linear group, combined with the graphical-models interpretation, points to a class of reparametrizations for the chain-graph models \citep{Andersson}, with undirected and directed acyclic graphs as special cases. An important insight is the interpretation of approximate zeros, explaining the modelling implications of enforcing sparsity after reparameterization: in effect, the relation between two variables would be declared null if relatively direct regression effects were negligible and others manifested through long paths. The insights have a bearing on methodology, some aspects of which are presented. A detailed simulation uses the theoretical insights to further explore regimes under which reparametrization is beneficial.

\medskip 

\emph{Some key words:} Causality; Chain graphs; Graphical models; Matrix logarithm; Reparametrization; Sparsity.
\end{abstract}

\section{Introduction}\label{secIntro}

Sparsity, the existence of many zeros or near-zeros in some domain, plays at least two roles in statistics, depending on context: to aid interpretation; and to prevent accumulation of error incurred through estimation of nuisance parameters. There is now a large literature concerned with enforcing sparsity on sample quantities, having assumed that the corresponding population-level object is sparse. The present paper is concerned with the more fundamental question of whether there are parametrizations enjoying a population-level sparsity not present to the same extent in the original formulation. In other words, from a parametrization that is natural from a modelling point of view, we seek a sparsity-inducing reparametrization.

Inducement of population-level sparsity, through a traversal of parametrization space or data-transformation space, is a relatively unexplored area. Four isolated examples were unified from this perspective by Battey (2023), starting from the work on parameter orthogonalization (Cox and Reid, 1987). The development of Gaussian graphical models \citep[e.g.][]{Lauritzen, CW1996} is also somewhat in this vein, together with the important work on graphical modelling for extremes by \citet{EH2020}.  Here, as in the graphical models work, we focus on interpretation and insight at the population level, leaving for future development the important question of how to deduce the sparsity scale from data.

The motivating question for this paper is whether, for a broad enough class of covariance structures, not obviously sparse in their natural parameter domain, a non-trivial sparsity-inducing reparametrization can be deduced in which sparsity respects positive definiteness. By non-trivial, we mean that it is possible to discriminate more effectively on the new scale between elements that are large and elements that are small. This rules out artificially sparse reparametrizations such as $\Sigma \mapsto c\Sigma$ for $c>0$ close to zero. \cite{Battey2017} and \cite{RybakBattey2021} provided a proof of concept. Their position was that covariance matrices and their inverses are often nuisance parameters, and it is therefore arguably more important that the sparsity holds to an adequate order of approximation in an arbitrary parametrization, than that the sparse parametrization has interpretable zeros. An example of this type is linear discriminant analysis, where the interest parameter is the linear discriminant. In the case of undirected Gaussian graphical models, the precision matrix is the interest parameter by virtue of the interpretation ascribed to its zeros. Thus, both aspects are of interest and are addressed here. A third type of situation, to which the present paper does not contribute, is when the covariance matrix is a nuisance parameter that has a known structure up to a low-dimensional parameter. This situation is common in some settings, for instance in the analysis of split-plot or Latin square designs with block effects treated as random. 

The starting point for the paper is the identification of new parametrizations in which sparsity conveniently manifests in a vector space. For these, we uncover the structure induced on the original scale through zeros in the new parameter domain, as well as the converse result: that matrices encoding such structure possess exact zeros after reparametrization. The scope is considerably broadened through the possibility of approximate zeros, of which there may be many more in the new parameter domain than in the original or inverse domains. An important insight is therefore the interpretation of approximate zeros, as this explains the modelling implications of enforcing sparsity after reparameterization. Under a, perhaps notional, causal ordering, the relation between two variables would be declared null if relatively direct regression effects were negligible and other effects manifested through long paths. Section \ref{secChain} unifies old and new parametrizations via a class of matrix decompositions representing the chain graphs, allowing for both directed and undirected edges, and recovering the four fundamental parametrizations as special cases.

Because the population-level sparsity manifests in a vector space, imposition of sparsity on sample analogues produces a covariance estimate that necessarily respects the positive definite cone constraint. The benefits are transferred to any sensible methodology, and we present one approach with high-dimensional statistical guarantees in \S \ref{secStatImpl}.

\section{Notation} \label{secNotation}

Table \ref{tableSubsets} indicates subsets of the vector space $\m(p)$ of $p \times p$ real matrices that are referenced throughout the paper. Most are matrix Lie groups with matrix multiplication as the group operation; those that are vector spaces have a natural matrix basis.

\begin{table}[]
	\begin{center}
		\begin{tabular}{l|c|c}
			Notation & Matrix subset &  Basis notation  \\ 
			\hline 
			$\pd(p)$ & symmetric positive definite matrices &     \\
			$\sym(p)$ & symmetric matrices &  $\mathcal{B}_{sym}$  \\
			$\D(p)$ & diagonal matrices &  $\mathcal{B}_{diag}$  \\
			$\gl(p)$ & nonsingular matrices  &   \\
			$\p(p)$ & permutation matrices & \\
			$\ortho(p)$ & orthogonal matrices  &  \\
			$\so(p)$ & special orthogonal matrices (determinant $+1$) &  \\
			$\sk(p)$ & skew-symmetric matrices & $\mathcal{B}_{sk}$ \\
			$\lt(p)$ & lower-triangular matrices & $\mathcal{B}_{lt}$  \\
			$\ltu(p)$ & lower-triangular with unit diagonal entries & \\
			$\lts(p)$ & strictly lower-triangular matrices & $\mathcal{B}_{lts}$ \\
			\hline
		\end{tabular}
	\end{center}
	\caption{Matrix subsets of $\m(p)$. \label{tableSubsets}}
\end{table}

Versions of the diagonal and lower triangular matrices with positive diagonal elements are differentiated using the subscript $+$. A generic vector subspace of $\M(p)$ is written $\V(p)$. Also extensively referenced is $\cone_p \subset \sym(p)$, the interior of a convex cone within $\sym(p)$ excluding the origin, as formalized in Appendix \ref{sec:PD} of the supplementary material. For the purpose of the present paper, $\cone_p$ can be thought of as the constrained set of $p(p+1)/2$ elements constituting the upper triangular part of a positive definite matrix. The symbol $\oplus$ denotes the direct sum of two vector spaces; $A\oplus B$ also represents a block-diagonal matrix with blocks $A$ and $B$. The index set $\{1,\ldots, p\}$ is written $[p]$. The length of a vector $v$ is written $\dim(v)$ and the cardinality of a finite set $\mathcal{A}$ is written $|\mathcal{A}|$. 

The sets of basis matrices in Table 1 are constructed from the canonical basis vectors $e_1,\ldots,e_p$ for $\RR^p$, where $e_i\in\RR^p$ is a zero vector with 1 as its $i$th component. Specifically $\mathcal{B}_{sym}:=\{B_1,\ldots,B_{p(p+1)/2}\}$ consists of $p(p-1)/2$ non-diagonal matrices $e_j e_k^\T + e_ke_j^\T$ for $j< k$ and $p$ diagonal matrices of the form $e_je_j^\T$, the latter also constituting $\mathcal{B}_{diag}$; $\mathcal{B}_{sk}:=\{B_1,\ldots,B_{p(p-1)/2}\}$ consists of skew symmetric matrices $e_j e_k^\T - e_ke_j^\T$ for $j<k$; $\mathcal{B}_{lt}:=\{B_1,\ldots,B_{p(p+1)/2}\}$, consists of lower triangular matrices $e_{k}e_j^\T$, $j\leq k$; and $\mathcal{B}_{lts}:=\{B_{1},\ldots,B_{p(p-1)/2}\}$ consists of strictly lower triangular matrices $e_{k}e_j^\T$ with $j<k$. The matrix exponential of a square matrix $A$ is defined as
\begin{equation}\label{eqMatrixExp}
e^A = \sum_{k=0}^\infty \frac{A^k}{k!}.
\end{equation}
Conversely, if a matrix logarithm $L$ of a square matrix $M$ exists, then $M=e^L$, as defined in equation \eqref{eqMatrixExp}. Conditions for the existence and uniqueness of the a matrix logarithm are given in Appendix \ref{secInversion}.

For random variables $ X_{1} $, $ X_{2} $ and $ X_{3}$, the statement that $ X_1 $ is conditionally independent of $ X_2 $ given $ X_3 $ is written in the familiar notation $ X_{1} \indep X_{2} | X_3 $, unconditional independence being expressed similarly as $ X_1 \indep X_2$. 

\section{Reparametrization}\label{secRepara}

The set $\cone_p$ is, from one perspective, the natural parameter domain for parametrizing the manifold $\pd(p)$. The question we seek to address is whether there is another parameter domain that is less direct, but in which a population-level sparsity is present, ideally with interpretable zeros or near-zeros. This is most compelling in the absence of considerable sparsity on the original or inverse scales; in that case, the reparametrization is said to be sparsity-inducing. The problem is initially addressed from the opposite direction, by considering the parameter domains in which sparsity can be fruitfully represented, and then studying the form of multivariate dependencies that are implied by exact zeros in these non-standard parametrizations. The deduced structures in \S \ref{secStructure} are both necessary and sufficient, showing that, in terms of exact zeros, the matrices are at least as sparse after reparametrization. We clarify in subsequent sections the extent to which, and manner in which, a covariance or precision matrix might be sparser after reparametrization, and the implications for estimation and interpretation.

In a sense to be clarified, a construction based on the chain graph representations of \S \ref{secChain} subsumes the dependence structures of \cite{Battey2017} and \cite{RybakBattey2021}, plus those identified in \S \ref{secStructure}, into a broader sparsity class. This ultimately points to a class of structures in which sparsity manifests in a vector space, and that enjoy a graphical models interpretation on the physically natural scale when a (full or partial) causal ordering is present. 

From an initial parameterization $\cone_p \to \pd(p)$ of $\pd(p)$, formalized in Appendix \ref{sec:PD}, we study four reparameterizations arising from maps $\cone_p \to \RR^{p(p+1)/2}$ such that sparsity in the new domain $\RR^{p(p+1)/2}$ respects the positive definiteness constraint on $\Sigma$. In other words, an arbitrary configuration of zeros in the new parameter domain $\RR^{p(p+1)/2}$ does not violate positive definiteness of $\Sigma$ or $\Sigma^{-1}$. The four fundamental parametrizations discussed in this work are, with $D(d)=\text{diag}(d_1,\ldots,d_p)$,
\begin{align}
\label{eqRepara}
\nonumber \alpha \mapsto \Sigma_{pd}(\alpha)\hspace{0.45cm}&:= e^{L(\alpha)}, \quad && L(\alpha) \in \sym(p), &&\alpha\in \RR^{p(p+1)/2}; \\
\nonumber (\alpha, d) \mapsto \Sigma_o(\alpha, d) \hspace{0.26cm} &:= e^{L(\alpha)}e^{D(d)} (e^{L(\alpha)})^\T, \quad && L(\alpha) \in \sk(p), && \alpha\in\RR^{p(p-1)/2}, \;\; 
d\in\RR^p; \\
\nonumber \alpha \mapsto \Sigma_{lt}(\alpha)\hspace{0.55cm}&:= e^{L(\alpha)} (e^{L(\alpha)})^\T , \quad && L(\alpha) \in \lt(p), &&\alpha\in \RR^{p(p+1)/2}; \\
(\alpha, d) \mapsto \Sigma_{ltu}(\alpha, d) \hspace{0.05cm}&:= e^{L(\alpha)}e^{D(d)} (e^{L(\alpha)})^\T, \quad && L(\alpha) \in \lts(p), && \alpha\in\RR^{p(p-1)/2}, \;\; 
d\in\RR^p.
\end{align}
That the four maps \eqref{eqRepara} are fundamental emerges from the Iwasawa decomposition of the group of nonsingular matrices, owing to which the paper has an enlightening group-theoretic underpinning. We have placed most of this discussion in the supplementary material in favour of a more broadly accessible exposition, but we return briefly to the Iwasawa decomposition in \S \ref{secIwasawa}.

For each of the four reparametrization maps, the parameter domain $\RR^{p(p+1)/2}$ is identified with a different vector space of the same dimension. These are, respectively, $\sym(p)$, $\sk(p)\times \D(p)$, $\lt(p)$, and $\lts(p)\times \D(p)$. The subscripts on $\Sigma$ in the parametrizations indicate which of the matrix sets, $\pd(p)$, $\so(p)$, $\ltp(p)$ and $\ltu(p)$ respectively, prescribed coordinates in terms of $\alpha$, are represented as the image of the matrix exponential. In each case $L(\alpha)\in\V(p)$ depends on $\alpha$ through its expansion
\begin{equation}\label{eqBasisExp}
L(\alpha)=\alpha_1 B_{1} + \cdots + \alpha_{m} B_{m}
\end{equation} 
in the canonical basis for $\V(p)$, as specified in \S \ref{secNotation}. The canonical basis is part of the definition of the reparametrization maps. Appendices \ref{sec:PD} and \ref{secInversion} establish the legitimacy of the maps, this hinging in the case of $\Sigma_o$ and $\Sigma_{lt}$ on some constraints on $\alpha$ or conditions on the covariance matrices. Thus, $\Sigma_{pd}$ and $\Sigma_{ltu}$ are favourable in this respect.

Another parametrization in which sparsity respects positive definiteness is in terms of the Cholesky factors themselves, rather than the matrix logarithm of the Cholesky factors. This is related to the $\Sigma_{lt}$ and $\Sigma_{ltu}$ parametrizations as discussed in \S \ref{secNotionalGaussian}. While sparsity in the Cholesky factors has not been explicitly considered, several authors have modelled the Cholesky components in terms of covariates; \citet{Pourahmadi1999} appears to have started this line of enquiry.

\section{Structure of $\Sigma(\alpha)$ induced by, and inducing, exact zeros in $\alpha$}\label{secStructure}

Consider the matrices $L\in\V(p)\subset \M(p)$ from \eqref{eqRepara}, all of which can be written in terms of a canonical basis $\mathcal{B}$ of dimension $m$ as $L(\alpha)=\alpha_1 B_1 + \cdots +\alpha_m B_m$. Suppose now that $\alpha=(\alpha_1,\ldots,\alpha_m)$ is sparse in the sense that $\|\alpha\|_0=\sum_j\ind\,\{\alpha_j\neq 0\}=s^*\ll m$. Corollaries \ref{corollBattey}--\ref{corollLDU} to be presented specify the structure induced on each of $\Sigma_{pd}(\alpha)$, $\Sigma_o(\alpha,d)$, $\Sigma_{lt}(\alpha)$ and $\Sigma_{ltu}(\alpha,d)$ through sparsity of $\alpha$. They also establish the converse result: that matrices possessing such structure produce exact zeros in $\alpha$, in a number to be quantified to the extent feasible. The basis coefficients $ \alpha $ and the structures expounded in Corollaries \ref{corollBattey}--\ref{corollLDU} turn out to be directly interpretable and are elucidated in \S \ref{secNotionalGaussian} and \S \ref{secPDO} from a graphical modelling perspective, assuming for some of the latter statements an underlying Gaussian model with a causal ordering. Theorem \ref{prop2.3} is a general result that applies in different ways to the four cases, resulting in Corollaries \ref{corollBattey}--\ref{corollLDU}.

For a matrix $M\in \M(p)$ possessing a matrix logarithm $M=e^L$,  let $ d_{r}^{*}$ and $ d_{c}^{*} $ denote the number of non-zero rows and columns of $ L $ respectively, and let $ d^{*} $ denote the number of indices for which the corresponding row or column of $ L $ is non-zero. Theorem \ref{prop2.3} establishes general conditions for logarithmic sparsity of a matrix $M=e^L \in \M(p)$. 

\begin{theorem}\label{prop2.3}
	Consider $M=e^L\in\M(p)$ where $L\in\V(p)$, a vector space with canonical basis $\mathcal{B}$ of dimension $m$. The matrix $M$ is logarithmically sparse in the sense that $L=L(\alpha)=\alpha_1 B_1 + \cdots + \alpha_m B_m$, $B_j\in \mathcal{B}$ with $\|\alpha\|_0=s^*$ if and only if $M$ has $p-d_r^*$ rows of the form $e_j^\T$ for some $j\in[p]$, all distinct, and $p-d_c^*$ columns of the form $e_j$. Of these, $p-d^*$ coincide after transposition. If $M$ is normal, i.e.~satisfies $M^\T M=M M^\T$, then $d_r^*=d_c^*=d^*$.
\end{theorem}

The quantities $d^*_r$, $d^*_c$ and $d^*$ are related to $s^*$ when $\alpha$ is sparse. A loose bound is $\max\{d^{*}_{r},d^{*}_{c}\} \leq 2s^{*}$, but $\max\{d^{*}_{r},d^{*}_{c}\}$ can be considerably smaller than this, as it depends on the configuration of basis elements picked out by the sparse $\alpha$. Indeed, $\max\{d^{*}_{r},d^{*}_{c}\}\ll p$ is possible even when $s^*$ exceeds $p$, provided that the configuration of non-zero elements of $\alpha$ produces zero rows or columns of $L$. 

Figure \ref{fig:M_example} shows an example of a structure of $ M $ established by Theorem \ref{prop2.3}. In particular settings, where the form of $\V(p)$ is made explicit, there may be additional structure, e.g.~lower triangular, that is not reflected in Figure \ref{fig:M_example}.

\begin{figure}%[h!] 
	\centering
	\includegraphics[trim=0.6in 0.6in 0.6in 0.7in, clip, height=0.25\paperwidth]{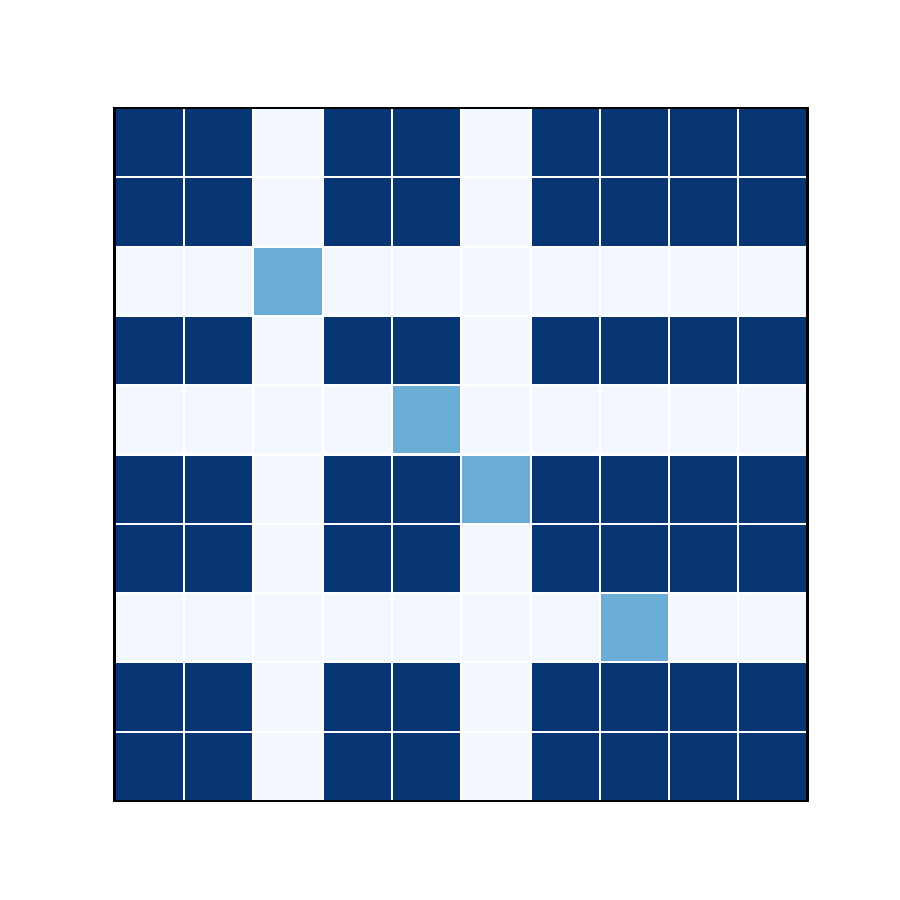}
	\caption{Example of a structure of $ M $ as established in Theorem \ref{prop2.3} with $ p = 10 $, $ d^{*}_{r} = 7 $, $ d^{*}_{c} = 8 $ and $ d^{*} = 9 $. The entries that are zero by Theorem \ref{prop2.3} are light blue, those equal to one are medium blue, and the remaining entries, whose values are unconstrained, are dark blue.}
	\label{fig:M_example}
\end{figure}

Corollaries \ref{corollBattey} and \ref{corollRybak} to be presented are not new. However, their proofs in Appendix \ref{appProof} are new, and presented in terms of the general formulation of Theorem \ref{prop2.3}. 
Although too restrictive to cover the new parameterizations $\Sigma_{lt}$ and $\Sigma_{ltu}$, and more general ones of \S \ref{secChain}, the elementary demonstration of Battey (2017) provides some insight into the nature of the considerations. The proof entailed showing that for a matrix $L(\alpha)\in\sym(p)$, represented in terms of the canonical basis $\mathcal{B}_{sym}$, a basis for $\im(L):=\{y\in\RR^p: Lx=y, \;\; x\in \RR^p\}$ is provided both by the eigenvectors of $\Sigma$ corresponding to the non-unit eigenvalues, and by $\mathcal{L}(\alpha)$, the set of unique non-zero column vectors of the basis matrices picked out from $\mathcal{B}_{sym}$ by the non-zero elements of $\alpha$. It follows that the rank of $L(\alpha)$, i.e.~the dimension of $\im(L)$, is the equal to $d^*=|\mathcal{L}(\alpha)|<2s^*$. The precise value of $d^*$ depends on the configuration of $\alpha$, but can be specified in expectation if the support of $\alpha$ is taken as a simple random sample of size $s^*$ from $[p(p+1)/2]$. Since the column space and the row space of a symmetric matrix coincide, the non-trivial eigenvectors of $L(\alpha)$ have only $d^*$ non-zero entries, which corresponds to a very specific structure on $\Sigma(\alpha)$, summarized in Corollary \ref{corollBattey}. Corollaries \ref{corollRybak} -- \ref{corollLDU} are the analogous results for the other three parametrizations from \eqref{eqRepara}.

\begin{corollary}\label{corollBattey}
	The image of the map $\alpha \mapsto \Sigma_{pd}(\alpha)=e^{L(\alpha)}$ is logarithmically sparse in the sense that $\|\alpha\|_{0}=s^*$ in the basis representation for $L(\alpha)$ if and only if $\Sigma$ is of the form $\Sigma=P \Sigma^{(0)} P^{\T}$, where $P\in\p(p)$ is a permutation matrix and $\Sigma^{(0)}=\Sigma^{(0)}_1 \oplus I_{p-d^*}$ with $\Sigma^{(0)}_1\in\pd(d^*)$ of maximal dimension, in the sense that it is not possible to find another permutation $P\in\p(p)$ such that the dimension of the identity block is larger than $p-d^*$.
\end{corollary}

\begin{corollary}\label{corollRybak}
	For an arbitrary diagonal component $D=\text{diag}\{d_1,\ldots,d_p\}$, the image of the map $\alpha \mapsto \Sigma_{o}(\alpha)=e^{L(\alpha)}e^D (e^{L(\alpha)})^\T$ is logarithmically sparse in the sense that $\|\alpha\|_{0}=s^*$ in the basis representation for $L(\alpha)$ if and only if $\Sigma$ is of the form $\Sigma=P \Sigma^{(0)} P^{\T}$, where $P\in\p(p)$ is a permutation matrix and $\Sigma^{(0)}=\Sigma^{(0)}_1 \oplus D_{p-d^*}$, where $D_{p-d^*}\in\D(p-d^*)$ and $\Sigma^{(0)}_1\in\pd(d^*)$ is of maximal dimension, in the sense that it is not possible to find another permutation $P\in\p(p)$ such that the dimension of the diagonal block is larger than $p-d^*$.
\end{corollary}

\begin{corollary}\label{corollLT}
	The image of the map $\alpha \mapsto \Sigma_{lt}(\alpha)=e^{L(\alpha)}(e^{L(\alpha)})^\T$ is logarithmically sparse in the sense that $\|\alpha\|_{0}=s^*$ in the basis representation for $L(\alpha)$ if and only if $\Sigma$ is of the form 
	$\Sigma=V V^\T$
	, where $V=I_p + \Theta$ and $\Theta\in\ltp(p)$ has $p-d_r^*$ zero rows and $p-d_c^*$ zero columns, of which $p-d^*$ coincide after transposition.
\end{corollary}

\begin{corollary}\label{corollLDU}
	For an arbitrary diagonal component $D=\text{diag}\{d_1,\ldots,d_p\}$, the image of the map $\alpha \mapsto \Sigma_{ltu}(\alpha)=e^{L(\alpha)}e^D(e^{L(\alpha)})^\T$ is logarithmically sparse in the sense that $\|\alpha\|_{0}=s^*$ in the basis representation for $L(\alpha)$ if and only if $\Sigma$ is of the form $\Sigma=U \Psi U^\T$, where $\Psi=e^D\in\Dp(p)$, $U=I_p + \Theta$ and $\Theta\in\lts(p)$ has $p-d_r^*$ zero rows and $p-d_c^*$ zero columns, of which $p-d^*$ coincide after transposition.
\end{corollary}

Zeros in $ \alpha $ thus produce structured patterns of zeros in $\Sigma$ and $\Sigma^{-1}$ in parametrizations $ \Sigma_{pd} $ and $ \Sigma_{o}$, and structured patterns of zeros in the Cholesky factors and lower-triangular matrices of the LDL decomposition in parametrizations $\Sigma_{lt}$ and $\Sigma_{ltu}$ respectively. These structures are interpreted in \S \ref{secNotionalGaussian}. 

Although arbitrary constraints on the value of $s^*$ are avoided in Theorem \ref{prop2.3} and the subsequent statements, a small value, e.g.~$s^*<p/2$, is guaranteed both to generate and to be implied by a simplification in the underlying conditional independence graph, under a notional Gaussian model. Relatively large values of $s^*$ can also entail graphical reduction in many cases, in the sense of introducing conditional independence relations relative to the saturated case. The situation of large $s^*$ is, however, less relevant from the standpoint of inducing sparsity through reparametrization. As an illustration, there are $p(p+1)/2$ basis elements for $L$ in the $\Sigma_{pd}$ parametrization. For $\alpha$ to induce a pattern of zeros in $\Sigma$ of the type discussed in Corollary \ref{corollBattey},  $L$ needs to have a zero row, which requires only $p$ zero coefficients in the basis expansion of $L$. Thus, $s^*$ can be as large as $s^{*}=p(p-1)/2$ for the structure to hold. 

To make a comparison between different structures of $ \Sigma(\alpha) $ more explicit, we consider a simple example with $p=5$. For the parametrizations $\Sigma_{pd}$ and $\Sigma_{o}$ we set $d^{*}=3$, corresponding to $s^{*}=6$ and $s^{*}=3$ respectively. For $\Sigma_{ltu}$ we consider two cases: $\Sigma_{ltu}^{r}$, for which $d_{r}^{*} < p$, $d^{*}_{c} = p$ (this serving as the definition of $\Sigma_{ltu}^{r}$); and $\Sigma^{c}_{ltu}$, for which $d_{r}^{*} = p$, $d^{*}_{c} < p$; in both cases  $s^{*} = 6$. The resulting covariance structure is depicted in Figure \ref{fig:sigma_example}. If the map $(\alpha, d) \mapsto \Sigma_{ltu}(\alpha, d)$ is instead replaced by an essentially equivalent representation $(\alpha, d) \mapsto \Sigma_{utu}(\alpha, d)$ in terms of upper triangular matrices, the analogous structures $\Sigma_{utu}^r(\alpha, d)$ and $\Sigma_{utu}^c(\alpha, d)$ are the same as for $\Sigma_{ltu}^c(\alpha, d)$ and $\Sigma_{ltu}^r(\alpha, d)$ respectively. 

\begin{figure}%[h!]
	\vspace{-0.2cm}
	\centering
	\begin{subfigure}[b]{0.48\textwidth}
		\centering
		\includegraphics[trim=0.6in 0.6in 0.6in 0.7in, clip, height=0.17\paperwidth]{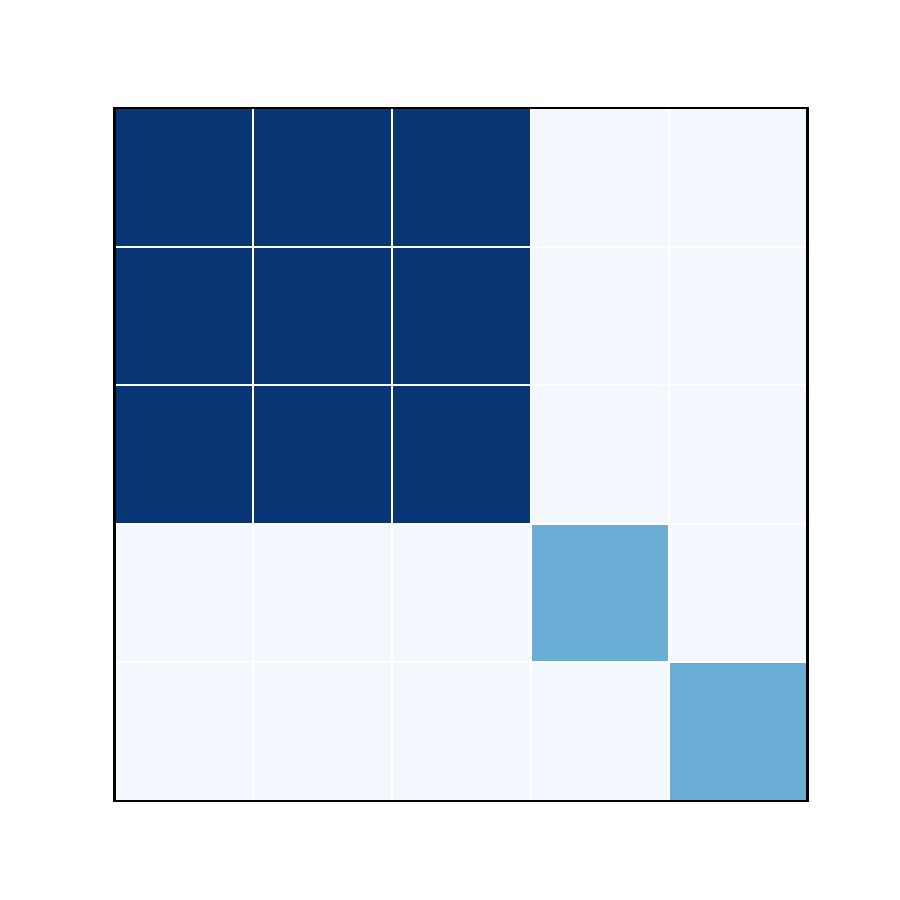}
		\caption{$ \Sigma_{pd} $}
	\end{subfigure}
	\hspace{-2cm}
	\begin{subfigure}[b]{0.48\textwidth}
		\centering
		\includegraphics[trim=0.6in 0.6in 0.6in 0.7in, clip, height=0.17\paperwidth]{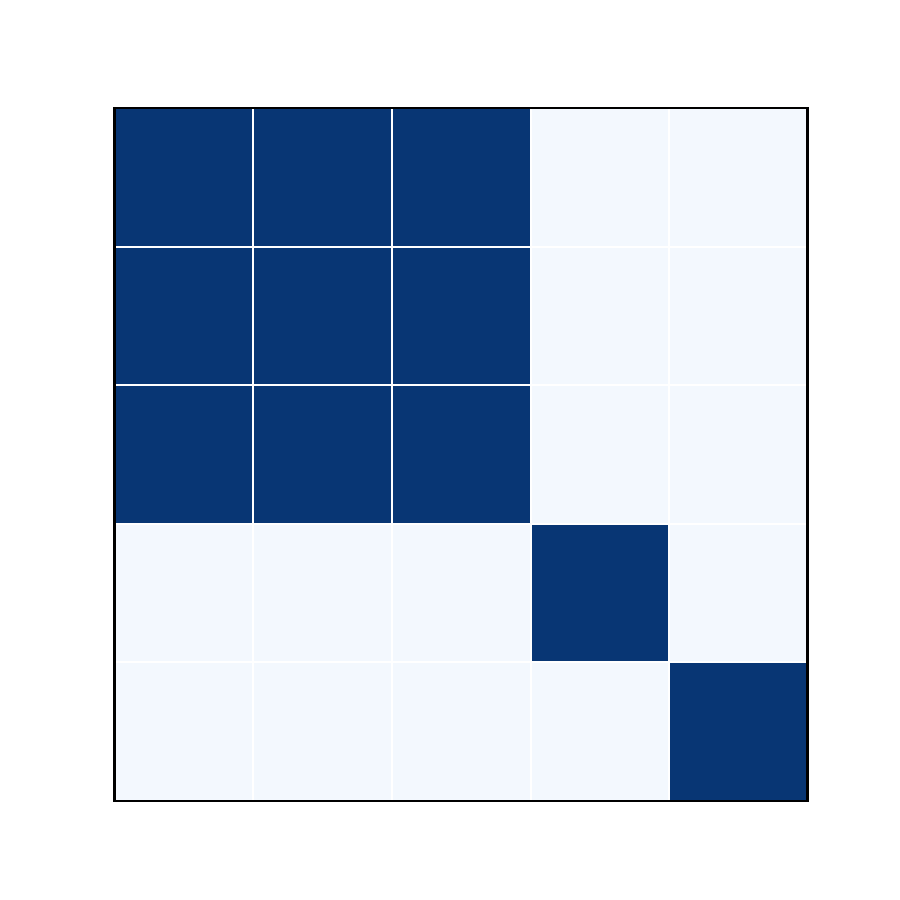}
		\caption{$ \Sigma_{o} $}
	\end{subfigure} 
	\\[-0ex]
	\par
	\centering
	\begin{subfigure}[b]{0.48\textwidth}
		\centering
		\includegraphics[trim=0.6in 0.6in 0.6in 0.7in, clip, height=0.17\paperwidth]{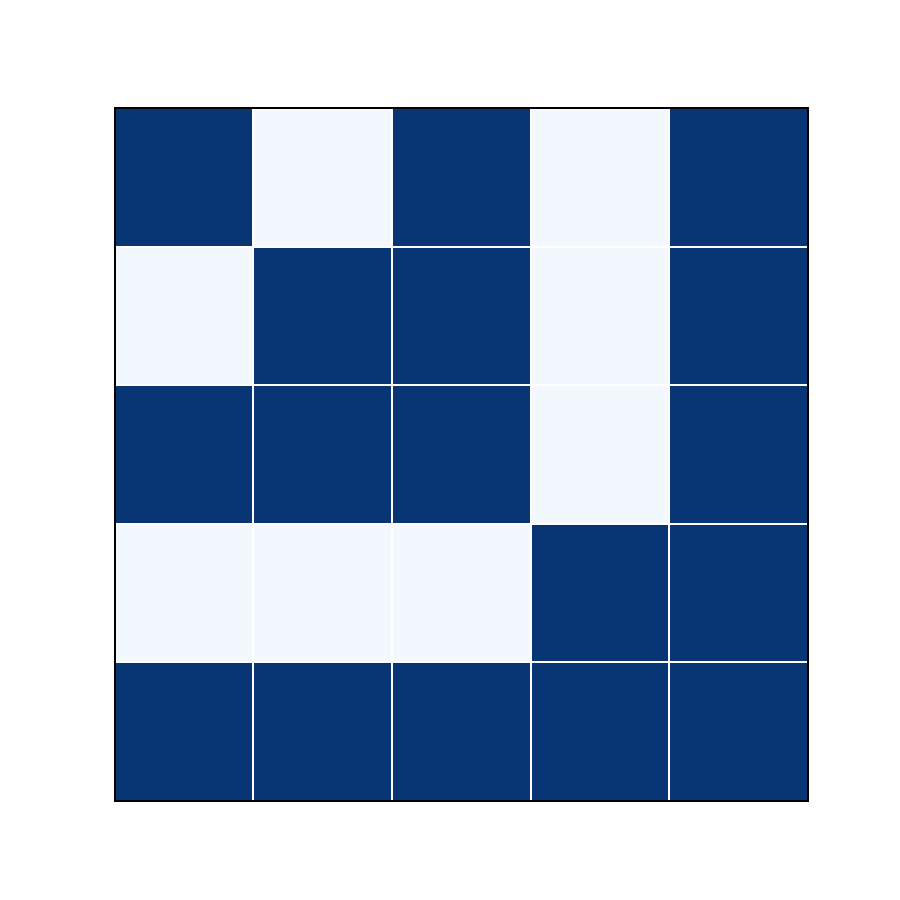}
		\caption{$ \Sigma_{ltu}^{r} $ }
	\end{subfigure}
	\hspace{-2cm}
	\begin{subfigure}[b]{0.48\textwidth}
		\centering
		\includegraphics[trim=0.6in 0.6in 0.6in 0.7in, clip, height=0.17\paperwidth]{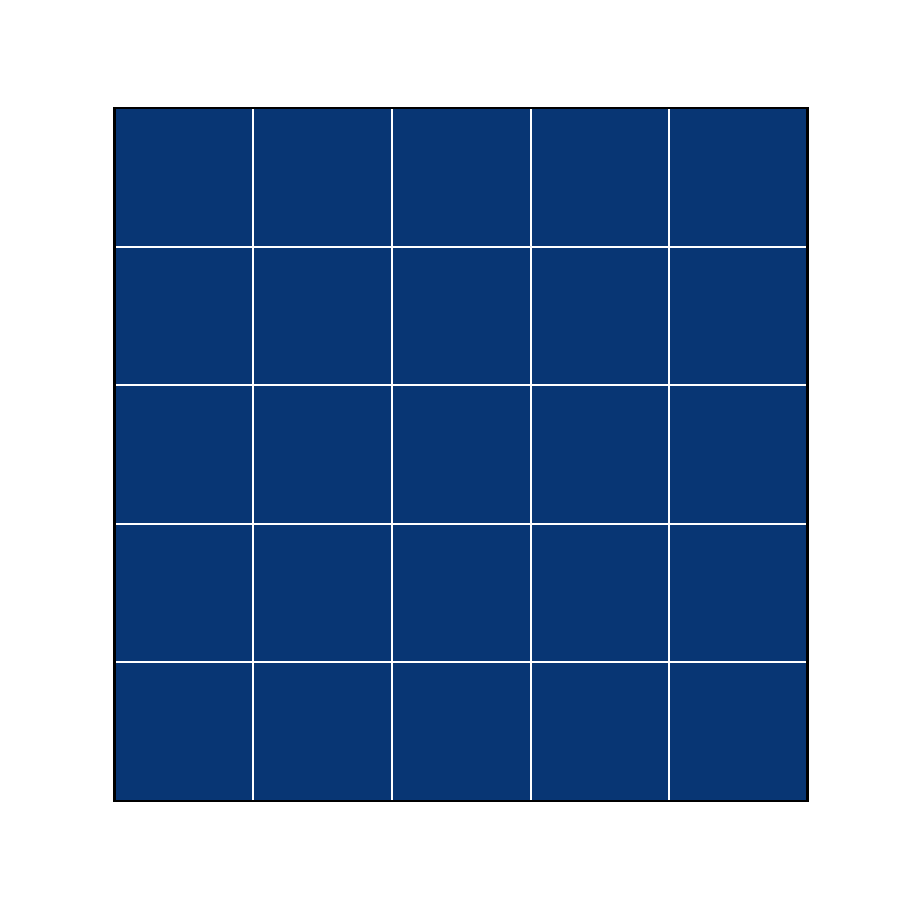}
		\caption{$ \Sigma_{ltu}^{c} $}
	\end{subfigure} 
	\caption{\label{fig:sigma_example} Structure of $ \Sigma(\alpha) $ induced by sparsity of $ \alpha $. Zero entries are depicted by light blue, unit entries by medium blue, and the unrestricted entries by dark blue.}
\end{figure}

For the $\Sigma_{ltu}$ parametrization, Figure \ref{fig:sigma_example} illustrates that, unlike a $\Theta$ with zero rows, a $\Theta$ with zero columns can generate a dense covariance matrix. Intuitively, for the same restriction on the sparsity of $ \alpha $, the corresponding covariance matrices $ \Sigma_{ltu}^{r} $ and $ \Sigma_{ltu}^{c} $ should represent relationships of similar inherent structural complexity. The following result confirms this intuition. Specifically, Lemma \ref{lemma:cov_submatrix} shows that, although $ \Sigma_{ltu}^{c} $ might have no zeros, the sparsity restriction on $ \alpha $ induces a low-rank structure on a submatrix of  $ \Sigma_{ltu}^{c} $. The existence of a low-rank structure has a statistical interpretation in terms of latent variables \citep[e.g.][]{FanPoet}. 

\begin{lemma}
	\label{lemma:cov_submatrix}
	Consider a random vector $ Y $, $ Y = (Y_{1}^\T, Y_{2}^\T, Y_{3}^\T)^\T $ with covariance matrix $ \Sigma $. The columns of $ \Theta $ in Corollary \ref{corollLDU} corresponding to $ Y_{2} $ are zero if and only if the submatrix
	\begin{align*}
	\begin{pmatrix}
	\Sigma_{11} & \Sigma_{13} \\
	\Sigma_{21} & \Sigma_{23}
	\end{pmatrix}
	\end{align*}
	of $ \Sigma $ has rank $\dim(Y_{1})$. 
\end{lemma}

If zeros in $d$ are allowed, the basis coefficients of the $\alpha\mapsto \Sigma_{lt}(\alpha)$ and $(\alpha,d)\mapsto\Sigma_{ltu}(\alpha, d)$ parametrizations are related. To see this, write 
\[
\Sigma_{lt}=\exp(L)\exp(L^\T)=\exp(L_{u} D_1) \exp((L_{u} D_1)^\T),
\] 
where $D_1\in \D(p)$ and $L_{u}\in \ltu(p)$. On writing $L_{u} D_1 = L_{s} + D_1$, where $L_{s}\in \lts(p)$ contains the strictly lower triangular part of $L_u D_1$,
\begin{equation}\label{eqCorrespondence}
\Sigma_{lt} = \exp(L)\exp(L^\T) = \exp(L_{s}) \exp(D_{1} + D_{1}^{\T}) \exp(L_{s}^\T),
\end{equation}
by the properties of the matrix exponential, which recovers the $\Sigma_{ltu}$ parametrization with  $D=2 D_1$. Thus if $d$ is allowed to have zeros, there is an exact relationship between the coefficients of the expansion of $\Sigma_{ltu}$ and those of $\Sigma_{lt}$. Since the transformation $ \Sigma_{ltu} $ provides a more convenient way of parametrizing regression graphs, subsequent discussion focuses on $ \Sigma_{ltu} $. Most insights derived for the $ \Sigma_{ltu} $ parametrization extend directly to $ \Sigma_{lt} $.

\section{Sparsity under the $\Sigma_{ltu}$ parametrization}\label{secNotionalGaussian}

\subsection{Causal ordering}

A familiar result interprets zeros in a precision matrix as conditional independencies under a Gaussianity assumption. The less familiar directed graphical models have important differences, both mathematically and conceptually. For instance, many different causal models may be compatible with the same structure of zeros in the precision matrix, and an undirected graph whose associated Gaussian model has a sparse precision matrix could be appreciably less sparse in $\Sigma^{-1}$ when the undirected edges are replaced by directed ones. The key factor determining this is whether there are common response variables occurring later in the causal ordering. 

Whether directed or undirected edges are more natural depends on context. The present section is concerned with directed edges. By postulating a, perhaps notional or provisional, causal ordering among the underlying random variables, substantive understanding can be attached to the interpretation of sparsity on the transformed scale. Through this route we develop insight into the implicit assumptions involved in enforcing sparsity when it is only approximately present, broadening the scope of the work.

\subsection{The matrix logarithm and weighted causal paths}\label{secInterpretLog}

With $[p]=\{1,\ldots,p\}$, let $a\subset [p]$ and $b=[p]\backslash a$ be disjoint subsets of variable indices. As a consequence of a block-diagonalization identity for symmetric matrices \citep{CoxWermuth1993, WermuthCox2004},
\begin{equation}\label{block}
\Sigma=
\begin{pmatrix}
\Sigma_{aa} & \Sigma_{ab} \\
\Sigma_{ba} & \Sigma_{bb}
\end{pmatrix} = \begin{pmatrix}
I_{|a|} & 0 \\
\Sigma_{ba}\Sigma_{aa}^{-1} & I_{|b|}
\end{pmatrix}\begin{pmatrix}
\Sigma_{aa} & 0 \\
0 & \Sigma_{bb.a} 
\end{pmatrix} \begin{pmatrix}
I_{|a|} & \Sigma_{aa}^{-1}\Sigma_{ab} \\
0 & I_{|b|}
\end{pmatrix}.
\end{equation}
The components $\Pi_{b|a}:=\Sigma_{ba}\Sigma_{aa}^{-1}\in\mathbb{R}^{|b|\times |a|}$, $\Sigma_{aa}\in\pd(|a|)$ and $\Sigma_{bb.a}:=\Sigma_{bb}-\Sigma_{ba}\Sigma_{aa}^{-1}\Sigma_{ab}\in \pd(|b|)$ are the so-called partial Iwasawa coordinates for $\pd(p)$ based on a two-component partition $|a|+|b|=p$ of $p$. For a statistical interpretation, let $Y=(Y_{a}^\T, Y_b^\T)^\T$ be a mean-zero random vector with covariance matrix $\Sigma$. Then $\Pi_{b|a}$ is the matrix of regression coefficients of $Y_a$ in a linear regression of $Y_b$ on $Y_a$, and $\Sigma_{bb.a}$ is the error covariance matrix, i.e.~$Y_{b}=\Pi_{b|a}Y_{a} + \varepsilon_{b}$ and $\Sigma_{bb.a}=\text{var}(\varepsilon_b)$. Applying the block-diagonalization identity recursively results in a block-diagonalization in $ 1 \times 1 $ blocks, which corresponds to the LDL decomposition of $ \Sigma $ inherent to the $\Sigma_{ltu}$ parametrization. Specifically, $\Sigma=U \Psi U^\T$ with $\Psi\in\Dp(p)$ and 
\begin{equation}\label{eqUB}
U=I_p + \Theta = (I_p - B)^{-1},
\end{equation}
where $\Theta\in \lts(p)$ is characterized in Proposition \ref{th:u_sparsity_interpretation} and $ B_{ij} $ is the regression coefficient of $ Y_{j} $ in a regression of $ Y_{i} $ on its predecessors $ Y_{1}, \ldots, Y_{i-1} $. Although in principle an arbitrary ordering can be chosen, it is natural to use a postulated causal ordering, if one is available. In the corresponding representation of $ Y $ as a directed acyclic graph with nodes $ Y_{1}, \ldots, Y_p $, a directed edge $ Y_j \rightarrow Y_i $ can exist only if $ j < i $, in which case the total effect of $Y_j$ on $Y_i$ can be expressed in terms of the regression coefficients. 

A simple example provides intuition prior to a formal statement in Proposition \ref{th:u_sparsity_interpretation}.

\begin{example}\label{exThree}
	Consider a set of three variables $ (Y_{1}, Y_{2}, Y_{3} ) $. The total effect of $ Y_{1} $ on $ Y_{3} $ is related to the conditional effects through Cochran's recursion \citep{CochranOmission}, also known as the trek rule, 
	\begin{align}
	\beta_{3.1} = \beta_{3.12} + \beta_{3.21} \beta_{2.1}.
	\label{eq:recursion_coch}
	\end{align}
	The coefficient $\beta_{3.1}$ is the regression coefficient of $Y_1$ in a regression of $Y_3$ on $Y_1$ only, having marginalized over $Y_2$, while  $\beta_{3.12}$ is the coefficient of $Y_1$ in a regression of $Y_3$ on $Y_1$ and $Y_2$. To make this concrete at the population level, $\beta_{3.1}$ is the total derivative of 
	\[
	f(y_1,\bar{y}_2):=\mathbb{E}(Y_3 \mid Y_1=y_1, Y_2=\bar{y}_2),
	\] 
	treating $y_1$ and $\bar{y}_2=\bar{y}_2(y_1)=\mathbb{E}(Y_2 \mid Y_{1}=y_1)$ as free variables, i.e.
	\[
	\beta_{3.1} = \frac{D f(y_1,\bar{y}_2)}{D y_1} = \frac{\partial f(y_1,\bar{y}_2)}{\partial y_1} + \frac{\partial f(y_1,\bar{y}_2)}{\partial \bar{y}_2}\frac{d \bar{y}_2(y_1)}{d y_1}. 
	\]
	The right-hand side of \eqref{eq:recursion_coch} corresponds to tracing the effects of $ Y_{3} $ on $ Y_{1} $ along two paths connecting the nodes in a recursive system of random variables $ (Y_{1}, Y_{2}, Y_{3}) $, with edge weights given by the corresponding regression coefficients, as depicted in Figure \ref{figDAG}. 
	
	\begin{figure}
		\begin{center}
			\begin{tikzcd}[column sep=1.5cm]
			|[draw, circle]| 1
			\arrow[r,"\beta_{2.1}", swap]
			\arrow[rr,bend left=45,"\beta_{3.12}"]
			& |[draw, circle]| 2
			\arrow[r,"\beta_{3.21}", swap]
			& |[draw, circle]| 3
			\end{tikzcd}
		\end{center}
		\caption{Directed acyclic graph with edge weights corresponding to regression coefficients.\label{figDAG}}
	\end{figure}

	A directed edge in a recursive system can exist from node $ j $ to node $ i $ only if $ j < i $. Thus, there are two possible paths from $ Y_{1} $ to $ Y_{3} $, namely $ Y_{1} \rightarrow Y_{3} $ and $ Y_{1} \rightarrow Y_{2} \rightarrow Y_{3} $, which correspond to the first and second term in \eqref{eq:recursion_coch} respectively. Let $ \upsilon_{ij}(l)  $ denote the effect of $ Y_j $ on $ Y_{i} $ along all paths of length $ l $, specified for this three-dimensional system as
	\begin{align*}
	\upsilon_{21}(1) &= \beta_{2.1}, \quad \quad \quad
	\upsilon_{31}(1) = \beta_{3.12}, \\
	\upsilon_{32}(1) &= \beta_{3.21}, \quad \quad \hspace{0.2cm}
	\upsilon_{31}(2) = \beta_{3.21} \beta_{2.1}.
	\end{align*}
	The lower-triangular matrices $ U $ and $ L = \log(U) $ have the form,
	\begin{align}  U = 
	\begin{pmatrix}
	1 & 0 & 0\\
	\beta_{2.1} & 1 & 0 \\
	\beta_{3.12} + \beta_{3.21} \beta_{2.1} & \beta_{3.21} & 1
	\end{pmatrix}, 
	\quad
	L = 
	\begin{pmatrix}
	0 & 0 & 0\\
	\beta_{2.1} & 0 & 0 \\
	\beta_{3.12} + \frac{\beta_{3.21} \beta_{2.1}}{2} & \beta_{3.21} & 0
	\end{pmatrix}. 
	\label{eq:dag_3vars}
	\end{align}
\end{example}

More generally, the following proposition establishes the interpretation of entries of the lower-triangular matrix $ U $ and its matrix logarithm, $ L $.
\begin{proposition}
	\label{th:u_sparsity_interpretation}
	Consider a parametrization $ \Sigma_{ltu}(\alpha, d) = e^{L(\alpha)} e^{D(d)} (e^{L(\alpha)})^{\T} $ and let $ U = \exp(L(\alpha)) $. Let $ \beta_{i.j[k]} $ denote the regression coefficient on $ Y_j $ in a regression of $ Y_i $ on $ Y_j $ and $Y_{1},\ldots, Y_k $. The $(i, j)$th elements of $ U $ and $ L $ have the form,
	\begin{align*}
	U_{ij} = \begin{cases}0 \quad \text{if}\quad i < j, \\
	1 \quad \text{if}\quad i = j, \\
	\sum_{l=1}^{p-1} \upsilon_{ij}(l)  \quad \text{if} \quad i > j,
	\end{cases} 
	\quad
	L_{ij} &= \begin{cases}
	0 \quad \text{if} \quad i \leq j 
	\\
	\sum_{l=1}^{p-1} \frac{\upsilon_{ij}(l)}{l}  \quad \text{if} \quad i > j,
	\end{cases}
	\end{align*}
	where
	\begin{align*}
	\upsilon_{ij}(l) = 
	\begin{cases}
	\sum_{k = j+l-1}^{i-1} \beta_{i.k [i-1]} \upsilon_{kj }(l-1) \quad \text{if} \quad i - j \geq l, \\
	0 \quad \text{otherwise}.
	\end{cases}
	\end{align*}
\end{proposition}

The entry $U_{ij}$ for $ j < i $ thus corresponds to the sum of effects of $ Y_{j} $ on $ Y_{i} $ along all paths connecting the two nodes, with edge weights given by regression coefficients. In contrast, $L_{ij}$, and by extension, the corresponding coefficient in the basis expansion of $ L $, is equal to the weighted sum of effects of $ Y_{j} $ on $ Y_{i} $ along all paths connecting the two nodes, with weights inversely proportional to the length of the corresponding path.

Proposition \ref{th:u_sparsity_interpretation} provides insight into the effect of logarithmic transformation relative to the identity transformation and the inverse transformation, whose resulting zeros encapsulate conditional independencies in a Gaussian model. Specifically, the entries of $B$ can be viewed as representing paths of length 1, corresponding to a complete discounting of longer paths, while $U=(I-B)^{-1}$ has entries aggregating the contributions along all paths, with weights equal to one, i.e.~no discounting of longer paths. In between these two extremes, logarithmic transformation weights a path of length $ l $ by a factor $ 1/l $, as reflected in Proposition \ref{th:u_sparsity_interpretation}. Moreover, the weights in the logarithmic transformation are such that the off-diagonal entries of $\log(I-B)$ and $\log((I-B)^{-1})$ are equal in absolute value, since $\log((I-B)^{-1}) = -\log(I - B)$. Section \ref{secApprox0LTU} discusses some of the implications of these distinctions, following a discussion of exact zeros in \S \ref{secExact0LTU}. 

\subsection{Exact zeros}\label{secExact0LTU}

The previous discussion makes clear that there can be configurations of zeros in $\alpha$ that do not produce whole rows or columns of zeros in $L$. In that case, no structural insights are available from Corollary \ref{corollLDU}, although an interpretation is still available for any exact zero of $L$ via Proposition \ref{th:u_sparsity_interpretation}. Corollary \ref{cor:ltu_sparsity_generalised} provides the relationship between exact zeros in $U$ and $L$ under a causal ordering.

\begin{corollary}
	\label{cor:ltu_sparsity_generalised}
	If no directed path exists from node $ j $ to node  $ i $, $ j < i $, then $B_{ij}= U_{ij} = L_{ij} = 0 $. If $ U_{ij} = 0 $ or $ L_{ij} = 0 $ for $ j < i$, then either effects of $ Y_{j} $ on $ Y_{i} $ along different paths cancel, in which case $ B_{ij} $ need not be zero, or there exists no directed path from $ j $ to $ i $, in which case $ B_{ij} $ is zero.
\end{corollary}

Under an assumption of no path cancellations, Corollary \ref{cor:ltu_sparsity_generalised} generalizes Corollary \ref{corollLDU} to situations in which the configuration of zeros in $\alpha$ does not produce a zero row or column of $L$.  To see this, note that a zero $ j$th row of $ \Theta $ implies that there are no directed paths between nodes $ Y_{1}, \ldots, Y_{j-1} $ and $ Y_{j} $, while a zero $ i$th column implies that there are no directed paths between node $ Y_{i}$ and nodes $ Y_{i+1}, \ldots, Y_{p} $. An example of a graph whose sparsity structure is described by Corollary \ref{cor:ltu_sparsity_generalised} but not by Corollary \ref{corollLDU} is depicted in Figure \ref{figDAG2}.  

\begin{figure}
	\begin{center}
		\begin{tikzcd}[column sep=1.5cm]
		|[draw, circle]| 1
		\arrow[r,"\beta_{2.1}", swap]
		\arrow[rr,bend left=45,"\beta_{3.12}", swap]
		\arrow[rrr,bend left=45,"\beta_{4.123}"]
		& |[draw, circle]| 2
		\arrow[r,"\beta_{3.21}", swap]
		& |[draw, circle]| 3
		& |[draw, circle]| 4
		\end{tikzcd}
	\end{center}
	\caption{Directed acyclic graph corresponding to $U$ and $L$ satisfying $ U_{42} = 0 $ and $ L_{42} = 0 $. \label{figDAG2}}
\end{figure}

Unlike the sparsity structures identified in the more general Corollary \ref{cor:ltu_sparsity_generalised}, the sparsity patterns described in Corollary \ref{corollLDU} can be interpreted in terms of conditional independencies under an additional assumption of Gaussianity.

\begin{proposition}
	\label{lemma:U_zeros}
	Consider a Gaussian random vector $Y=(Y_1,\ldots,Y_p)^\T$ with zero mean and covariance matrix $\Sigma = U \Psi U^{T} $, where $ U = I + \Theta $, $ \Theta \in \lts(p) $ and $ \Psi \in \Dp $. Then, 
	\begin{enumerate}
		\item If the $ j$th column of $ \Theta $ is zero, then $ Y_{j} \indep Y_{j+1}, \ldots, Y_{p} | Y_{1}, \ldots, Y_{j-1} $. Consequently, $ \Sigma^{-1}_{ji} = 0 $ for $ i \in \{ j+1, \ldots, p \} $.
		\item If the $ j$th row of $ \Theta $ is zero, then $ Y_{j} \indep Y_{1}, \ldots, Y_{j-1}$ and $\Sigma_{ij} = 0$ for $ i \in \{ 1, \ldots, j-1 \}. $ 
	\end{enumerate}  
\end{proposition}
For the same number of edges in a graph, the number of zeros in the Gaussian precision matrix depends on the directions of the arrows relative to the configuration of arrows; no reordering of variables can produce a sparser representation. This arises because conditioning in the interpretation of precision matrices is on all variables, rather than only those that occur earlier in the ordering. Given a pair of variables $ i$ and $ j $, marginalization over a third variable induces an edge between $ i$ and $ j $ if the marginalized variable is a transition node or a source node.  By contrast, conditioning is edge-inducing if the conditioning variable is a sink node. In a diagrammatic representation due to \citet{CW1996}, with $\sslash \hspace{-2.1mm}\circ$ and \raisebox{-0.4mm}{$\square$}\hspace{-2.3mm}$\circ\,$ representing marginalization and conditioning respectively,
\begin{eqnarray*}
	i \longleftarrow \sslash \hspace{-2.1mm}\circ \longrightarrow j,  \quad & i\longleftarrow \sslash \hspace{-2.1mm}\circ \longleftarrow j, \quad  & i \longrightarrow \text{\raisebox{-0.4mm}{$\square$}} \hspace{-2.3mm}\circ \longleftarrow j \\
	i\; \text{-}\text{-}\text{-}\text{-} \; j, \quad \; \;  \quad & \quad \; i \longleftarrow j, \quad \; \;  \quad	&	\quad \;\; i\; \text{-}\text{-}\text{-}\text{-} \; j. 
\end{eqnarray*}
where $\text{-}\text{-}\text{-}\text{-}$ indicates that no direction in the induced edge is implied. 

These marginalization and conditioning identities also imply that $U$ will typically be sparser in the correct causal ordering than in an erroneous ordering, modulo coincidental path cancellations, as by definition, any source or transition nodes can only be present among the conditioning sets, and there can be no conditioning on sink nodes. 

Figure \ref{fig:graphs_indiv} depicts two examples of directed graphs that are compatible with the covariance matrices from Figure \ref{fig:sigma_example} (c) and (d) respectively. The indices of non-zero off-diagonal entries of the corresponding lower-triangular matrices are $\{(3,1), (3,2), (5,1), (5,2), (5,3), (5,4) \} $ and $\{(2,1), (3,1), (4,1), (5,1), (4,3), (5,3) \} $. These are non-zero entries of both $ B $ and $ \Theta $, as a convergent Taylor representation of the matrix inverse in \eqref{eqUB} shows that if $ \Theta $ has zero rows or columns, the corresponding row/columns of $ B $ are also zero.

Interpretation of the precision matrix is more appropriate for undirected graphs. For instance, if the edges in Figure 5 (a) were undirected, it would hold that $4 \indep \{1,2,3\}\mid 5$ and $\Sigma_{4j}^{-1}$ would be zero for $j=1,2,3$. That variable 5 is a sink node, however, invalidates this result, as conditioning on the common sink node induces an edge between variable 4 and all other variables. 

\begin{figure}
	\captionsetup[subfigure]{justification=centering}
	% \hfill
	% \hfill
	\hspace{0.5cm}
	\centering
	\begin{subfigure}[h]{0.4\textwidth}
		\centering
		\begin{overpic}[trim={0cm 20cm 10cm 0cm},clip,width = 1\textwidth]{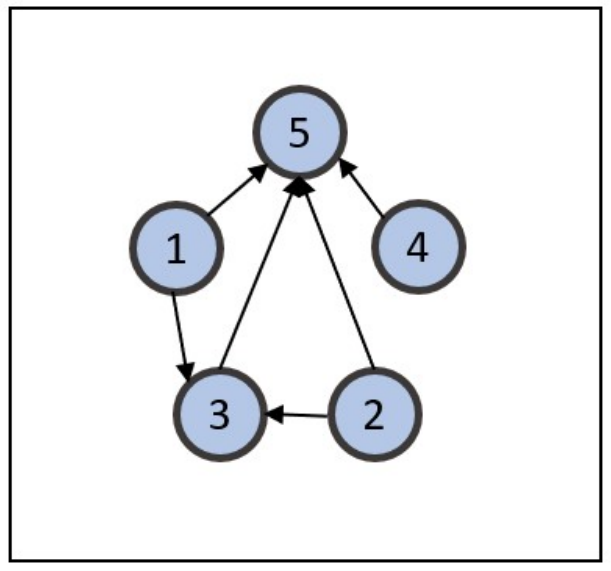}
			% \put(20.5,56.5){\textbf{1}}
			% \put(32.7,28.5){\textbf{2}}
			% \put(58,28.5){\textbf{3}}
			% \put(65,56.5){\textbf{4}}
			% \put(46,75.4){\textbf{5}}
		\end{overpic}  
		\caption{$ \Sigma_{ltu}^{r}(\alpha) $ }
	\end{subfigure}
	\centering
	\hspace{0.2cm}
	\hspace{0.5cm}
	\begin{subfigure}[h]{0.4\textwidth}
		\centering
		\begin{overpic}[trim={0cm 20cm 10cm 0cm},clip,width = 1\textwidth]{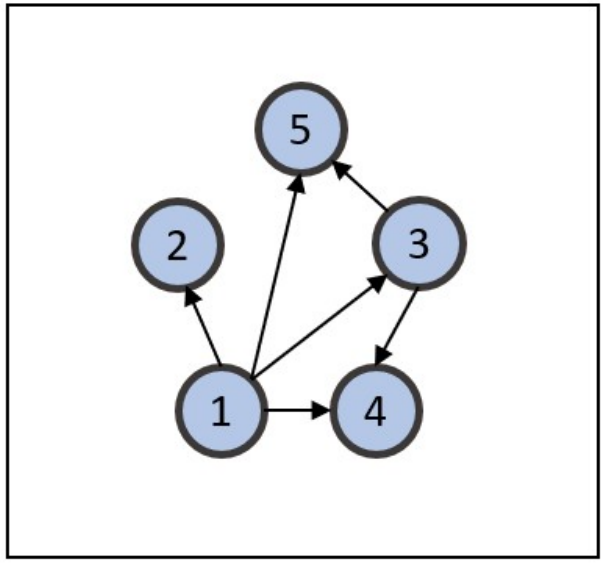}
			% \put(25.5,48.5){\textbf{2}}
			% \put(32.5,24.0){\textbf{1}}
			% \put(59,28.5){\textbf{3}}
			% \put(66.3,56){\textbf{4}}
			% \put(47,74.5){\textbf{5}}
		\end{overpic}
		\caption{$ \Sigma_{ltu}^{c}(\alpha) $}
	\end{subfigure} 
	\caption{Directed acyclic graphs corresponding to the example of Figure \ref{fig:sigma_example}. Arrows indicate directed edges and nodes correspond to random variables. 	\label{fig:graphs_indiv}}
\end{figure}

\subsection{Approximate zeros}\label{secApprox0LTU}

By Proposition \ref{th:u_sparsity_interpretation}, the element $ (i, j) $ of $ L = \log(U) $, and by extension, the corresponding coefficient in the basis expansion of $ L $, is equal to the weighted sum of effects of $ Y_{j} $ on $ Y_{i} $ along all paths connecting the two nodes, with weights inversely proportional to the length of a given path. As a result, in the absence of cancellations of effects along paths of different lengths, which produces an exact zero, a logarithmic transformation reduces the contribution of long paths relative to short paths in the absolute entries of the matrix. This leads to an interpretation of a near-zero as meaning that short paths between variables are associated with small conditional effects, while any large conditional effects are mediated by a string of intermediate variables, where conditioning is on all variables that occur earlier in the causal ordering. The approximation inherent to any statistical algorithm that sets small values of $\alpha$ to zero is thus as follows: the relation between nodes $i$ and $j<i$ would be declared null if relatively direct regression effects were negligible and other effects manifested through long paths.

All of $B$, $U$ and $L$ contain the same information in different guises, that in $B$ being the most easily interpretable. Once sparsity is sought, however, the sparse approximations to $B$, $U$ and $L$ place emphasis on different aspects. 

Since components of $B$ are the direct effects, thresholding on this scale (i.e. setting absolute entries below a given threshold to zero) implicitly assumes that the direct effects are the most important to recover. Consider three variables with connections $1 \rightarrow 2 \rightarrow 3$ and no direct edge between $1$ and $3$. Suppose that edge weight $1 \rightarrow 2$ is very large, while that of  $2 \rightarrow 3$ is small and hence thresholded to zero. The effect of 1 on 3 is not reflected in the resulting thresholded approximation to $ B $, even though this effect may be appreciable in view of the large $1 \rightarrow 2$ effect.

At the other extreme, the entries of $U$ represent the sum of effects along all paths, in which the information about short paths is absorbed in a composite. The cost, potentially, is a small number of near-zeros, and recovery of distant effects, as thresholding implicitly assumes that paths of all lengths are equally important.  Consider a simplistic example in three variables to illustrate a particular point, ignoring other aspects. These variables are ``parents smoke'', ``individual smokes'', ``individual has lung cancer''. Since longer paths are not discounted, it may superficially appear that ``parents smoke'' has a larger effect on ``individual has lung cancer'' than ``individual smokes'', as it has a positive direct effect as well as a positive indirect effect via the intermediate variable ``individual smokes''. Superficial interpretation of $U$ in this context may lead to an erroneous conclusion that cancer is largely related to inherited characteristics. 

Thresholding on the scale of $L$ is a compromise between these two extremes. In the first example we can still recover, after sparse approximation, the effect $1 \rightarrow 3$ that would be lost to thresholding on the original scale, while in the second example, we can still identify ``individual smokes'' as the main cause of cancer.  

Interestingly, the above conclusion bears some resemblance to a suggestion from influential work in computer science about learning in networks. \citet{GroverLeskovec} define two types of neighbourhoods of a node in a graph: one consisting of direct neighbours, and another that involves traversing long paths in the graph. The authors argue that information from both types should be combined, with a hyperparameter specifying the relative importance of paths. The $\Sigma_{ltu}$ parametrization specifies the analogue of this hyperparameter as a discount rate on long paths, given by the inverse path length.

\section{Sparsity under the $\Sigma_{pd}$ and $\Sigma_{o}$ parametrizations}\label{secPDO}

\subsection{Sparsity-induced structures}\label{secSigPDSigO}

The ordering is immaterial for the $\Sigma_{pd}$ and $\Sigma_o$ parametrizations. The structures elucidated in Corollaries \ref{corollBattey} and \ref{corollRybak} imply the same structure on the scale of the precision matrix, and therefore correspond to an appreciable number of conditional independencies. We do not put forward that such structures are likely to hold exactly. Their purpose is instead to approximate a more complex reality, so as to aid interpretation or limit the accumulation of estimation error in procedures like linear discriminant analysis, where the covariance matrix is a nuisance parameter. With this in mind, \S\ref{secApprox0PD} establishes the interpretation of exact and approximate zeros under the $\Sigma_{pd}$ parametrization, while \S \ref{secChain} explains the relevance of the $\Sigma_{pd}$ parametrization in the context of chain graphs, broadening its scope.

\subsection{Interpretation of zeros under the $\Sigma_{pd}$ parametrization}\label{secApprox0PD}

The interpretation of basis coefficients in the transformation $ \Sigma_{pd} $ is less straightforward than in the case of $ \Sigma_{ltu} $. Under additional assumptions we can recover an interpretation of zero entries in the matrix logarithm of $ \Sigma \in \pd(p) $.

Consider an undirected Gaussian graphical model for $ Y_{1},\ldots, Y_{p} $, with no self-loops, and covariance matrix $ \Sigma $. Write $ V = \Sigma^{-1} $ for the precision matrix. The regression coefficients are entries of the matrix $ \tilde{V} = \text{diag}(V)^{-1} V $, where $ \text{diag}(A) $ denotes a diagonal matrix whose diagonal entries are equal to those of $ A $. The form of $ \tilde{V} $ is analogous to that of normalized graph Laplacian in spectral graph theory \citep[e.g.][]{Luxburg2007}.  Let $ \upsilon_{ij}^{u}(l) $ denote the total effect of a unit change in $ Y_j$ on $ Y_i$ along all paths of length $ l $, the superscript $u$ in $ \upsilon_{ij}^{u}(l)$ distinguishing this quantity for an undirected graph. As compared with $ \upsilon_{ij}(l)$ from \S \ref{secInterpretLog}, in the undirected case, a given node can be connected to all others. The following proposition establishes 
the interpretation of entries of $ \log(\Sigma \text{diag}(V)) = - \log(\tilde{V}) $. This result can be seen as an adaptation of Proposition \ref{th:u_sparsity_interpretation} for undirected graphical models. 
\begin{proposition}
	\label{th:sigma_interpretation}
	For a matrix $ \Sigma \in \pd(p) $, let $ \tilde{V}= \text{diag}(V)^{-1} V  $ and $ \tilde{\Sigma} = \Sigma \text{diag}(V) $. 
	Then the elements $ (i, j) $ of $ \tilde{\Sigma} $ and of its matrix logarithm have the form,
	\[
	\tilde{\Sigma}_{ij}  = \sum_{l=1}^{\infty} \upsilon^{u}_{ij}(l), \quad \quad 
	\log(\tilde{\Sigma} )_{ij} =  -\log(\tilde{V} )_{ij} = 
	\sum_{l=1}^{\infty} \frac{\upsilon^{u}_{ij}(l)}{l}  
	\]
	if the infinite sums converge.
\end{proposition}

Unlike in Proposition \ref{th:u_sparsity_interpretation}, expressions for elements of transformed matrices in Proposition \ref{th:sigma_interpretation} involve infinite sums. As a result, Proposition \ref{th:sigma_interpretation} requires an additional assumption of convergence.
The interpretation established in Proposition \ref{th:sigma_interpretation} holds for the matrix logarithm of a suitably scaled covariance matrix, rather than for $ \log(\Sigma) $. However, a direct calculation yields the following result, which is also implicit in Corollary \ref{corollBattey}.
\begin{corollary}\label{corollZerosL}
	Assume that no cancellation of effects of $ Y_j $ on $ Y_i $ occurs along different paths. Then $ L_{ij} = 0 $ if and only if $ \tilde{V}_{ij} = 0 $. 
\end{corollary}
Thus, in the absence of cancellation of effects, an element $ (i,j) $ of $ L=\log(\Sigma)$, as well as the corresponding basis coefficient, is zero if and only if there is no path between nodes $ i $ and $ j $. It is often the case, in spite of the previous sentence, that $L$ is sparser than $\Sigma$ and $\Sigma^{-1}$  under more general notions of sparsity. This is probed by simulation in \S \ref{secSim}.

\section{Chain graphs and the Iwasawa decomposition}\label{secChain}

\subsection{A unifying parametrization for chain graphs}\label{secChainInterp}

The interpretation of \eqref{block} in terms of partial Iwasawa coordinates points, via a group-theoretic treatment, to an encompassing formulation. This is first developed from a statistical perspective.

Corollary \ref{th:chain_sparsity} to be presented is a unifying result in which the interaction of sparsity on the transformed scale with structure on the original scale is elucidated, recovering three of the results of \S \ref{secRepara} as special cases in which connected components consist of either all undirected edges (Corollary \ref{corollBattey}) or all directed edges (Corollaries \ref{corollLT} and \ref{corollLDU}).

A graph $ G = (V, E) $ is a chain graph if it contains both directed and undirected edges, but no semi-directed cycles \citep[][p.~83]{DrtonEichlerMLE}. When two nodes $ v, w \in V $ are connected by a path consisting solely of undirected edges, we say that $ u $ and $ w $ are equivalent. Let $ \mathcal{U} $ be a set of equivalence classes, called chain components, of this equivalence relation. Define a new graph $ \mathcal{D} = (\mathcal{U}, \mathcal{E}) $ with nodes $ \mathcal{U} $ and edges $ \mathcal{E} $ between chain components. Since we assume that there are no semi-directed cycles in $ G $, the graph of $ \mathcal{D} $ is a directed acyclic graph.  

Chain graphs are usually characterized by the alternative Markov property \citep{Andersson}, satisfied under a mean-zero Gaussianity assumption if and only if the precision matrix can be written as
\begin{equation}
\Sigma^{-1} = (I - B^{\T}) \Omega^{-1} (I - B),
\label{eq:AMPDrton}
\end{equation}
where $ \Omega^{-1}_{uv} $ is zero if an undirected edge $ (u,v) $ is not in the graph, and $ B_{vu} = 0 $ if a directed edge $ u \rightarrow v $ is not in the graph \citep{DrtonEichlerMLE}. Since there exists at most one edge between any pair of nodes, $ \Omega_{uv} \neq 0$ implies $B_{vu}=0 $, and vice-versa. Every directed acyclic graph can be represented by a triangular matrix, so we can always find a permutation matrix $P\in \p(p)$ such that decomposition \eqref{eq:AMPDrton} of $ P \Sigma^{-1} P^{\T} $ yields a strictly lower-triangular matrix $ B $ and a block-diagonal matrix $ \Omega^{-1} $. That there exists a $P\in \p(p)$ that simultaneously rearranges $B$ and $\Omega^{-1}$ follows from the assumption that there is an underlying chain graph.  From now on we assume that an ordering of variables has been chosen such that $ B $ is triangular and $ \Omega^{-1} $ is block-diagonal. The factorization \eqref{eq:AMPDrton} then represents the precision matrix as a product of block-diagonal and block-triangular matrices. The block-triangular matrix captures connections between chain components, while the block-diagonal matrix describes connections within the chain components.

Suppose that a random vector $ Y \sim N(0, \Sigma) $ is partitioned into $ c $ blocks, $ Y = (Y_{1}, \ldots, Y_{c}) $,  where each block $ Y_{i} $ constitutes a chain component, that is, the variables within $ Y_{i} $ form a connected undirected graphical model. Let $ p_{i}$ denote the dimension of the sub-vector $ Y_{i} $. Decomposition of the precision matrix \eqref{eq:AMPDrton} implies a decomposition of $\Sigma$ in terms of a block-diagonal component $\Omega$:
\begin{equation}
\quad \quad \Sigma = T \Omega T^{\T}, \quad \quad \Omega =\Omega_{1} \oplus \Omega_{2} \oplus \cdots \oplus \Omega_{c}, \quad \quad \Omega_{i} \in \pd(p_i),
\label{eq:fact_sigma}
\end{equation}
where $ T = (I-B)^{-1} \in \lt_{s}(p) $ has diagonal blocks $ I_{p_1}, \ldots, I_{p_c} $. Factorization \eqref{eq:fact_sigma} can be obtained by successive block-triangularization of $ \Sigma $, and the corresponding block-triangular transformation can be parametrized as
\begin{align*}
(\alpha, \delta) \mapsto \Sigma_{bt}(\alpha, \delta) := e^{L(\alpha)} e^{D(\delta)} ( e^{L(\alpha)} )^{\T}, 
\end{align*}
where $ D(\delta) = D_{1}(\delta_{1}) \oplus \cdots \oplus D_{c}(\delta_{c}) $, $ D_{i}(\delta_{i}) \in \sym(p_{i}) $, $ \delta_{i} \in \RR^{p_{i} (p_{i}+1)/2} $ and $ \delta = (\delta_{1}^{\T} \: \ldots \: \delta_{c}^{\T} )^{\T} $. The matrix logarithm $ L(\alpha) $ is block-triangular with $ c $ diagonal blocks, each equal to a $ p_i \times p_i $ identity matrix. With $ p_{\delta} $ the dimension of $ \delta $, the dimension of $\alpha$ is $p(p+1)/2 - p_{\delta}$. 

That $\Sigma_{bt}$ represents a unifying structure is seen on noting that when $ D(\delta) $ is diagonal we recover the parameterization $ \Sigma_{ltu}(\alpha, \delta)$, and therefore also $\Sigma_{lt}(\alpha)$ by the discussion surrounding equation \eqref{eqCorrespondence}. At the other extreme, when $ D(\delta) $ consists of a single block of dimension $ p \times p $, we recover $ \Sigma_{pd}(\delta) $. The remaining parametrization $\Sigma_o$ is not directly recoverable as a special case of $\Sigma_{bt}$, although there is an indirect connection because $L \in \sk(p)$ can be decomposed as $L=L_s-L_s^\T$ with $L_s \in \lts(p)$. Additional details are in Appendix \ref{secParamIwa}.

To distinguish the sparsity structure induced on $ T $ and $ \Omega $, we use $ d_{r}^{*} $ and $ d_{c}^{*} $ to represent the number of non-zero rows and columns of $ L $, and $ d^{*} $ to represent the number of unique indices of non-zero rows and columns of $ L $. An analogous quantity for $ D(\delta) $ is written $ d^{*}_{\delta} $.

\begin{corollary}
	\label{th:chain_sparsity}
	The image of the map $(\alpha, \delta) \mapsto \Sigma_{bt}(\alpha, \delta)=e^{L(\alpha)}e^{D(\delta)}(e^{L(\alpha)})^\T$ is logarithmically sparse in the sense that $\|\alpha\|_{0}=s_{\alpha}^*$ and $\|\delta\|_{0}=s_{\delta}^*$ in the basis representation of $ L(\alpha) $ and $ D(\delta) $ respectively if and only if $ \Sigma = T \Omega T^{T} $ where $ \Omega = \Omega_{1} \oplus \Omega_{2} \oplus \cdots \oplus \Omega_{c} $, for some $ c \leq p $, and
	\begin{enumerate}
		\item $($Sparsity of DAG\,$)$: $T=I + A$ and $A \in\lts(p)$ has $p-d_r^*$ zero rows and $p-d_c^*$ zero columns, of which $p-d^*$ coincide after transposition.
		\item $($Sparsity of chain components\,$)$: $\Omega$ is of the form $\Omega=P \Omega^{(0)} P^{\T}$, where $P$ is a permutation matrix and $\Omega^{(0)}=\Omega^{(0)}_1 \oplus D_{p-d^{*}_{\delta}}$, where $D_{p-d^{*}_{\delta}}\in\D(p-d^{*}_{\delta})$ and $\Omega^{(0)}_1\in\pd(d^{*}_{\delta})$ is block-diagonal and of maximal dimension, in the sense that it is not possible to find another permutation matrix such that the dimension of the diagonal block is larger than $p-d^{*}_{\delta}$.
	\end{enumerate}
\end{corollary}

Since $ L(\alpha) = \log(T) = - \log(I - B) $, the coefficients of $ \log(I - B) $ in the appropriate basis are equal to $ - \alpha $. Thus, the sparsity indices of $ I - B $ and $ T $ on the logarithmic scale coincide. In contrast, no obvious relationship exists between the sparsity of $ I-B $ and $ T $, since a zero entry in $ I - B $ does not imply a zero entry in $ T $, and vice versa. A similar point applies to $ \Omega  $ and $ \Omega^{-1} $ since $ D(\delta) = \log(\Omega) = - \log(\Omega^{-1}) $.

A result analogous to Corollary \ref{th:chain_sparsity} can be obtained for the precision matrix on noticing that $ d_{r}^{*} $ and $ d_{c}^{*} $ are equal to the numbers of non-zero columns and rows of $ - L(\alpha) = \log(I-B) $ respectively. Part (1) of Corollary \ref{th:chain_sparsity} thus describes structures arising when the sparsity patterns of $ T $ and $ I - B $ coincide. For example, suppose that the $ i $th column of $ I-B $ is zero, that is, node $ i $ has no descendants. This is reflected on the transformed scale by a zero $i$th row of $ L (\alpha) $.

\subsection{Connection to the Iwasawa decomposition of the general linear group}\label{secIwasawa}

The parametrization $\Sigma_{bt}$ represents the general form of the Iwasawa decomposition of the general linear group, pertaining to the partition $p_1+\cdots+p_c=p$ of $p$. This provides a group-theoretic perspective on how the parametrization $\Sigma_{bt}$ unifies and generalizes three of the four parameterizations from \eqref{eqRepara}, detached from any consideration of causal ordering. Appendix \ref{sec:GL} provides further discussion. The identification $\pd(p) \cong \gl(p)/\ortho(p)$ characterizes a positive definite matrix as a non-singular one whose `orthogonal component' has been discounted. Explicitly, $\Sigma_A=A^\T A$ is positive definite for every $A \in \gl(p)$, and left-multiplication by any orthogonal matrix gives an equivalence class $[A]=\{OA: O\in \ortho(p)\}\subset \gl(p)$ such that $\Sigma_A=V^\T V$ for any $V\in[A]$. Relatedly, \citet{DZ2017} identified the subgroup of $\gl(p)$ acting on $\Sigma$, and studied corresponding equivariant estimators that preserve the chain graph property.

The group theoretic perspective also unlocks a connection between information geometry and graphical modelling when a causal ordering is available. In particular, the information geometry of the zero-mean multivariate Gaussian \citep{Skovgaard1984} induced by the Fisher information metric tensor coincides with the quotient geometry of $\pd(p) \cong \gl(p)/\ortho(p)$ under a Riemannian metric that is invariant to the transitive action of $\gl(p)$. Upon representing a positive definite covariance matrix $\Sigma$ in the partial Iwasawa coordinates $(\Sigma_{aa},\Sigma_{ba}\Sigma_{aa}^{-1},\Sigma_{bb.a})$, the metric is endowed with an interpretation compatible with a suitable graphical model and a corresponding $\Sigma_{ltu}$ parametrization. 

\section{Propagation of error and estimation}\label{secStatImpl}

\subsection{Existing results for $\Sigma_{pd}$}\label{secRecapBattey2019}

A question of practical relevance is whether increased sparsity on the transformed scale translates to an inferential advantage. Intuition in the case of the $\Sigma_{pd}$ parametrization can be obtained through consideration of the simplest estimator exploiting sparsity, namely the thresholding estimator of \citet{BickelLevinaRegEst,BickelCovarThresh} applied on the transformed scale. For the $\Sigma_{pd}$ parametrization, thresholding sets to zero the entries of a pilot estimator $\hat L=\log \hat\Sigma$ that are below a threshold in absolute value. Success of the approach hinges on the elementwise consistency of $\hat L$ for $L$. We are thus interested in how the estimation error $\hat \Sigma - \Sigma$ propagates to the scale of the matrix logarithm. A simplified setting provides insight into the considerations involved. 

Consider initially a small perturbation of $\Sigma$, of the form $\Sigma+ \varepsilon I$ for $\varepsilon>0$, which preserves the eigenvectors. The argument in Appendix \ref{secLogPerturbation} shows that, using a complex-variable representation of the matrix logarithm, the error propagates to the $(j,k)$th entry on the logarithmic scale as
\begin{equation}\label{eqApprox}
[\log(\Sigma + \varepsilon I) - \log(\Sigma)]_{j,k} = \varepsilon\sum_{r,v} \Bigl(\frac{1}{2\pi i}\ointctrclockwise_{\gamma}\frac{\log(z)}{(z-(\lambda_{r}+\varepsilon))(z-\lambda_{v})} dz\Bigr)o_{jr}o_{kv}\sum_{\ell,s}o_{\ell r}o_{sv}.
\end{equation}
Consider the summation over $\ell$ and $s$ in \eqref{eqApprox}. For $r=v$,
\[
\sum_{\ell,s}o_{\ell v}o_{sv} =\sum_s o_{s v} o_{s v} + \sum_{s, \ell \neq s} o_{\ell v} o_{s v} = 1
\]
by the orthonormality identity $O^\T O = OO^\T = I$. For $r\neq v$, the double summation  is approximately zero by the observation that cross-products $o_{\ell r}o_{s v}$ are of order $1/p$ and zero on average for large $p$. Subject to this last approximation,  \eqref{eqApprox} simplifies to
\begin{equation}\label{eqApprox2}
[\log(\Sigma + \varepsilon I) - \log(\Sigma)]_{j,k} = \varepsilon\sum_{v} \Bigl(\frac{1}{2\pi i}\ointctrclockwise_{\gamma}\frac{\log(z)}{(z-(\lambda_{v}+\varepsilon))(z-\lambda_{v})} dz\Bigr)o_{jv}o_{kv},
\end{equation}
where the term in parenthesis is given by the sum of the residues at the two singularities,
\[
\frac{1}{2\pi i}\ointctrclockwise_{\gamma}\frac{\log(z)}{(z-(\lambda_{v}+\varepsilon))(z-\lambda_{v})} dz=\frac{\log(\lambda_v + \varepsilon) - \log(\lambda_v)}{\lambda_{v}+\varepsilon - \lambda_v} = \frac{\log(\lambda_v + \varepsilon) - \log(\lambda_v)}{\varepsilon},
\]
whose first-order Taylor expansion around $\varepsilon =0$ is $\lambda_v^{-1}$, i.e.~the derivative of $\log \lambda_v$. Thus,
\[
[\log(\Sigma + \varepsilon I) - \log(\Sigma)]_{j,k}= \varepsilon \sum_{v} \lambda_v^{-1}o_{jv}o_{kv} + O(\varepsilon^2) = \varepsilon [\Sigma^{-1}]_{j,k} + O(\varepsilon^2),
\]
showing how the perturbation $\varepsilon$ propagates to the scale of the matrix logarithm. 

Realistic pilot estimators of $\Sigma$ entail perturbations of both eigenvectors and eigenvalues, and the previous argument then requires that pilot estimators provide consistent estimates $\hat o_v$ of eigenvectors in the sense that $\hat o_{r}^\T o_v \rightarrow_p 0$ for $r\neq v$ and $\hat o_v^\T o_v \rightarrow_p 1$. A more complete development for specific pilot estimators can be found in \citet{Battey2019}, where results are also presented for the propagation of error in the converse direction under the spectral norm, having exploited sparsity on the scale of the matrix logarithm.

\subsection{New results for $\Sigma_{ltu}$}\label{secNewEstLTU}

A broadly analogous scheme applies to estimation under a sparse $\Sigma_{ltu}$ parameterization. When a causal ordering of variables or blocks of variables is available, a natural pilot estimator for $ U $ regresses each variable on its causal predecessors, called parents,
and a corresponding pilot estimator of $ \Omega $ is the sample covariance matrix of the resulting residuals. In the absence of a causal ordering, pilot estimators for both $ U $ and $ \Omega $ can be obtained through an LDL decomposition of the sample covariance matrix. 
The construction of the estimator and establishment of its properties entails a multitude of notation and technical conditions, collected in Appendix \ref{secEstimation}. Among those technical conditions, the following are less standard and are associated with a pilot estimator that exploits the causal ordering.

\begin{condition}
	There exists $ l^{*} \in \mathbb{N} $, such that for any pair of nodes $ (i,j) $, $ j < i  $, $ \sum_{l = l^{*}+1}^{p} \frac{1}{l} \delta_{i|j}(l) = C (\log p /n)^{\varphi / (2(\varphi+1))} $ for $ \varphi > 0 $, where $ \delta_{i|j}(l) $ denotes the sum of effects of node $ j $ on node $ i $ along all paths of length $ l $.
	\label{ass:decaying_paths:maintext}
\end{condition}

Condition \ref{ass:decaying_paths:maintext} assumes that there is a value $l^{*}$ such that effects along paths of length exceeding $l^*$ are negligible when weighted inversely by the path length. While our proofs require Condition \ref{ass:decaying_paths:maintext}, simulations in Appendix \ref{sec:th_simuls} suggest that this is not necessary.

Proposition \ref{th:sigma_consistent:maintext} establishes spectral-norm consistency of the proposed estimator under conditions detailed in Appendix \ref{secEstimation}. These include Condition \ref{ass:app_2:maintext}, which characterizes the sparsity of $ L $ and $ D $. 

\begin{condition}
	Assume that $ L \in \mathcal{U}(q_{l}, s_{l}(p)) \cap \lts(p) $ and $ D\in \mathcal{U}(q_{\omega}, s_{\omega}(p)) \cap \pd(p) $, where $  q_{l}, q_{\omega} \in [0, 1] $, $ s_{l}(p)/p \rightarrow 0 $, $ s_{\omega}(p)/p \rightarrow 0 $ and
	\begin{equation}\label{eqSparseClass}
	\mathcal{U}(q, s(p)) = \biggl\{ A \in \M(p): \max_{i} \sum_{j=1}^{p} |A_{ij}|^{q} = s(p)  \biggr\}.  
	\end{equation}
	\label{ass:app_2:maintext}
\end{condition}

\begin{proposition}
	\label{th:sigma_consistent:maintext}
	Suppose that the tuning parameters of equations \eqref{eq:thresholdingL} and \eqref{eq:thresholdingO} of Appendix \ref{secEstimation} satisfy $\tau_{l} \asymp (n^{-1} \log p)^{1/2}$ and $ \tau_{\omega} \asymp s_l(p)^{2} (n^{-1} \log p)^{(3/2 - q_l)(1-q_{\omega})} $. Under Conditions %\ref{ass:max_entry},
	\ref{ass:decaying_paths:maintext} and \ref{ass:app_2:maintext} and Conditions \ref{ass:normal}-\ref{ass:min_eigval} of Appendix \ref{sec:theory_estim}, with $ \varphi \geq 1/2 $, the estimator $ \tilde{\Sigma} =  \tilde{T} \tilde{\Omega} \tilde{T}^{\T} $ of  $ \Sigma = T \Omega T^{\T} $ satisfies
	\begin{align*}
	\| \tilde{\Sigma} - \Sigma \|_{2} 
	&= O_{p} \bigl( \max \{r_t, r_{\omega}\} \bigr),
	\end{align*}
	where \begin{align*}
	r_{t} 
	&=  s_{l}(p)^{2} (n^{-1} \log p)^{3/2 - q_l}, 	    \\
	r_{\omega} &=  s_{\omega}(p) s_{l}(p)^{2-2 q_{\omega}} (n^{-1} \log p)^{(3/2-q_{l})(1-q_{\omega})}.
	\end{align*}
\end{proposition}

An important question concerns the implications of misspecification of the causal ordering. Although this would annul the interpretation of \S \ref{secNotionalGaussian}, the role of the causal ordering in Proposition \ref{th:sigma_consistent:maintext} is via the degree of sparsity present, which is reflected in the rates in Proposition \ref{th:sigma_consistent:maintext}. Thus, to the extent that the conditions are still satisfied, Proposition \ref{th:sigma_consistent:maintext} remains valid.

\section{Some numerical insights}\label{secSim}

\subsection{Approximate sparsity in the four logarithmic domains}\label{secApproxSparsity}

The prospect of routinely inducing sparsity through logarithmic transformation under the four maps \eqref{eqRepara} is only realistic under a notion of approximate sparsity that allows for slight departures from zero. Simulations in Appendix \ref{appApproxSparsity} give an indication of how approximate sparsity in the four logarithmic domains associated with $\Sigma_{pd}$, $\Sigma_o$, $\Sigma_{lt}$ and $\Sigma_{ltu}$ transfers to a commensurate notion of approximate sparsity in the inverse domain, this being the most widely used parameter domain in which to perform sparse estimation. Tables \ref{tabPD}--\ref{tabLTU} also compare the performance of thresholding estimators on the different scales, suggesting in all cases except $\Sigma_o$ that exploiting sparsity on the most sparse scale, i.e.~the logarithmic scale under the relevant parametrization, transfers substantial benefits to estimation of $\Sigma^{-1}$. In the case of $\Sigma_o$, the simple thresholding approach of Appendix \ref{appApproxSparsity} appears too simplistic, presumably owing to the constraints on $\alpha$ needed to make the parametrization injective.

\subsection{Exploration of sparsity regimes}\label{secSimNew}

In \S \ref{secApproxSparsity}, the matrices on the transformed scale were sparse by construction. We now investigate whether the logarithmic transformation can be useful in less idealized situations. 

Consider a Gaussian directed acyclic graph with covariance matrix $ \Sigma = (I - B)^{-1} D^{-1} (I - B^\T)^{-1} $. When $ B $ contains many zeros or near-zeros, $ \Sigma $ is also likely to be sparse, and the estimator $ \hat{\Sigma}_{\tau} $ of \cite{BickelCovarThresh} would be a natural choice. Sparsity of $\Sigma$ typically decreases both as the number of non-zero elements of $ B $ increases, and as the magnitude of non-zero entries (the weights of directed edges) becomes large. This follows from Proposition \ref{th:u_sparsity_interpretation}, whereby the entry $ (i, j) $ of the matrix $ (I - B)^{-1} $ corresponds to the sum of effects of node $ i $ on node $ j $ along all paths connecting the two nodes. As the number, or magnitude, of non-zero edge weights increases,  cumulative effects inevitably increase. Although a similar phenomenon is expected for $ L = -\log(I-B) $, Proposition \ref{th:u_sparsity_interpretation} suggests that the sparsity of $ L $ should decrease more slowly, due to the discounting of longer paths. Eventually, as the number of entries of $ B $ below a threshold increases, or as the absolute value of these entries increase, we expect a strong accumulation of effects on variables ordered last, as these will have the largest number of incoming paths. Since the number of possible paths increases exponentially with the number of nodes, the accumulation of effects can cause entries of $ \Sigma $ to be unbounded. 

There are thus three regimes. When the edge matrix $ B $ is sparse or its non-zero entries have small absolute values, thresholding in the original domain will typically yield better results. As the number of non-zero entries of $ B $ increases, or the edge weights increase, thresholding in the logarithmic domain is expected to be advantageous. With a further decrease in the sparsity of $ B $ or increase in the edge weights, there is no approximate sparsity in either domain.

To verify this empirically, we compared the performance of thresholding on the original and logarithmic scales under the $\Sigma_{ltu}$ parametrization, for different values of edge weights and different levels of sparsity for $ B $. Specifically, we took $ D $ as the identity matrix and generated an edge matrix $ B $ by randomly selecting a prespecified percentage of its entries, and assigning a fixed value $ \epsilon >0$ to those entries. Positivity of $\epsilon$ avoids cancellations of effects along different paths. For each covariance matrix, we generated a sample of size $n =150$ from the corresponding multivariate normal distribution and constructed thresholding estimates on the scales of interest, following the recommendation of \cite{BickelCovarThresh} for selecting the threshold $\tau$. Specifically, a sample covariance matrix was estimated on two disjoint subsets of the data, of size $n/3$ and $2n/3$. The estimate $\tilde{\Sigma}$ based on the larger sample was treated for the purpose of tuning as the target covariance matrix. Thresholding was applied to the matrix estimated on the smaller sample, yielding a sparse estimate $ \bar{\Sigma}_{\tau} $. The threshold was then chosen to minimize the relative $ \ell_2$-norm error, $ \| \tilde{\Sigma} - \bar{\Sigma}_{\tau} \|_{2} / \| \tilde{\Sigma} \|_{2} $ across $ 5 $ random splits. The final estimate $ \hat{\Sigma}_{\tau} $ was based on the selected threshold and the full sample. The same procedure, with the obvious modifications, was used to select the threshold used for sparse estimation on the logarithmic scale under the $\Sigma_{ltu}$ parametrization, resulting in estimates $\hat{U}_{\tau}$ of $U$ and $ \hat{U}_{\tau} \hat{D} \hat{U}_{\tau}^{\T} $ of $\Sigma$.

\begin{figure}%[h!]
	\vspace{-0.2cm}
	\centering
	\begin{subfigure}[b]{0.48\textwidth}
		\centering
		\includegraphics[trim=0.0in 0.0in 0.0in 0.0in, clip, height=0.24\paperwidth]{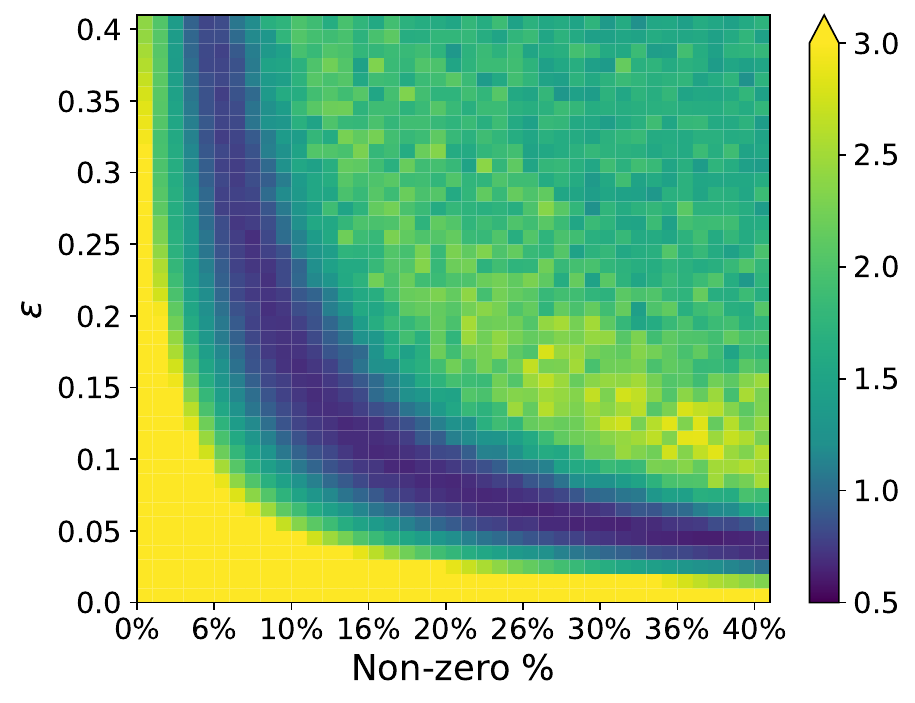}
		\caption{$ \| \hat{U}_{\tau} \hat{D} \hat{U}_{\tau}^{\T} - \Sigma \|_{2} / \| \hat{\Sigma}_{\tau} - \Sigma \|_{2} $}
	\end{subfigure}
	\hspace{-0.8cm}
	\begin{subfigure}[b]{0.49\textwidth}
		\centering
		\includegraphics[trim=0.0in 0.15in 0.0in 0.0in, clip, height=0.24\paperwidth]{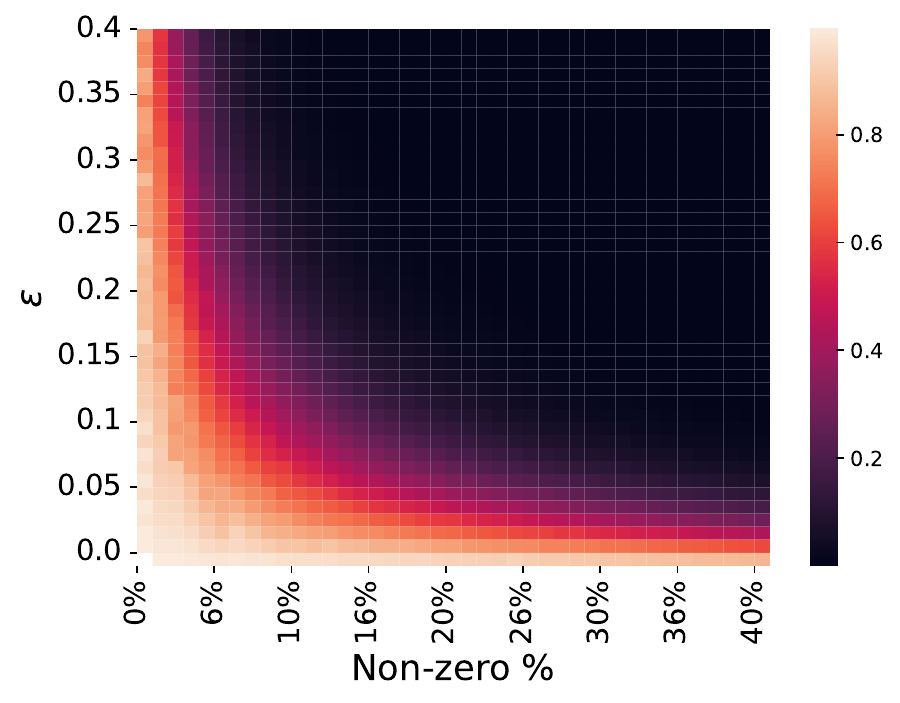}
		\caption{$r(L) / r(\Sigma) $}
	\end{subfigure} 
	\\[-0ex]
	\vspace{0.5cm}
	\par
	\centering
	\begin{subfigure}[b]{0.49\textwidth}
		\centering
		\includegraphics[trim=0.0in 0.0in 0.0in 0.0in, clip, height=0.24\paperwidth]{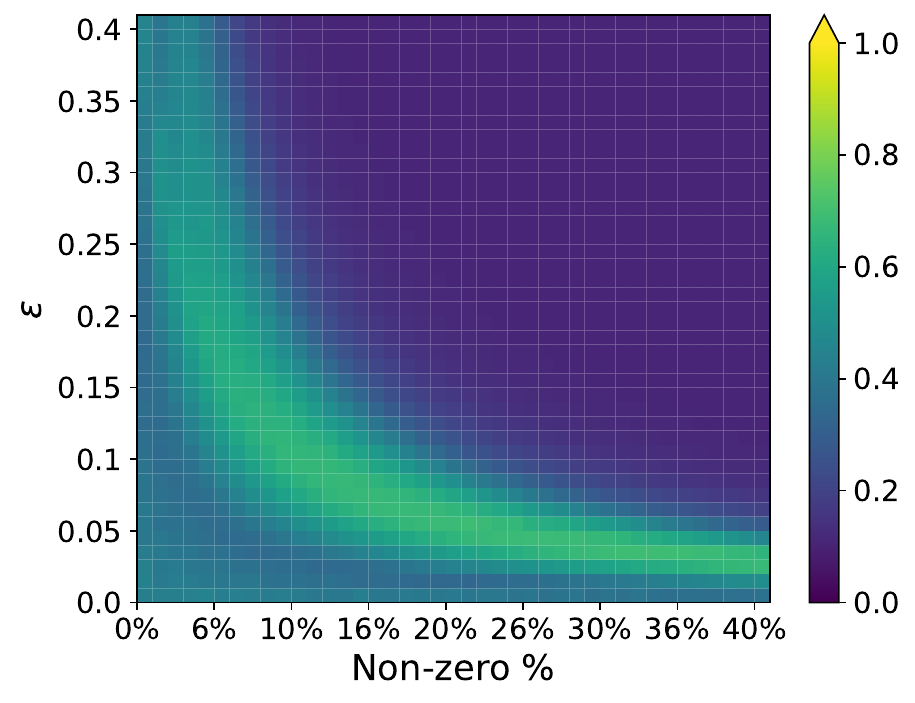}
		\caption{$  \| \hat{\Sigma}_{\tau} - \Sigma \|_{2} / \| \Sigma \|_{2} $ }
	\end{subfigure}
	\hspace{-0.8cm}
	\begin{subfigure}[b]{0.49\textwidth}
		\centering
		\includegraphics[trim=0.0in 0.0in 0.0in 0.0in, clip, height=0.24\paperwidth]{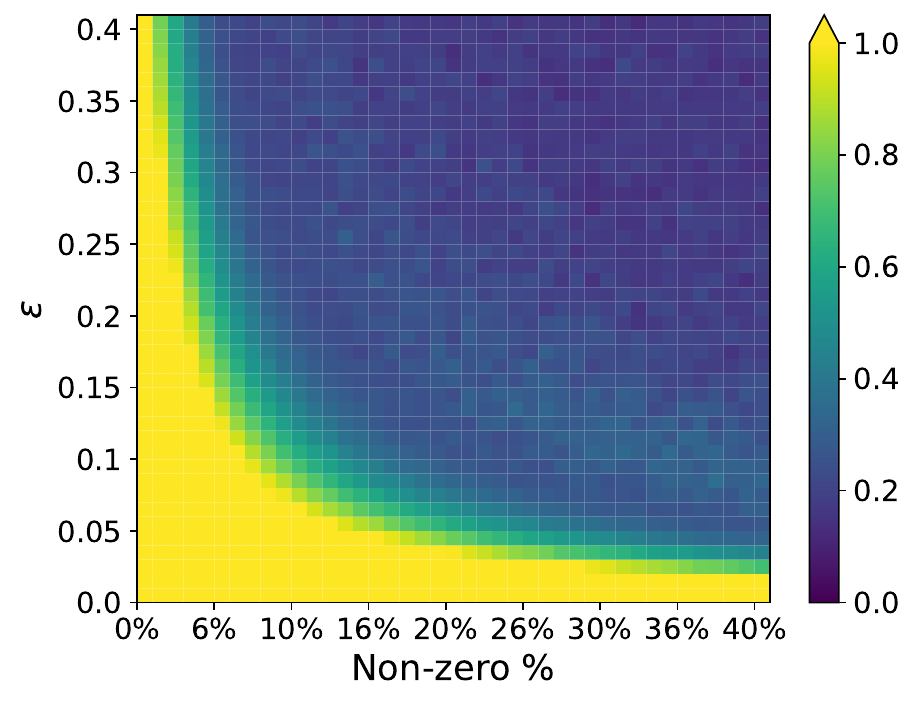}
		\caption{$ \| \hat{U}_{\tau} \hat{D} \hat{U}_{\tau}^{\T} - \Sigma \|_{2} / \| \Sigma \|_{2} $}
	\end{subfigure}  
	\caption{\label{fig:ldl_performance} Relative $\ell_2$ errors (a, c and d) and relative $\ell_1$ row norms (b) for different combinations of $ \epsilon $ ($y$-axis) and levels of sparsity of $ B $, measured by the percentage of non-zero entries ($x$-axis). Pixels are median values over $ 100 $ simulations with $ n = 150 $, $ p = 100 $ for different combinations of $ \epsilon $ ($y$-axis) and levels of sparsity of $ B $, measured by the percentage of non-zero entries ($x$-axis).
	}
\end{figure}

Results in Figure \ref{fig:ldl_performance} (a) show that thesholding on the original scale outperforms thresholding on the logarithmic scale under the $\Sigma_{ltu}$ parametrization for high levels of sparsity of $ B $ and small values of $ \epsilon $, while the opposite is true for medium levels of sparsity and $ \epsilon $. For large values, the covariance matrix is highly non-sparse and neither sparsity scale is suitable. The standalone performance of  $ \hat{\Sigma}_{\tau} $ and $ \hat{U}_{\tau} \hat{D} \hat{U}_{\tau}^{\T} $  is shown in Figures \ref{fig:ldl_performance} (c) and (d). Results indicate that when $ \hat{U}_{\tau} \hat{D} \hat{U}_{\tau}^{\T} $ outperforms $ \hat{\Sigma}_{\tau} $ it is due both to the poorer performance of the latter estimate and improved performance of the former. The performance of $ \hat{U}_{\tau} \hat{D} \hat{U}_{\tau}^{\T} $ exhibits a sharp transition as the sparsity of $B $ decreases and $ \epsilon $ increases. This may suggest that the logarithmic transformation is detrimentally distorting for very sparse covariance matrices. Since %Condition \ref{ass:app_2:maintext} and the sparsity assumption of \cite{BickelCovarThresh} 
the sparsity conditions for thresholding are closely related to row-wise norms, we show in Figure \ref{fig:ldl_performance} (b), for comparison to (a), the ratio of maximum $ \ell_{1} $ row norm of the two matrices, defined for $ A \in \M(p) $ as
\begin{equation}\label{eqr}
r(A) = \max_{i \in \{2, \hdots, p \}} \sum_{j = 1}^{i-1} | A_{ij} |.
\end{equation}
In order to make the metric comparable for lower-triangular and symmetric matrices, $ r(A) $ only considers the entries of $ A $ below the diagonal. The contours of equal $ r(L) / r(\Sigma) $ in (b) closely resemble those of the relative errors in (a). We probe this relationship further in Figure \ref{fig:sparsity_performance} for the same simulation setting. 

\begin{figure}[h!]
	\vspace{-0.2cm}
	\centering
	\begin{subfigure}[b]{0.45\textwidth}
		\centering
		\includegraphics[trim=0.0in 0.0in 0.0in 0.0in, clip, height=0.24\paperwidth]{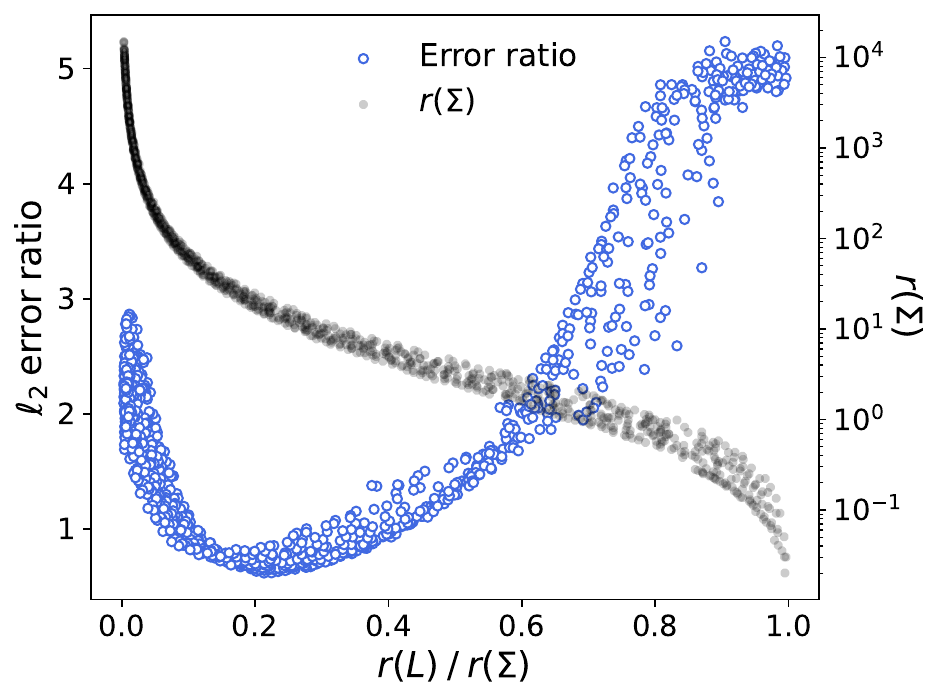}
		\caption{}
	\end{subfigure}
	\hspace{0.2cm}
	\begin{subfigure}[b]{0.4\textwidth}
		\centering
		\includegraphics[trim=0.0in 0.0in 0.0in 0.0in, clip, height=0.24\paperwidth]{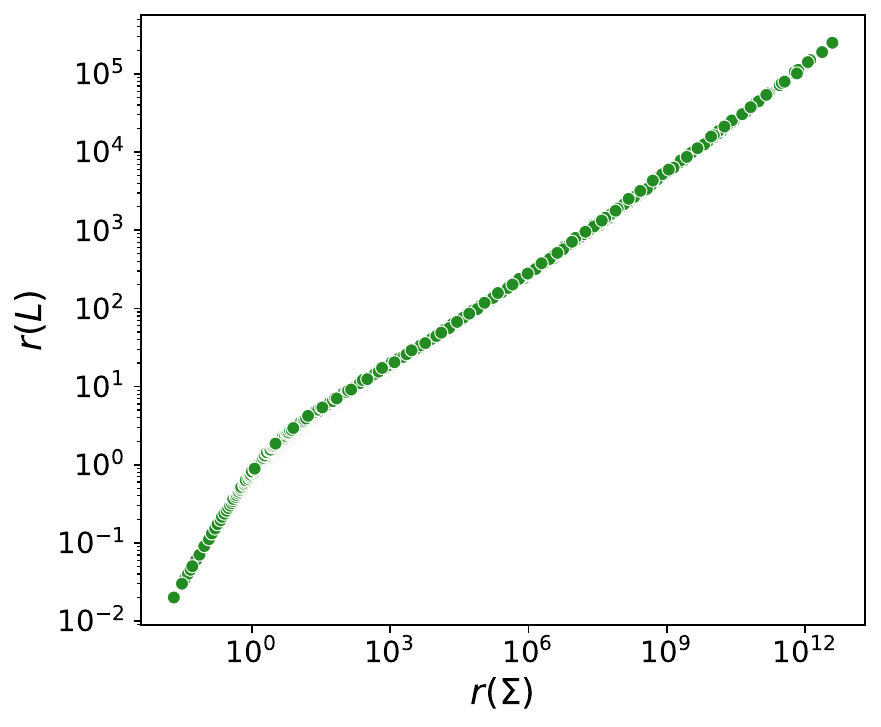}
		\caption{}
	\end{subfigure} 
	\caption{\label{fig:sparsity_performance} (a) $ \ell_2 $ error ratio $ \| \hat{U}_{\tau} \hat{D} \hat{U}_{\tau}^{\T}  - \Sigma \|_{2} /  \| \hat{\Sigma}_{\tau} - \Sigma \|_{2} $ (left-axis, blue) and $ r(\Sigma) $ (right axis, black) plotted against the ratio  $ r(L) / r(\Sigma) $. (b) Maximum row-sum of $ L $ versus maximum row sum of the lower-triangular part of $ \Sigma $. Each point in both plots corresponds to a median over $ 100 $  simulations, with $ n = 150 $, $ p = 100$, for each combination of $ \epsilon $ and percentage of non-zero entries of $ B $. }
\end{figure}

The performance of $ \hat{O}_{\tau} \hat{\Lambda} \hat{O}_{\tau}^{\T} $ and $ \exp(\hat{L}_{\tau}) $ relative to $ \hat{\Sigma}_{\tau} $ is shown in Figure \ref{fig:other_performance}. Interestingly, the pattern of relative behaviour mirrors that of $\hat{U}_{\tau} \hat{D} \hat{U}_{\tau}^{\T} $.

\begin{figure}%[h!]
	\vspace{-0.2cm}
	\centering
	\begin{subfigure}[b]{0.49\textwidth}
		\centering
		\includegraphics[trim=0.0in 0.0in 0.0in 0.0in, clip, height=0.24\paperwidth]{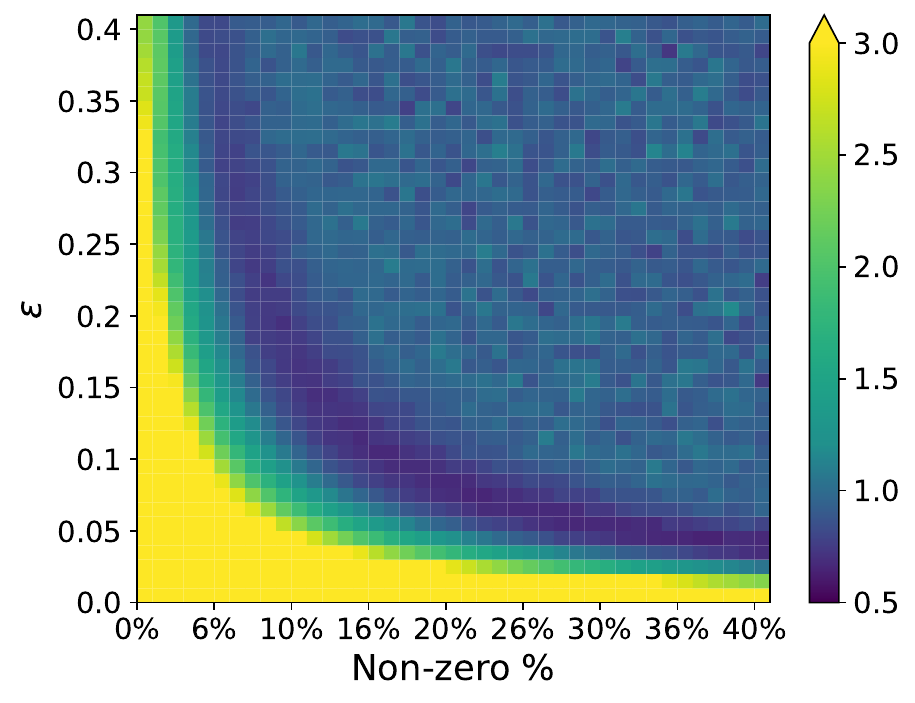}
		\caption{$ \| \hat{O}_{\tau} \hat{\Lambda} \hat{O}_{\tau}^{\T} - \Sigma \|_{2} / \| \hat{\Sigma}_{\tau} - \Sigma \|_{2} $}
	\end{subfigure}
	\hspace{-0.8cm}
	\begin{subfigure}[b]{0.49\textwidth}
		\centering
		\includegraphics[trim=0.0in 0.0in 0.0in 0.0in, clip, height=0.24\paperwidth]{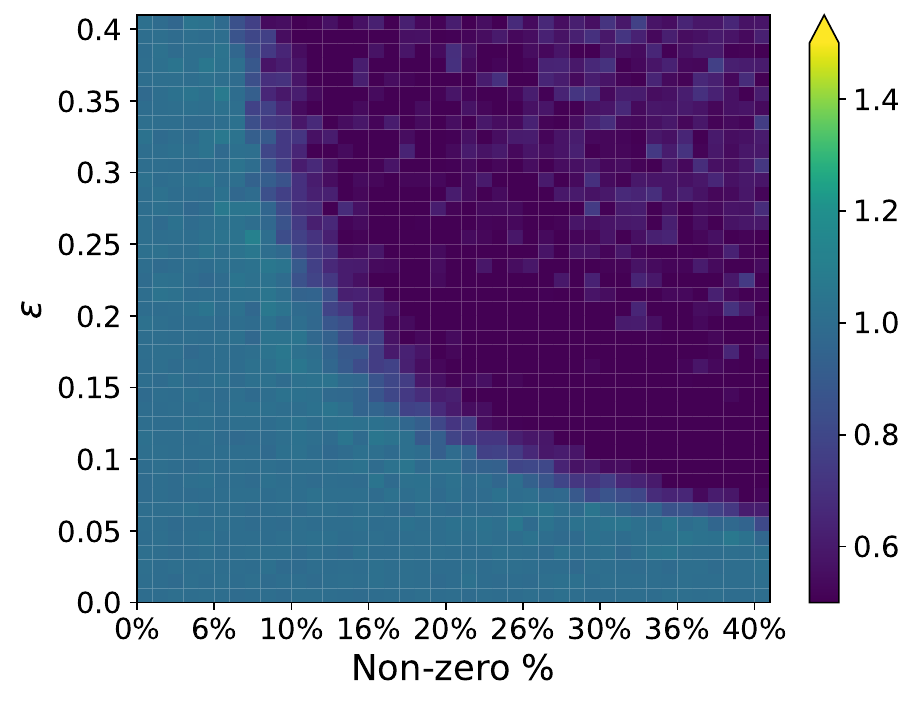}
		\caption{$ \| \hat{O}_{\tau} \hat{\Lambda} \hat{O}^{\T}_{\tau} -  \Sigma \|_{2} / \| \hat{U}_{\tau} \hat{D} \hat{U}_{\tau}^{\T} - \Sigma \|_{2} $}
	\end{subfigure} 
	\\[-0ex]
	\vspace{0.5cm}
	\par
	\centering
	\begin{subfigure}[b]{0.49\textwidth}
		\centering
		\includegraphics[trim=0.0in 0.0in 0.0in 0.0in, clip, height=0.24\paperwidth]{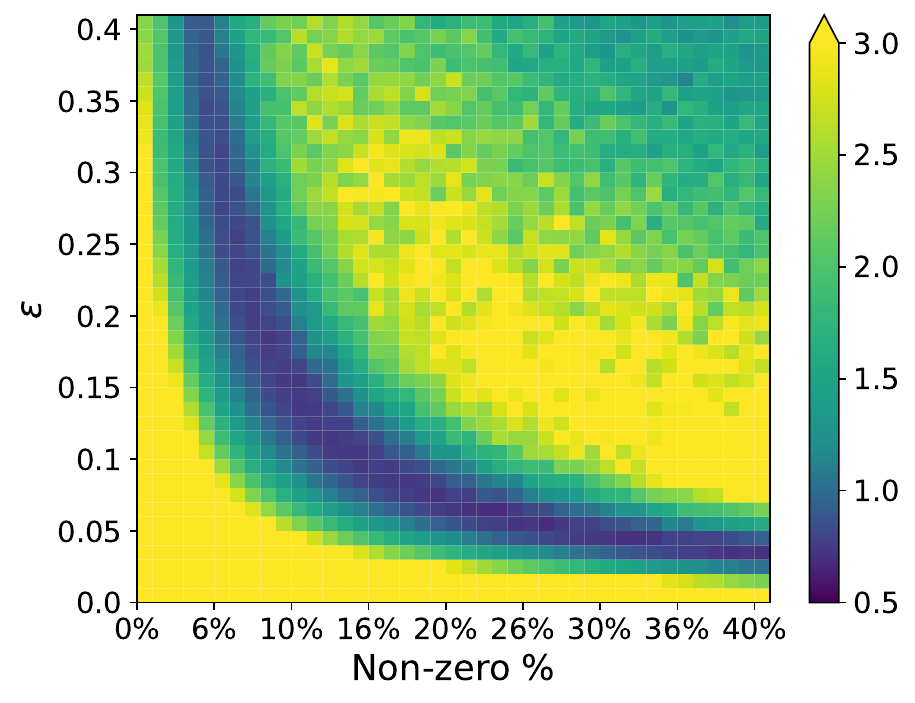}
		\caption{$ \| \exp(\hat{L}_{\tau}) - \Sigma \|_{2}  / \| \hat{\Sigma}_{\tau} - \Sigma \|_{2} $ }
	\end{subfigure}
	\hspace{-0.8cm}
	\begin{subfigure}[b]{0.49\textwidth}
		\centering
		\includegraphics[trim=0.0in 0.0in 0.0in 0.0in, clip, height=0.24\paperwidth]{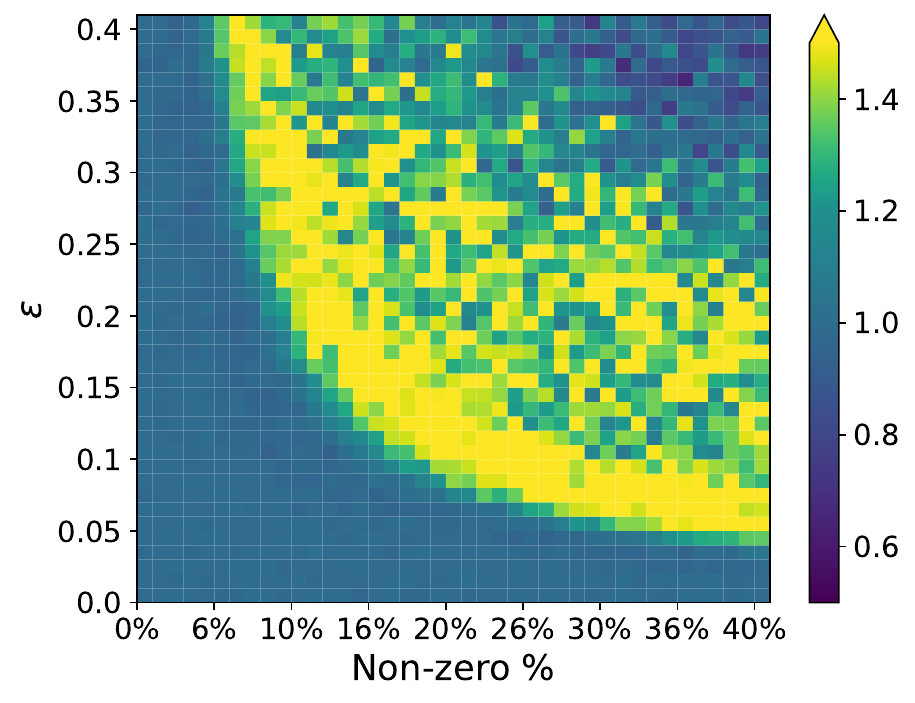}
		\caption{$ \| \exp(\hat{L}_{\tau}) -  \Sigma \|_{2} / \| \hat{U}_{\tau} \hat{D} \hat{U}_{\tau}^{\T} - \Sigma \|_{2} $}
	\end{subfigure} 
	\caption{\label{fig:other_performance} Relative $ \ell_2$ errors for different combinations of $ \epsilon $ ($y$-axis) and levels of sparsity of $ B $, measured by the percentage of non-zero entries ($x$-axis). Each entry corresponds to the ratio of median errors over $ 100 $ Monte Carlo simulations with $ n = 150 $ and $ p = 100 $.}
\end{figure}

\subsection{Classification of leukemia patients}\label{secReal}

We assessed the use of the new sparsity scales for classification of leukemia patients from high-dimensional observations. The details are described in Appendix \ref{appReal}, where the results show a slight reduction in the misclassification rates, from an mean (median) of 6\% (5\%) on the original scale to 4\% (5\%) after reparametrization, and to 4\% (2\%) for the $\Sigma_{ltu}$ reparametrization.

\section{Closing discussion}\label{secClosing}

The work has uncovered insights into the interpretation of sparsity on non-standard scales, identifying situations in which an assumption of sparsity might be more reasonable on a transformed scale and elucidating the graphical models interpretation of zeros after reparametrization. We have touched only briefly on methodological aspects. When the covariance matrix or its inverse is a nuisance parameter, the most compelling open questions probably relate to how one might test for sparsity across several different scales, or find the best sparsity scale empirically. The work also points to the development of more sophisticated estimators than those used in the simulations of \S \ref{secSim}, perhaps in the vein of \citet{Zwiernik2025}, who proposed an elegant formulation covering constraints in the $\Sigma_{pd}$ parametrization. The algebraic geometry of logarithmically sparse positive definite matrices has recently been considered by \citet{GibbsManifolds} and \citet{Pavlov2024}. Presumably an analogous algebraic geometric treatment would be feasible for the other parametrizations, although the statistical implications of such insights are unclear.

%%%%%%%%%%%%%%%%%%%%%%%%%%%%%%%%%%%%
%%%%%%%%%%%%%%%%%%%%%%%%%%%%%%%%%%%%
%%%%%%%%%%%%%%%%%%%%%%%%%%%%%%%%%%%%
%%%%%%%%%%%%%%%%%%%%%%%%%%%%%%%%%%%%
%%%%%%%%%%%%%%%%%%%%%%%%%%%%%%%%%%%%

\section*{Acknowledgements}
We are grateful to two anonymous referees for careful reading and probing questions. HB acknowledges support from an EPSRC fellowship (EP/T01864X/1).
KB thanks Sergey Oblezin for helpful discussions on Lie group decompositions, and acknowledges support from grants EPSRC EP/V048104/1, EP/X022617/1, NSF 2015374 and NIH R37-CA21495.

\section*{Supplementary material}

The supplementary material contains Appendices A--K referenced in the main text. These include proofs and technical details, details of the estimator discussed in \S \ref{secNewEstLTU}, assumptions and proofs for Proposition \ref{th:sigma_consistent:maintext} and extensive numerical work.

\bigskip
\bigskip

\newpage

\beginsupplement

\begin{appendix}
	
	\section{Matrix groups}\label{appGroupTheory}
	The following concepts from group theory are relevant to parts of the exposition and proofs. The action of a group $G$ on a set $X$ is a continuous map $G \times X \to X, (g,x) \to gx$. The \emph{orbit} $[x]$ of $x\in X$ is the equivalence class $\{gx:g \in G\}$, and the set of orbits $X/G:=\{[x]: x \in X\}$ is a partition of $X$ known as the \emph{quotient} of $X$ under the action of $G$. A group action is said to be \emph{transitive} if between any pair $x_1,x_2 \in X$ there exists a $g \in G$ such that $x_2=gx_1$; in other words, orbits of all $x \in X$ coincide. The subset $G_x=\{g \in G:gx=x\}$ of $G$ that fixes $x$ is known as the \emph{isotropy group} of $x$ and if $G_x$ equals the identity element for every $x$, the group action is said to be \emph{free}. 
	
	For every subgroup $H \subset G$ we can consider the (right) coset, or the quotient, $G/H=\{Hg: g \in G\}$ consisting of equivalence classes of $g$, where $\tilde g \sim g$ if $h\tilde g=g$ for some $h \in H$. For groups $G$ that act transitively on $X$, the map $X \to G/G_0$ is a bijection, where $G_0$ is the isotropy group of the identity element. 	
	
	A group $G$ that is also a differential manifold is a Lie group. Lie groups thus enjoy a rich structure given by both algebraic and geometric operations. The tangent space at the identity element, denoted by $\g$, has a special role in that, together with the group operation, it generates the entire group, i.e. every element $g \in G$ can be accessed through elements of $\g$ and the group operation. It is referred to as the Lie algebra and is a vector subspace of the same dimension as the group. Relevant to the matrices introduced in \S \ref{secNotation}, we have:
	\begin{enumerate}
		\item[(i)] If $G=\gl(p)$ with matrix multiplication as the group action, then $\g=\m(p)$.
		\item[(ii)] If $G=\ortho(p)$ or $\so(p)$ with matrix multiplication as the group action, then $\g=\sk(p)$, the set of skew-symmetric matrices. 
		\item[(iii)]  If $G=\lt_{+}(p)$ with matrix multiplication as the group action, then $\g$ is the set $\lt(p)$ of lower triangular matrices. 
		\item [(iv)] If $G=\ltu(p)$ with matrix multiplication as the group action, then $\g$ is the set $\lts(p)$ of lower triangular matrices with zeros along the diagonal. 
		
		\item[(iv)]  If $G=\pd(p)$ with logarithmic addition as the group action \citep{Arsigny} then $\g=\sym(p)$. 
	\end{enumerate}
	When $G$ is a matrix Lie group, the usual matrix exponential can be related to the map $\expo: \g \to G$  such that $\expo(A)=e^A$. Properties of $\expo$ (e.g., injectivity) will depend on that of the matrix exponential, and hence on the topology of $G$. If $G$ is compact and connected then $e$ is surjective. If it is injective in a small neighbourhood around the origin in $\g$, then it bijective. Since $G$ is also a differentiable manifold, a geometric characterization of the matrix exponential $e$ is that it will coincide with the Riemannian exponential map under a bi-invariant Riemannian metric on $G$. 
	
	An excellent source of reference for matrix groups is \cite{AB}.
	
	\section{Formalized definition of reparametrization}\label{sec:PD}

	Let $\vvec: \sym(p) \rightarrow \mathbb{R}^{p^2}$ be the vectorization map taking a symmetric matrix to a $p^2$-dimensional column vector. Define the half-vectorization map 
	\[
	\halfvec:\sym(p) \to \mathbb R^{p(p+1)/2}, \quad \halfvec(x):=A \vvec(x), 
	\]
	where the matrix 
	\[
	A:=\sum_{i \geq j}(u_{ij}\otimes e_j^\T \otimes e^\T_i) \in \mathbb R^{p(p+1)/2 \times p^2}
	\]
	picks out the upper triangular part of the vectorization,
	and $u_{ij}$ is a $p(p+1)/2$-dimensional unit vector with 1 in position $(j-1)p+i-j(j-1)/2$ and 0 elsewhere.  Its inverse 
	\[
	\mathbb R^{p(p+1)/2}  \ni x \mapsto  \halfvec^{-1}(x):=(\halfvec(I_p)^\T \otimes I_p)(I_p \otimes x) \in \sym(p)
	\]
	exists through the Moore-Penrose inverse of $A$. Let 
	\[
	\cone_p:=\left\{\sigma \in \mathbb R^{p(p+1)/2}: \left \langle \halfvec^{-1}(\sigma) y,y\right\rangle >0, \thickspace y \in \mathbb R^{p} \right \},
	\]   
	be a constrained set within $\mathbb R^{p(p+1)/2}$. The definition of $\cone_p$ implicitly engenders an injective parametrization 
	\[
	f: \cone_p \to \mathbb \sym(p), f(\sigma)=\halfvec^{-1}(\sigma),
	\] 
	with image $f(\cone_p)=\pd(p) \subset \sym(p)$. Since $\pd(p)$ is open in $\sym(p)$, with respect to $f$ it is a parametrized submanifold of $\sym(p)$.

	A \emph{reparametrization} of $\pd(p)$ corresponds to an injective map $h: N \to \sym(p)$ from a domain $N$ with non-singular derivative such that there is a diffeomorphism $\psi:N \to \cone_p $ with $h=f \circ \psi$. The derivative condition ensures that $N$ is of dimension $p(p+1)/2$. The two parametrizations $f$ and $h$ are said to be equivalent since $f(\cone_p)=h(N)=\pd(p)$, and $h$ is a reparametrization of $f$ (and vice versa). The commutative diagram in the left half of Figure \ref{commfig} illustrates the type of reparametrization used in this paper. 
	\begin{figure}[!ht]
		\begin{center}
			\begin{tikzcd}
				N \arrow[r,"\psi"] \arrow[dr, "h"]
				&  \cone_p \arrow[d,"f"]\\
				& \pd(p) \arrow[r, "g"]
				&\big\{\mathbb P_{\Gamma}: \Gamma \in \pd(p)\big\}
			\end{tikzcd}
		\end{center}
		\caption{The map $h$ is a reparameterization of $\pd(p)$ initially parameterized using $f$. In contrast,  map $g$ parametrizes a statistical manifold of  parametric probability measures.}
		\label{commfig}
	\end{figure}
	
	In contrast, the statistical model or manifold is determined via an injective map $g: \pd(p) \to \big\{\mathbb P_{\Gamma}: \Gamma \in \pd(p)\big\}$ that maps a covariance matrix $\Gamma$ to a parametric probability measure $\mathbb P_{\Gamma}$ on some sample space. Every $\Sigma$ is obtained from unique points in $\cone_p$ and $N$, and the statistical model given by $g$ is impervious to reparametrization of the manifold $\pd(p)$. Reparametrization of the statistical model amounts to applying a diffeomorphism $\pd(p) \to \pd(p)$ such that correspondences $\Gamma \mapsto\mathbb P_{\Gamma}$ change but not the image $g(\pd(p))$. This form of reparametrization is not considered in the present work.
	
	A matrix  $\Sigma \in \pd(p)$ with respect to the parametrization $f$ is sparse if $\sigma \in \cone_p$ is sparse. On the other hand, the structure of sparsity in $\Sigma$ with respect to the domain $N$ depends on the map $h$ and how $N$ is prescribed coordinates. We will use sparsity to refer  to the domain or to the range of a parametrization interchangeably, with context disambiguating the two.

	\section{Legitimacy of the four maps from Section \ref{secRepara}}\label{secInversion}

	Consider first the $\Sigma_{pd}$ parametrization. Starting from the natural parametrization $\sigma \mapsto f(\sigma)=\Sigma$, where $\sigma\in \cone_p$, the same $\Sigma$ is reached via the more circuitous route $\sigma \mapsto e\circ b_{sym}\circ \phi_{pd}(\sigma)$ involving the composition of three maps (see the upper left panel of Figure \ref{figAll}). This composition consists of a diffeomorphism $\phi_{pd}:\cone_p\rightarrow \mathbb{R}^{p(p+1)/2}$ that takes $\sigma$ to $\alpha$, a bijective map $b_{sym}: \RR^{p(p+1)/2} \to \sym(p)$ that maps $\alpha \in \RR^{p(p+1)/2}$ to the symmetric matrix $L(\alpha)$ via the expansion \eqref{eqBasisExp} in the canonical basis $\mathcal{B}_{sym}$, and the matrix exponential $e:\sym(p) \rightarrow \pd(p)$. The legitimacy of this parametrization is ensured by Propositions C.2 and C.3 below.
	
	The second parameterization $\Sigma_o$ is also not new: this was considered by \citet{RybakBattey2021} who applied the matrix logarithm to $O$ in the spectral decomposition $\Sigma=O\Lambda O^\T$, where $O\in\ortho(p)$ is an orthonormal matrix of eigenvectors and $\Lambda=e^D\in\Dp(p)$ is a diagonal matrix of corresponding eigenvalues. Without loss of generality \citet{RybakBattey2021} took the representation in which $O\in\so(p)$ and considered the map $\log: \so(p) \rightarrow \sk(p)$, yielding a different vector space from that in $\Sigma_{pd}$ in which to study sparsity. Allowance for additional sparsity via $d=\text{diag}(D)$ can be easily incorporated and corresponds to further structure. The composition of three maps described in Figure 1 (top right) consists of a map $\phi_o:\cone_p \to \RR^{p(p+1)/2}$, a map $b_{sk}$ from $(\alpha,d)$ to $D(d)$ and $L(\alpha)$ via \eqref{eqBasisExp} in the canonical basis $\mathcal{B}_{sk}$, and the matrix exponential $e:\sk(p) \rightarrow \pd(p)$ and $e:\D(p) \rightarrow \Dp(p)$. The situation regarding invertibility of the maps is more nuanced than for $\Sigma_{pd}$, owing to the non-uniqueness of the decomposition $\Sigma=O\Lambda O^\T$  and multivaluedness of the matrix logarithm of $O\in \so(p)$. For the purpose of the present paper, the implications are negligible, as we can make the parametrization injective under some conditions in $\Sigma$. This is clarified in Proposition C.4. 
	
	The situation is analogous for the two new reparametrization maps $\Sigma_{lt}$ and $\Sigma_{ltu}$, depicted in the bottom row of Figure \ref{figAll}. The constructions can alternatively be expressed in terms of upper triangular matrices with analogous parametrizations $\Sigma_{ut}$ and $\Sigma_{utu}$ and there are no substantive differences in the conclusions of section \ref{secStructure} and section \ref{secNotionalGaussian}. As with $\Sigma_o$, invertibility of $\Sigma_{lt}$ is not guaranteed without further constraints, since $e:\lt(p) \to \ltp(p)$ is not injective, while the $\Sigma_{ltu}$ parametrization enjoys invertibility without any restrictions on the parameter domain (Proposition C.4). The so-called LDL decomposition of $\Sigma$ is $\Sigma=U\Psi U^\T$, $U\in\ltu(p)$ where $\Psi=e^D\in\Dp(p)$. Analogously to the previous cases, the matrix logarithm $\log:\ltu(p)\rightarrow \lts(p)$ is applied to $U$ and represented in the canonical basis $\mathcal{B}_{ltu}$. 
	
	\begin{multicols}{2}
		
		\hspace{0.4cm} \begin{tikzcd}
			\alpha \arrow[d, "b_{sym}"] 
			&  
			\arrow[l, "\phi_{pd}"] 
			\arrow[d, "\halfvec^{-1}"] \sigma \\
			{\small L\in \sym(p) } \arrow{r}{e^L}
			&  {\small \pd(p) \ni \Sigma}
		\end{tikzcd}   
		\vspace{0.6cm}

		\hspace{0.2cm}	\begin{tikzcd}[row sep=normal, column sep=large]
			\alpha \arrow[d, "b_{lt}"] 
			& 
			\arrow[l, "\phi_{lt}"]
			\arrow[d, "\halfvec^{-1}"] \sigma \\
			{\small L\in	\lt(p) } \arrow{r}{e^L (e^L)^\T}
			& {\small \pd(p) \ni \Sigma}
		\end{tikzcd}
		
		\hspace{0.5cm}\begin{tikzcd}[row sep=normal, column sep=5em]
			(\alpha,d) \arrow[d, shift right=1.5ex,"b_{sk}"] \arrow[d, shift left=1.9ex]
			& 
			\arrow[l, "\phi_{o}"] 
			\arrow[d, "\halfvec^{-1}"] \sigma \\		
			\hspace{-1.5cm} {\small L,D \in \sk(p) \times \D(p)} \arrow[r, "e^L e^D (e^L)^\T"]
			&  {\small \pd(p) \ni \Sigma}
		\end{tikzcd}
		\vspace{0.2cm}

		\hspace{0.3cm}	\begin{tikzcd}[row sep=normal, column sep=5em]
			(\alpha,d) \arrow[d, shift right=1.5ex,"b_{ltu}"] \arrow[d, shift left=2.4ex]
			& 
			\arrow[l, "\phi_{ltu}"] \arrow[d, "\halfvec^{-1}"] \sigma \\		
			\hspace{-1.5cm} {\small L, D \in	\lts(p) \times \D(p)} \arrow[r, "e^L e^D (e^L)^\T"]
			&  {\small \pd(p) \ni \Sigma }
		\end{tikzcd}
	\end{multicols}
	\begin{figure}[h]
		\begin{center}
			\vspace{-1cm}
			\caption{Reparametrization maps for $\protect \Sigma_{pd} $ (top left), $\Sigma_{o}$ (top right), $\Sigma_{lt}$ (bottom left), and $\Sigma_{ltu}$ (bottom right). \label{figAll} }
		\end{center}
	\end{figure}

	Proposition \ref{propRepara} establishes existence of the maps $\phi_\bullet$ introduced above.
	\begin{proposition}\label{propRepara}
		The convex set $\cone_p$ is diffeomorphic to $\mathbb R^{p(p+1)/2}$. 
	\end{proposition}
	
	\begin{proof}
		The set $\pd(p)$ is a symmetric space of dimension $p(p+1)/2$ of noncompact type and can thus has nonpositive (sectional) curvature when equipped with a Riemannian structure \citep{SG}. The map $\halfvec^{-1}: \cone_p \to \pd(p)$ is injective, and we can thus pullback the metric from $\pd(p)$ to $\cone_p$ making it non-positively curved. The set $\cone_p$ is simply connected and complete, and by the Cartan-Hadamard theorem \citep{SG} it is diffeomorphic to $\mathbb R^{p(p+1)/2}$.
	\end{proof}
	
	The inverses $\phi_\bullet^{-1}$ determine precisely how sparsity in $\alpha$ or $(\alpha,d)$ manifests in a point in the convex cone, and thus, quite straightforwardly, in the covariance matrix $\Sigma(\alpha)$. However, they are difficult to determine in closed form. The maps $\Sigma_{\bullet}$ prescribe a path from $\alpha$ to $\Sigma(\alpha)$ (and similarly for  $(\alpha,d)$) and can be viewed as suitable surrogates, but need not be diffeomorphisms even when injectivity is guaranteed.
	
	Injectivity of $\Sigma_{pd}, \Sigma_{o}, \Sigma_{lt}$ and $\Sigma_{ltu}$ hinge on injectivity of the matrix exponential, and uniqueness of eigen, Cholesky and LDL decompositions for the latter three. We first consider the matrix exponential. 
	
	Propositions \ref{propExistenceLog} and \ref{propUniquenessLog} describe conditions for existence and uniqueness of the matrix logarithm, which affect invertibility of the four reparametrizations in \S \ref{secRepara}.
	\begin{proposition}[\cite{Culver1966}]\label{propExistenceLog}
		Let $M\in \M(p)$. There exists an $L\in \M(p)$ such that $M=e^L$ if and only if $M\in\gl(p)$ and each Jordan block of $M$ corresponding to a negative eigenvalue occurs an even number of times.
	\end{proposition}
	\begin{proposition}[\cite{Culver1966}]\label{propUniquenessLog}
		Let $M\in\M(p)$ and suppose that a matrix logarithm exists. Then $M=e^L$ has a unique real solution $L$ if and only if all eigenvalues of $M$ are positive and real, and no elementary divisor (Jordan block) of $M$ corresponding to any eigenvalue appears more than once.
	\end{proposition}

	Proposition \ref{propExistenceLog} covers all four logarithm maps $\log: \pd(p)\rightarrow \sym(p)$, $\log:\so(p)\rightarrow \sk(p)$, $\log: \ltp(p)\rightarrow \lt(p)$ and $\log: \ltu(p)\rightarrow \lts(p)$. Conditions that ensures uniqueness in Proposition \ref{propUniquenessLog} are satisfied only by $\log: \pd(p)\rightarrow \sym(p)$ and $\log: \ltu(p)\rightarrow \lts(p)$. A geometric version of the sufficient condition (``if" part) in Proposition \ref{propUniquenessLog} claims uniqueness if $M$ lies in the ball $\mathcal B_{I_p}(1):=\{X \in \m(p): \|X-I_p\|_2 <1\}$ around $I_p$, where $\| X\|_2$ is the spectral norm of $X$. This provides a sufficient (not necessary) condition to ensure that $\log(\exp Y)=Y$.

	Relatedly, perhaps more appropriate from the perspective of reparametrization of $\pd(p)$, are conditions that ensure injectivity of the matrix exponential $e: \m(p)\to \gl(p)$. As a consequence of Proposition \ref{propUniquenessLog}, $e$ is injective when restricted to Lie subalgebras $\sym(p)$ and $\lts(p)$, but not $\lt(p)$ and $\sk(p)$. The geometric version of the sufficient condition in Proposition \ref{propUniquenessLog} then asserts that the matrix exponential is injective when restricted to the  ball $\mathcal{B}_{0}(\ln 2):=\{L \in \m(p): \|L\|_2 <\ln 2\}$ around the origin $0$ (zero matrix) within $\m(p)$ \citep[e.g.][Proposition 2.4]{AB}. This provides a sufficient (not necessary) condition to ensure that $\exp(\log X)=X$. The condition is close to being necessary for $e: \sk(p)\to \so(p)$ and $e: \lt(p)\to \ltp(p)$. For example, $e: \sk(p) \to \so(p)$ with $p=2$ maps
	\[
	2\pi nB_1=
	2\pi n
	\begin{pmatrix}
	0 & 1\\
	-1  & 0
	\end{pmatrix}
	= 
	\begin{pmatrix}
	0 & 2 \pi n\\
	-2\pi n & 0
	\end{pmatrix}, 
	\quad 
	n \in \mathbb Z
	\]
	to the identity $I_2$; thus, $\Sigma_o((2\pi,d))=\Sigma_o((4\pi,d))$ for any fixed $d \in \RR^2$. The issue arises because skew-symmetric matrices of the form considered comprise the kernel of $e: \sk(p) \to \so(p)$. We see that $\|2\pi B_1\|_2>\ln 2$ and thus violates the sufficient condition.

	Moving on to the decompositions, the following proposition elucidates on conditions that ensure injectivity of the four maps from \S \ref{secRepara} and legitimize them as reparametrizations of $\pd(p)$. 
	\begin{proposition} \label{propInjectivity}
		\begin{enumerate}
			\itemsep 0.5em
			\item []
			\item [(i)]The maps $\alpha \mapsto \Sigma_{pd}(\alpha)$ and $\alpha\mapsto \Sigma_{ltu}$ are injective on $\RR^{p(p+1)/2}$. 
			\item[(iii)] Assume that the elements of $d$ are distinct. The map $(\alpha, d)\mapsto \Sigma_o(\alpha, d)$ is injective when $\alpha$ is restricted to $N_{o} \subset \RR^{p(p-1)/2}$ such that the image $b_{sk}(N_{o} \times \RR^p) \subseteq \mathcal{B}_0(\ln 2) \times \D(p)$ within $\sk(p) \times \D(p)$, and upon choosing $ORP^\top$ for a particular permutation $P \in \p(p)$ of and combination of signs $R \in \D(p) \cap \ortho(p)$ for the columns of $O$ and a permutation $P \Lambda P^\top$ of elements of $\Lambda$, where $b_{sk}$ is as in \S \ref{secRepara}.
			\item[(ii)] The map $\alpha \mapsto \Sigma_{lt}(\alpha)$ is injective when restricted to $N_{lt} \subset \RR^{p(p+1)/2}$ such that the image $b_{lt}(N_{lt}) \subseteq \mathcal{B}_0(\ln 2)$ within $\lt(p)$, where $b_{lt}$ is as in \S \ref{secRepara}.  
		\end{enumerate}
	\end{proposition}
	\begin{proof}
		Injectivity of $\Sigma_{{pd}}$ follows from Proposition \ref{propUniquenessLog}. It is well-known that the Cholesky and LDL decompositions as used in the definitions of $\Sigma_{lt}$ and $\Sigma_{ltu}$ respectively are unique \citep{GV}. Uniqueness of the LDL decomposition also stems from uniqueness of the Iwasawa decomposition of $\gl(p)$ \citep{Terras} through the identification $\pd(p) \cong \gl(p)/\ortho(p)$. The maps $L(\alpha) \mapsto e^{L(\alpha)}(e^{L(\alpha)})^\T$ and $(L(\alpha),D(d)) \mapsto e^{L(\alpha)}e^{D(d)}(e^{L(\alpha)})^\T$ are thus injective. When combined with injectivity of the exponential map $e:\lts \to \ltu$ from Proposition \ref{propInjectivity}, the parameterization $\Sigma_{ltu}$ is injective.

		The situation concerning uniqueness of the eigen decomposition $\Sigma=O\Lambda O^\T$ is involved, even after restriction to a subset of $\sk(p)$ that renders the exponential $e:\sk(p) \to \so(p)$ injective. First note that $O \in \so(p)$ under our parameterization using the exponential map. Then, observe that $ORR^\T\Lambda RR^\T O^{\T}=O\Lambda O^\T$ for any $R \in \so(p)$, and thus pairs $(OR,R^{\T}\Lambda R)$ map to the same $\Sigma$ for \emph{every} $R \in \so(p)$ for which $R^{\T}\Lambda R=\Lambda$, since $OR \in \so(p)$. Indeed, $\{R \in \so(p): R^{\T}\Lambda R=\Lambda\}$ fixes $\Lambda$ and is the isotropy subgroup $\so(p)_{\Lambda}$ in $\so(p)$ (see Supplementary Material \ref{appGroupTheory}). In addition to $\so(p)_{\Lambda}$, another source of indeterminacy comes from permutations $OP$ and $P\Lambda P^{\T}$, with  $P \in \p(p)$ and $|P|=1$ so that $P \in \so(p)$. Put together, this implies that every pair $(ORP, P\Lambda P^{\T})$ maps to the same $\Sigma$ as long as $R \in \so(p)_{\Lambda}$ and $P \in \so(p)$.

		The situation can be salvaged if the positive elements of $\Lambda$ are all distinct so that $\so(p)_{\Lambda}$ reduces to the set $\D(p) \cap \so(p)$ of diagonal rotation matrices with $\pm 1$ entries \citep[Theorem 3.3]{GJS}. In this case, the map $\pi: \so(p) \times \D_{+}(p) \to \pd(p)$ is a $2^{p-1} p!$ covering map with fibers $\pi^{-1} (\Sigma)$ consisting of matrices obtained by $p!$ permutations of elements of $\Lambda$, and a similar permutation of eigenvectors of $\Sigma$, and $2^{p-1}$ matrices in the set $\D(p) \cap \so(p)$ of diagonal matrices mentioned above with unit determinant, which determine signs of the eigenvectors of $\Sigma$; there are $2^{{p-1}}$ such diagonal matrices and not $2^{p}$ owing to the unit determinant constraint.

		Uniqueness can be ensured upon identifying a global \emph{cross section} $S_o \subset \so(p) \times \D_{+}(p)$ that picks out one element from every fiber such that $\pi^{-1}(\Sigma)\cap S_o$ is a singleton for every $\Sigma \in \pd(p)$. For example, $S_{o}$ can be defined by selecting a particular permutation $P \Lambda P^{\T}$ and $ORP$ of the eigenvalues and eigenvectors of $\Sigma$ (e.g., elements of $\Lambda$ arranged in a decreasing order); since $R \in \D(p) \cap \so(p)$, a fixed rule for choosing signs of the eigenvectors determines a unique $R$. Then, $S_o$ contains pairs $(ORP^\T , P\Lambda P^\T)$ for a fixed permutation $P \in \so(p)$. The cross section $S_o$ is bijective with the quotient $(\so(p) \times \D_{+}(p))/\sim$ under the equivalence relation $\sim$ that identifies any two pairs $(O,\Lambda)$ that map to the same $\Sigma$.

		The proof for injectivity of $\Sigma_{lt}$ follows upon noting that the exponential map $e:\lt(p) \to \lt_{+}(p)$ is injective when restricted to the given ball within $\lt(p)$. This completes the proof. 
		
	\end{proof}

	\section{The Iwasawa decomposition of $\gl(p)$ and its Lie algebra}
	\label{sec:GL}
	
	The matrix $X^\T X$ is positive definite for every $X \in \gl(p)$, and the map $X \mapsto X^\T X$ is invariant to the action $(X,O) \to OX$ for $O \in \ortho(p)$ of the orthogonal group. A positive definite matrix $S$ can be transformed to any other under the transitive action
	\[
	(X,S) \to X S X^\T, \quad X \in \gl(p), \thickspace S \in \pd(p)
	\]
	of $\gl(p)$, and $\pd(p)$ is hence a homogeneous space: a differentiable manifold with a transitive differentiable action of $\gl(p)$. For example, between any pair $S_1,S_2 \in \pd(p)$ the invertible matrix $X=S_2^{1/2}S_1^{-1/2}$ transforms $S_1$ to $S_2$ under the above action. The orthogonal group $\ortho(p)$ is the stabilizer of $X=I_p$ and fixes $X \in \gl(p)$,  and we thus obtain the identification with $\gl(p)$ via the group isomorphism
	\begin{equation*}
		\pd(p)\cong \gl(p)/\ortho(p),
	\end{equation*}
	where $\gl(p)/\ortho(p)$ is the set of equivalence classes $[X]:=\{OX: O\in \ortho(p)\}$ or orbits of elements  $X \in \gl(p)$.  The benefit with this representation of $\pd(p)$
	lies in the use of the Iwasawa decompositions of the group $\gl(p)$, and its lie algebra $\M(p)$, to define new parametrizations of $\pd(p)$; the decomposition of $\gl(p)$ corresponds to the LDL decomposition of $\gl(p)$. 
	
	The \emph{Iwasawa decomposition} of $X\in \gl(p)$ determines a unique triple $(O, D, U) \in \ortho(p) \times \Dp(p) \times \ltu(p)$ such that $X=ODU$ \citep[see e.g.][Ch.~4]{Terras}. Since $DU$ is lower triangular with positive diagonal entries, we also recover the well-known QR decomposition. From the Iwasawa decomposition we have $X^\T X = U D^2 U^\T \in \pd(p)$, and we recover the unique LDL decomposition of the positive definite matrix $X^\T X$ \citep{GV}. Additionally, the Iwasawa decomposition into $\ortho(p)$, $\Dp(p)$ and $\ltu(p)$ at the group level ($\gl(p)$) has a corresponding decomposition of the Lie algebra of $\gl(p)$:
	\begin{equation}
		\label{IWdecom}
		\m(p)=\sk(p)\oplus\D(p) \oplus \lts(p),
	\end{equation}
	where $\sk(p)$, $\D(p)$ and $\lts(p)$ are the Lie algebras of $\ortho(p)$, $\Dp(p)$ and $\ltu(p)$, respectively.

	\section{Unification of the four fundamental parametrizations }
	\label{secParamIwa}
	The four parametrizations considered in this work are based on the Iwasawa decomposition of $\gl(p)$ and its Lie algebra $\m(p)$, since the matrices $L(\alpha)$ and $D(d)$ are elements of the Lie algebras in \eqref{IWdecom}.  The map $\alpha \mapsto \Sigma_{lt}(\alpha)$ is based on the sum $\lt_s (p)\oplus \D(p)$ of two constituent Lie subalgebras from \eqref{IWdecom}, which coincides with another Lie subalgebra $\lt(p)$ of $\m(p)$ consisting of all lower triangular matrices. The parametrization $\Sigma_{ltu}$ represents a full use of the Iwasawa decomposition \eqref{IWdecom}. 
	
	The parameterization $\Sigma_{pd}$ relates to the Iwasawa decomposition via the \emph{Cartan decomposition} \citep[p.268]{Terras} of the the Lie algebra $\m(p)$ of $\gl(p)$:
	\[
	\m(p)=\sk(p) \otimes \sym(p),
	\]
	which at the group level corresponds to the singular value decomposition of an invertible matrix. By further decomposing the symmetric part of the Cartan decomposition, the Iwasawa decomposition represents a refinement. In other words, since every $L \in \sym(p)$ can be decomposed as $L=L_s+L_s^\T+D$ for $L_s \in \lts(p)$ and $D \in \D(p)$, we have that
	\[
	\sym(p)=\lts(p)\oplus \D(p).
	\]
	The identification $\pd(p)\cong \gl(p)/\ortho(p)$ implies that the orthogonal component of $\gl(p)$ is ignored in $\Sigma_{pd}$, and the Lie algebra $\sk(p)$ of the orthogonal group $\ortho(p)$ containing the skew-symmetric parts of $\gl(p)$ is thus unused.

	The $\Sigma_o$ parametrization, on the other hand, uses the Lie algebras $\sk(p)$ and $\D(p)$ in \eqref{IWdecom}. However, since every skew symmetric $L \in \sk(p)$ can be decomposed as $L=L_s-L_s^\T$ with $L_s \in \lts(p)$, the Lie algebra $\sk(p)$ can be generated from the Lie algebra $\lt(p)$, and thus links the parameterization $\Sigma_o$ with the Iwasawa decomposition of $\gl(p)$.

	\section{Change of basis}\label{secBasis}

	A change of a matrix basis $\mathcal B=\{B_1,\ldots,B_d\}$ is achieved by the action of a nonsingular $W \in \gl(p)$ as $W\mathcal BW^{-1}:=\{WB_1W^{-1},\ldots,WB_dW^{-1}\}$. The group $\gl(p)$ acts equivariantly on the map $\sum _j \alpha_j B_j \mapsto \Sigma(\alpha)=e^{\sum _j \alpha_j B_j}$, since $e^{WAW^{-1}}=We^AW^{-1}$ for every $A \in \m(p)$ and $W \in \gl(p)$. Hence, 
	\[
	e^{W( \sum _j \alpha_j B_j )W^{-1}}=We^{\sum _j \alpha_j B_j } W^{-1}=W \Sigma(\alpha) W^{-1},
	\]
	may belong to $\pd(p)$ depending on the $W$ chosen. The four maps $\Sigma_{pd},\Sigma_o,\Sigma_{lt},\Sigma_{ltu}$ are thus well-defined only upon fixing a basis $\mathcal B$ for the considered Lie subalgebra.

	\section{Proofs for Section \ref{secStructure}}\label{appProof}
	
	\subsection{Preliminary lemmas}\label{appPrelim}
	
	\begin{lemma}[Axler, 2015]\label{lemmaAxler}
		Let $\V$ be a real inner-product space and let $T:\V\rightarrow \V$ be a linear operator on $\V$ with matrix representation $M=M(T)$. The following are equivalent: (i) $M$ is normal; (ii) there exists an orthonormal basis of $\V$ such that $M=O\tilde{B}O^{-1}$ where $ O $ is orthogonal and the blocks of the block-diagonal matrix $\tilde{B}$ are either $1\times 1$ or $2\times 2$ of the form 
		\begin{equation}\label{eqAxler}
			\begin{pmatrix}
				a & -b \,\\
				b & \;\; a \,
			\end{pmatrix}
			=
			\rho
			\begin{pmatrix}
				\cos \theta  & - \sin \theta \\
				\sin \theta & \;\;\; \cos \theta
			\end{pmatrix}.
		\end{equation}
		where $a,b\in\RR$, $b>0$, $\rho>0$ and $\theta\in[0,2\pi]$. Each $ 1 \times 1 $ block $ \lambda $ is an eigenvalue of $M$, and for each $ 2 \times 2 $ block \eqref{eqAxler}, $a + bi$ and $a - bi$ are eigenvalues of $M$. %\kbc{Matrix $O$ needs to be defined here.}
	\end{lemma}
	
	The representation \ref{eqAxler} in terms of polar coordinates is convenient for subsequent calculations involving the matrix logarithm.
	
	\begin{lemma}\label{lemma5.4}
		Let $M\in\M(p)$ be a normal matrix. The matrix logarithm $L$, if it exists, takes the form $OBO^{-1}$, where $O\in \ortho(p)$  is orthonormal and $B$ is block diagonal with blocks of the form described in Lemma \ref{lemmaAxler}.
	\end{lemma}
	
	\begin{proof}
		By Lemma \ref{lemmaAxler}, $M=O\tilde{B}O^{-1}$, where $O\in\ortho(p)$ is orthonormal and $\tilde{B}$ is block diagonal. Let $\lambda$ be an eigenvalue of $M$. From Proposition \ref{propExistenceLog}, existence of a logarithm requires that any negative real eigenvalues have associated with them an even number of blocks. By Lemma \ref{lemmaAxler} negative eigenvalues appear in $1\times 1$ blocks, since $2\times 2$ blocks correspond to complex conjugate pairs of eigenvalues. It follows that the matrix logarithm of a normal matrix exists if and only if negative eigenvalues have even multiplicity, in which case, we can without loss of generality construct blocks of size $2\times 2$ for a negative eigenvalue $\lambda$ of the form $\tilde{B}_{\lambda}=\lambda I_2 = -|\lambda|I_2$. Then $\log(\tilde{B}_\lambda)=\log\{(|\lambda|I_2)(-I_2)\}$, and since $I_2$ and $-I_2$ commute, $\log(\tilde{B}_\lambda)=\log(|\lambda|I_2)+\log(-I_2)$, where
		\[
		\log(-I_2)=\pi		\begin{pmatrix}
		0 & -1 \\
		1 & 0
		\end{pmatrix}.
		\]
		Thus
		\[
		\log(\tilde{B}_\lambda) = \begin{pmatrix}
		\log|\lambda| & -\pi \\
		\pi & \log|\lambda|
		\end{pmatrix},
		\]
		which is of the form in Lemma \ref{lemmaAxler}. For $2\times 2$ blocks $\tilde{B}_{\CC}$ corresponding to complex conjugate pairs of eigenvalues of $M$, a similar argument together with
		\begin{equation*}
			\log \begin{pmatrix}
				\cos \theta  & - \sin \theta \\
				\sin \theta & \;\;\; \cos \theta
			\end{pmatrix} = \begin{pmatrix}
				0 & -\theta \,\\
				\theta & \;\; 0 \,
			\end{pmatrix}
		\end{equation*}
		shows that
		\[
		\log(\tilde{B}_{\CC}) = \begin{pmatrix}
		\log \rho & -\theta \,\\
		\theta & \;\; \log \rho \,
		\end{pmatrix}
		\]
		which is also of the form of Lemma \ref{lemmaAxler}. 
	\end{proof}
	
	\begin{lemma}\label{lemma5.2}
		Let $M=e^L\in\M(p)$. Then $M$ is normal if and only if $L$ is normal.
	\end{lemma}
	
	\begin{proof}
		Suppose that $L$ is normal, that is $L^\T L=L L^\T$. By the Jordan decomposition $L=QJQ^{-1}$, normality of $L$ implies normality of $J$. The matrix exponential $M=\exp(L)=Q\exp(J)Q^{-1}$ is normal if and only if $\exp(J)^\T\exp(J) = \exp(J)\exp(J)^\T$. Two general properties of the matrix exponential are that for matrices $A,B\in\M(p)$ such that $AB=BA$, $\exp(A)^\T = (\exp(A))^\T$ and $\exp(A)\exp(B)=\exp(A+B)$. Thus $\exp(J^\T)\exp(J)=\exp(J+J^\T)=\exp(J)\exp(J^\T)$ showing that $M$ is normal. The converse statement follows by Lemmas \ref{lemmaAxler} and \ref{lemma5.4}.
	\end{proof}
	
	\begin{lemma}[Weierstrass's M-test, e.g.~Whittaker and Watson, 1965, p.49]\label{lemma5.5}
		Let $S_k(x)=s_1(x)+\cdots +s_k(x)$ be a sequence of functions such that, for all $x$ within some region $\mathcal{R}(x)$, $S_k(x)\leq T_k = t_1+\cdots+t_k$, where $(t_j)_{j\in\mathbb{N}}$ are independent of $x$ and $T_k$ is a positive convergent sequence. Then $S_k(x)$ converges to some limit, $S(x)$ say, uniformly over $\mathcal{R}(x)$. 
	\end{lemma}

	\begin{lemma}\label{lemma5.3}
		For $A\in\M(p)$, define $\psi(A)=\sum_{k=0}^{\infty} A^k/(k+1)!$.
		Then for any operator norm $\|\cdot\|_{\text{op}}$, provided that $\|A\|_{\text{op}}$ is bounded, $\|\psi(A)\|_{\text{op}}\leq \|\exp(A)\|_{\text{op}}$. Additionally, $\psi(A)\in \gl(p)$ for any $A\in \M(p)$.
	\end{lemma}
	
	\begin{proof}
		Let $\psi_k(A)=\sum_{n=0}^k A^n/(n+1)!$ and let $r\geq 0$ be such that $\|A\|_{\text{op}}\leq r$. Since the operator norm is subadditive and submultiplicative
		\[
		\|\psi_k(A)\|_{\text{op}}\leq \sum_{n=0}^k \frac{\|A^n\|_{\text{op}}}{(n+1)!} \leq \sum_{n=0}^k \frac{\|A\|_{\text{op}}^n}{n!} \leq \sum_{n=0}^k \frac{r^n}{n!} \rightarrow e^r
		\]
		as $k\rightarrow \infty$. Lemma \ref{lemma5.5} applies with $S_{k}(A)=\|\psi_k(A)\|_{\text{op}}$.
		
		For the second statement, it suffices by the Jordan decomposition of $A$ to show that $\tilde{\psi}(\lambda)\neq 0$ where $\lambda \in \mathbb{C}$  is an eigenvalue of $A$ and $\tilde{\psi}(x) = \sum_{n=0}^k x^n/(n+1)! $ for $ x \in \CC $. For $\lambda=0$, $\psi(\lambda)=1$ by definition, whereas for $z\in\CC$, $z\neq 0$, $\psi(z)=(e^z-1)/z\neq 0$. 
	\end{proof}
	
	\begin{lemma}\label{lemma4.1}
		Consider $M=e^L\in\M(p)$ where $L\in\V(p)$, a vector space with canonical basis $\mathcal{B}$ of dimension $m$ 
		The matrix $M$ is logarithmically sparse in the sense that $L=L(\alpha)=\alpha_1 B_1 + \cdots + \alpha_m B_m$, $B_j\in \mathcal{B}$ with $\|\alpha\|_0=s^*$ if and only if $M=P \tilde{M}P^\T$, where $P \in \p(p)$ is a permutation matrix and
		\begin{equation}\label{eqTildeM}
			\tilde{M}=\left(\begin{array}{cc}O_{11} & 0 \\ 0 & I_{q}\end{array}\right)\left(\begin{array}{cc}C_{11} & C_{12} \\ 0 & I_{q}\end{array}\right)\left(\begin{array}{cc}O_{11}^\T & 0 \\ 0 & I_{q}\end{array}\right),
		\end{equation}
		where $q\geq p-d_r^*$ and $O_{11}\in \ortho(p-q)$, $C_{11}\in\M(p-q)$. Moreover, if $M$ is normal, $C_{11}$ is also normal and $C_{12}=0$ in \eqref{eqTildeM}.
	\end{lemma}
	
	\begin{proof}
		Suppose first that $M=e^L$ is logarithmically sparse. By definition, $L$ has $p-d_r^*$ zero rows. Thus, there exists a permutation matrix $ P $ such that the last $p-d_r^*$ rows of $P^T L P $ are zero. Thus write $P L P^\top $ as a partitioned matrix with upper blocks $L_{1}$, $L_{2}$ of dimensions $d_r^*\times d_r^*$ and $d_{r}^*\times p-d_r^*$, the remaining blocks being zero. From the definition of a matrix exponential, 
		\[
		e^{P L P^\top} = \left(\begin{array}{cc}\exp(L_{1}) & \sum_{k=0}^\infty L_1^k L_2/(k+1)! \\ 0 & I_{p-d_r^*}\end{array}\right)
		\]
		which is of the form given in equation \eqref{eqTildeM} with $O_{11}=I_{d_r^*}$. The result follows since $ e^{P LP^\top} = P e^{L} P^\top $.

		To prove the reverse direction, we need to construct $ L $ such that, 
		\begin{align}
			\exp(L) = 
			\begin{pmatrix}
				O_{11} C_{11} O^{\T}_{11} & O_{11} C_{12} \\
				0 & I
			\end{pmatrix},
			\quad
			L = 
			\begin{pmatrix}
				A & B \\
				0 & 0
			\end{pmatrix},
			\label{eq:logM_structure}
		\end{align}
		where the number of zero rows of $ L  $ is greater or equal to $ p - d_{r}^{*} $. Let $  \Gamma_{1} = O_{11} C_{11} O_{11}^{\T} $ and $ \Gamma_{2} = O_{11} C_{12} $. Set $ A = \log(\Gamma_{1}) $, which exists by assumption since eigenvalues of $ \Gamma_{1} $ are eigenvalues of $ \tilde{M} $. 
		By the matrix Taylor series expansion of the matrix exponential,
		\begin{align*}
			\exp(L) = 
			\begin{pmatrix}
				\Gamma_{1} & \sum_{k=0}^{\infty} \frac{A^{k}}{(k+1)!} B \\
				0 & I
			\end{pmatrix}.
		\end{align*}
		It thus remains to be shown that $ \sum_{k=0}^{\infty}  \frac{A^{k}}{(k+1)!} B = \Gamma_{2} $ some $ d_{r}^{*} \times p - d_{r}^{*} $ matrix $ B $. Let $ \psi(A) = \sum_{k=0}^{\infty}  \frac{A^{k}}{(k+1)!} $. By Lemma \ref{lemma5.3}, $ \psi(A) $ is invertible, and we can take $ B = \psi(A)^{-1} \Gamma_{2} $. Note that the matrix logarithm of $ \tilde{M} $ is not unique. However, for a real eigenvalue $ \lambda $ of $ \tilde{M} $, the logarithm of the Jordan block $ J_{k}(\lambda) $ has periodicity $ i 2 \pi q I $, $ q \in \ZZ $ \citep{Culver1966}. Thus, for $ \lambda = 1 $, every real matrix $ L $, $ L = \log(\tilde{M}) $ will have the form \eqref{eq:logM_structure} in the sense of the last $ p-d_{r}^{*} $ rows being equal to canonical basis vectors, with a non-zero diagonal element. The result follows by observing that for $ M = P \tilde{M} P^{\T} $, $ \log(M) = PLP^{\T} $.
	\end{proof}

	\begin{lemma}\label{lemma2.2}
		Let $M\in\M(p)$. With $M=e^L$ and $L\in\V(p)\subset \M(p)$, a vector space of dimension $m$, %and canonical basis $\mathcal{B}$, 
		Let $ M = QJQ^{-1} $ be a real Jordan decomposition of $ M $ (e.g.~Horn and Johnson, 2012, p.~202) and let $\mathcal{A}\subset [p]$ denote the set of indices for columns of $Q$ corresponding to eigenvectors whose eigenvalues are not equal to one. Thus, the cardinality $|\mathcal{A}^c|$ of the complementary set is the geometric multiplicity of the unit eigenvalue of $M$, and $|\mathcal{A}|=p-|\mathcal{A}^c|$. The dimension of $\mathcal{A}$ satisfies $|\mathcal{A}|\leq \max\{d_r^*,d_c^*\}$.
	\end{lemma}
	
	\begin{proof}
		Suppose that $|\mathcal{A}|\leq d_r^*$. Since $\rank(L)=\rank(L^\T)$, the geometric multiplicities of the unit eigenvalues of $M$ and $M^\T$ are equal. Thus, $|\mathcal{A}|\leq d_c^*$ and therefore $|\mathcal{A}|\leq \max\{d_r^*,d_c^*\}$. 
		
		To prove that $|\mathcal{A}|\leq d_r^*$, consider the real Jordan decomposition $M=QJQ^{-1}$. Let $q_j$ denote the $j$th column of $Q$. We show that $\ssp\{q_{j}: j\in \mathcal{A}^c\}=\kernel(L)$ by establishing containment on both sides. Let $v\in \ssp\{q_{j}: j\in \mathcal{A}^c\}$. Then there exist coefficients $\beta_j \in \RR$ such that
		\[
		Lv = \sum_{j\in\mathcal{A}^c} \beta_j L q_j = \sum_{j\in\mathcal{A}^c} \beta_j Q\log(J)Q^{-1}q_j = \sum_{j\in\mathcal{A}^c} \beta_j Q\log(J)e_j = 0
		\]
		where the final equality follows since $\lambda_j=1$ for all $j\in \mathcal{A}^c$, so the $j$th diagonal entry of $\log(J)$ is zero. It follows that $\ssp\{q_{j}: j\in \mathcal{A}^c\}\subseteq  \kernel(L)$..
		
		For the converse containment, suppose for a contradiction that there exists $v\in\kernel(L)$ such that $v\notin \ssp\{q_{j}: j\in \mathcal{A}^c\}$. Since $Q$ has full rank, its columns are linearly independent and there exist coefficients $\beta_1,\ldots,\beta_p$, each in $\RR$ such that $v=\beta_1 q_1 + \cdots + \beta_p q_p$. Since $q_j\in\kernel(L)$ for $j\in\mathcal{A}^c$ by the previous argument,
		\[
		0 = Lv = \sum_{j\in \mathcal{A}}\beta_j L q_j = \sum_{j\in \mathcal{A}}\beta_j Q\log(J)e_j.
		\]
		By definition of $\mathcal{A}$, $J_{jj}\neq 1$ for any $j\in \mathcal{A}$, thus the equality $Lv=0$ implies $\beta_j=0$ for all $j\in \mathcal{A}$, a contraction, since the columns $Q\log(J)e_j$, $j\in\mathcal{A}$ are linearly independent.
	\end{proof}
	
	For normal matrices in $\M(p)$, i.e.~those satisfying $M^\T M=M M^\T$, $d_r^*=d_c^*$ and Lemma \ref{lemma2.2} recovers Lemma 2.1 of Battey (2017) and Proposition 3.1 of Rybak and Battey (2021). 
	
	\subsection{Proof of Theorem \ref{prop2.3}}\label{proofProp2.3}

	From Lemma \ref{lemma4.1}, $p-d_r^*$ rows of $M$ are of the canonical form $e_j^\T$, and since zero columns of $L$ are zero rows of $L^\T$, it is also true by Lemma \ref{lemma4.1} applied to $M^\T$ that $p-d_c^*$ columns of $M$ are of canonical form $e_j$. If $d^*$ rows and columns of $M$ are of the canonical form, then $M=P(V\oplus I_{p-d^*})P^\T$, where $P\in\p(p)$ and $V\in\M(d^*)$. The matrix logarithm of $M$ is $L=P(\log (V) \oplus 0_{p-d^*})P^\T$. The converse direction follows by applying the exponential map to $L=P(\log (V) \oplus 0_{p-d^*})P^\T$ and invoking Lemma \ref{lemma4.1} in the converse direction. 
	\qed

	\subsection{Proof of Lemma \ref{lemma:cov_submatrix}}
	Consider a random vector $(Y_{1}^\T, Y_{2}^\T, Y_{3}^\T)^\T$ with a covariance matrix $ \Sigma $. By the assumptions of Lemma \ref{lemma:cov_submatrix}
	\begin{align*}
		\Sigma &= 
		\begin{pmatrix}
			A & 0 & 0 \\
			B & I & 0 \\
			C & 0 & D
		\end{pmatrix}
		\begin{pmatrix}
			A^{\T} & B^{\T} & C^{\T} \\
			0 & I & 0 \\
			0 & 0 & D^{\T}
		\end{pmatrix}
		=
		\begin{pmatrix}
			AA^{\T} & \quad AB^{\T} &   \quad AC^{\T} \\
			BA^{\T} & \quad BB^{\T} + I & \quad B C^{\T} \\
			CA^{\T} & \quad CB^{\T} & \quad CC^{\T} + DD^{\T}
		\end{pmatrix} \\
		& =:
		\begin{pmatrix}
			\Sigma_{11} & \Sigma_{12} & \Sigma_{13} \\
			\Sigma_{21} & \Sigma_{22} & \Sigma_{23} \\
			\Sigma_{31} & \Sigma_{32} & \Sigma_{33}
		\end{pmatrix},
	\end{align*}
	where the matrices $ A $ and $D$ are lower triangular with unit entries on the diagonal. 
	Then, 
	\begin{align*}
		\Sigma_{23} - \Sigma_{21} \Sigma_{11}^{-1} \Sigma_{13} = BC^{\T} - BA^{\T} (AA^{\T})^{-1}  AC^{\T} = 0,
	\end{align*}
	since $ A $ is full-rank by definition. Consider a submatrix
	\begin{align*}
		\Sigma_{bc|a}
		= 
		\begin{pmatrix}
			\Sigma_{23} & \Sigma_{21} \\
			\Sigma_{13} & \Sigma_{11}
		\end{pmatrix}
		= 
		\begin{pmatrix}
			I & \Sigma_{21} \Sigma_{11}^{-1} \\
			0 & I
		\end{pmatrix}
		\begin{pmatrix}
			\Sigma_{23} - \Sigma_{21} \Sigma_{11}^{-1} \Sigma_{13} & 0 \\
			0 & \Sigma_{11}
		\end{pmatrix}
		\begin{pmatrix}
			I & 0 \\
			\Sigma_{11}^{-1} \Sigma_{13} & I
		\end{pmatrix}.
	\end{align*}
	Since $ \Sigma_{11} $ is invertible, and for any two matrices $ M $, $ N $, with compatible dimensions $ \text{rank}(MN) \leq \min\{\text{rank}(M), \text{rank}(N) \}$, the result follows.
	\qed

	\section{Proofs for Sections \ref{secNotionalGaussian} and \ref{secPDO}}
	
	\subsection{Preliminary lemmas}
	
	\begin{lemma}
		\label{th:directed_paths_recursion}
		Let $ \beta_{i.jk} $ denote a regression coefficient of $ Y_j $ in a regression of $ Y_i $ on $ Y_j $ and $ Y_k $. Let $ \upsilon_{ij}(l)  $ denote the effect of $ Y_j $ on $ Y_{i} $ along all paths of length $ l $ in a recursive directed acyclic graph. Then,
		\begin{align*}
			\upsilon_{ij}(l) = 
			\begin{cases}
				\sum_{k = j+l-1}^{i-1} \beta_{i.k [i-1]} \upsilon_{kj }(l-1) \quad \text{if} \quad i - j \geq l, \\
				0 \quad \text{otherwise}.
			\end{cases}
		\end{align*}
	\end{lemma}
	\begin{proof}[Proof of Lemma \ref{th:directed_paths_recursion}]
		We use proof by induction. Consider $ l = 1 $ and take any pair $ (i, j) $ such that $ i \geq j+1 $. Then,
		$ \upsilon_{ij}(1) = \beta_{i.j[i-1]} $ as claimed. Now consider $ l > 1 $. Every path from node $ j $ to node $ i $ can be decomposed into a path of length $ l-1 $ from $ j $ to node $ k $ for some $ k \in \{j+1, \ldots, i-1 \} $ and a path with length one from $ k $ to $ i $. The total effect of $ Y_j $ on $ Y_i $  along such a path is equal to $ \beta_{i.k[i-1]} \upsilon_{kj}(l-1) $. The total effect along all paths of length $ l $ is the sum of single paths over all nodes $ k \in \{j+1, \ldots, i-1 \} $, which yields the result.
	\end{proof} 
	
	\begin{lemma}
		Let $ \beta_{i.k[p] \backslash \{ i \}} $ denote a coefficient of $ Y_{k} $ in a regression of $ Y_{i} $ on $ Y_1, \ldots, Y_{i-1} $, $ 
		Y_{i+1}, \ldots, Y_{p} $. Let $ \upsilon_{ij}^{u}(l) $ denote the total effect of a unit change in $ Y_j$ on $ Y_i$ along all paths of length $ l $ in an undirected graphical model with edge weights given by regression coefficients. Then, 
		\begin{align*}
			\upsilon_{ij}^{u}(l) &= 
			\sum_{k \neq i} \beta_{i.k [p]\backslash \{i\}} \upsilon^{u}_{kj }(l-1), 
			\quad \beta_{i.k[p] \backslash \{ i \}} = -V_{ij} /
			V_{ii}. 
		\end{align*}
	\end{lemma}
	\begin{proof}
		The proof is analogous to that of Lemma \ref{th:directed_paths_recursion}. The only difference is that an edge can exist between any pair of nodes $ (i, j) $, $ i \neq j $ and the effect of $ Y_j $ on $ Y_i $ is given by a regression coefficient $ \beta_{i.j[p] \backslash \{ i \}} = - \tilde{V}_{ij} = - V_{ij} / V_{ii} $ \citep{Lauritzen}. 
	\end{proof}
	
	\subsection{Proofs of main results}
	
	\begin{proof}[Proof of Proposition \ref{th:u_sparsity_interpretation}] 
		For an arbitrary partition $ Y = (Y_{a}^{\T}, Y_{b}^{\T})^{\T} $, let $ \Sigma^{-1} $ be partitioned accordingly as
		\begin{align*}
			\Sigma^{-1} =
			\begin{pmatrix}
				\Sigma^{aa} & \Sigma^{ab} \\
				\Sigma^{ba} & \Sigma^{bb}
			\end{pmatrix}.
		\end{align*}
		The block upper-triangular decomposition takes the form,
		\begin{align*}
			\Sigma^{-1} 
			=
			\Upsilon \Gamma \Upsilon^{\T}
			= 
			\begin{pmatrix}
				I_{aa} & \; \Sigma^{ab} (\Sigma^{bb})^{-1} \\
				0 & I_{bb}
			\end{pmatrix}
			\begin{pmatrix}
				\Sigma^{aa.b} & 0 \\
				0 & \Sigma^{bb}
			\end{pmatrix}
			\begin{pmatrix}
				I_{aa} & 0 \\
				(\Sigma^{bb})^{-1} \Sigma^{ba} & I_{bb}
			\end{pmatrix}, 
		\end{align*}
		where $ \Sigma^{aa.b} = \Sigma^{aa} - \Sigma^{ab} (\Sigma^{bb})^{-1} \Sigma^{ba} $.
		Then, 
		\begin{align*}
			U^{-1} = \Upsilon^{\T} = 
			\begin{pmatrix}
				I_{aa} & 0 \\
				-(\Sigma^{bb})^{-1} \Sigma^{ba} & I_{bb}
			\end{pmatrix}.
		\end{align*}
		The matrix of regression coefficients of $ Y_{b} $ in a regression of $ Y_{a} $ on $ Y_{b} $ is equal to $  (\Sigma^{bb})^{-1} \Sigma^{ba}$  \citep{WermuthCox2004}, which is the negative non-zero off-diagonal block of $U^{-1}$.
		
		By partitioning $ \Sigma^{-1} $ recursively until $ \Upsilon $ is upper-triangular, we obtain that the $i$th row of $ U^{-1} $ contains minus the regression coefficients of $ Y_{i} $ on $ Y_{1}, \ldots, Y_{i-1} $. Let $ \bar{U} = I - \Upsilon^{\T} $. Then, $ \bar{U} \in \lts(p) $ and $\bar{U}_{ij}=\beta_{i.j[i-1]}$ for $j<i$, i.e., the element $ (i,j) $ of $ \bar{U} $ equals the coefficient of $ Y_{j} $ in a regression of $ Y_{i} $ on $ Y_{1}, \ldots, Y_{i-1} $. Then,
		\begin{align}
			U = (I - \bar{U})^{-1} =  I + \sum_{j=1}^{p-1} \bar{U}^{j},
			\label{eq:u_expansion}
		\end{align}
		where we used that, for a nilpotent matrix $ N $ of degree $ k $, $ (I + N)^{-1} = I + \sum_{j=1}^{k-1} (-1)^{j} N^{j} $. The result for $ U_{ij} $ follows from \eqref{eq:u_expansion}.
		
		Using the properties of the matrix logarithm,
		\begin{align*}
			L = \log(U) = \log[(I - \bar{U})^{-1}] = - \log(I - \bar{U})= \sum_{k=1}^{p-1} \frac{\bar{U}^{k}}{k},
		\end{align*}
		which establishes the claim about $ L_{ij} $.
	\end{proof}
	
	\begin{proof}[Proof of Proposition \ref{th:sigma_interpretation}]
		The matrix $ \tilde{V} $ has entries
		\begin{align*}
			\tilde{V}_{ij} = \begin{cases}
				-\beta_{i.j[p] \backslash \{i \}} \hspace{0.5cm} \text{for} \quad i \neq j,  \\
				1 \hspace{1.92cm} \text{for} \quad i = j.
			\end{cases}
		\end{align*}
		Thus, the element $ (i, j) $ of matrix $ I - \tilde{V} $ is equal to the effect of $ Y_j $ on $ Y_{i} $ along a path of length one. Since
		$ (I-\tilde{V})^{l} = (I - \tilde{V}) (I - \tilde{V})^{l-1} $, the element $ (i, j) $ of $ (I - \tilde{V})^{l}  $ is equal to the effect of $ Y_j $ on $ Y_i $ along all paths of length $ l $. Provided that the sum on the right hand side converges,  the power expansion of the matrix inverse and logarithm gives
		\begin{align*}
			\tilde{\Sigma} &= \tilde{V}^{-1} = \sum_{k=0}^{\infty} (I - \tilde{V})^{k} \\
			\log(\tilde{\Sigma}) &= \log(\tilde{V}^{-1}) = \sum_{k=0}^{\infty} \frac{(-1)^{k+1}}{k}(I - \tilde{V})^{k}
		\end{align*}
	\end{proof}
	
	\begin{proof}[Proof of Lemma \ref{lemma:U_zeros}]
		The result follows from a power series expansion of matrix inverse, Lemma \ref{lemma:cov_submatrix}, and Proposition 2.1 and Corollary 2.2 of \cite{Uhler2019gaussian}.
	\end{proof}

	\section{Derivation of equation \eqref{eqApprox}}\label{secLogPerturbation}
	
	A version of this argument appears in Battey (2019).
	
	A function $f$ of a $p\times p$ matrix $A$ satisfies \citep[p.44]{Kato}
	\begin{equation}\label{eqCauchy}
		f(A)=\frac{1}{2\pi i} \ointctrclockwise_{\gamma_A} f(z)(zI-A)^{-1}dz,
	\end{equation}
	where $I$ is the identity matrix and $\gamma_A$ is a simple closed curve lying in the region of analyticity of $f$ and enclosing all the eigenvalues of $A$ in its interior.
	
	From \eqref{eqCauchy}, the error on the scale of the matrix logarithm is
	\[
	\log(\Sigma + \varepsilon I) - \log(\Sigma)=\frac{1}{2\pi i}\left(\ointctrclockwise_{\hspace{1pt}\gamma_\varepsilon} \log(z)(zI-(\Sigma + \varepsilon I))^{-1}dz - \ointctrclockwise_{\gamma} \log(z)(zI-\Sigma)^{-1}dz\right),
	\]
	where $\gamma_\varepsilon$ must enclose $\gamma$ by positivity of $\varepsilon$. Then provided that the eigenvalues of $\Sigma$ are bounded away from zero, $\gamma_\varepsilon$ can be chosen so as not to cross the imaginary axis and the previous display simplifies to 
	\begin{eqnarray}\label{cauchy2}
		\log(\Sigma + \varepsilon I) - \log(\Sigma)&=&\frac{1}{2\pi i}\ointctrclockwise_{\gamma_\varepsilon} \log(z)\{(z I-(\Sigma+\varepsilon I))^{-1} - (z I-\Sigma)^{-1}\}dz \\
		\nonumber &=& \frac{\varepsilon}{2\pi i}\ointctrclockwise_{\gamma_\varepsilon} \log(z)(z I-(\Sigma + \varepsilon I))^{-1}(z I-\Sigma)^{-1}dz
	\end{eqnarray}
	by Cauchy's theorem, where we have used that ${A^{-1}-B^{-1}=A^{-1}(B-A)B^{-1}}$ for invertible matrices $A$ and $B$, where $B-A=\varepsilon I$. Let $\Sigma=O \Lambda O^\T$ be the spectral decomposition of $\Sigma$, where $O$ have orthonormal columns $o_1,\ldots,o_p$ and $\Lambda=\text{diag}\{\lambda_1,\ldots,\lambda_p\}$. Then $(z I-\Sigma)^{-1}=O(zI - \Lambda)^{-1}O^\T$ and similarly for the expression involving $\Sigma+\varepsilon I$. It follows that the $(j,k)$th entry of the difference in log transformations is
	\begin{equation}
		[\log(\Sigma + \varepsilon I) - \log(\Sigma)]_{j,k} = \varepsilon\sum_{r,v} \Bigl(\frac{1}{2\pi i}\ointctrclockwise_{\gamma}\frac{\log(z)}{(z-(\lambda_{r}+\varepsilon))(z-\lambda_{v})} dz\Bigr)o_{jr}o_{kv}\sum_{\ell,s}o_{\ell r}o_{sv},
	\end{equation}
	which is equation \eqref{eqApprox}.
	\section{Estimation under the sparse $\Sigma_{ltu}$ parametrization}\label{secEstimation}
	
	\subsection{Construction of estimator}
	
	Recall the notation $ \Sigma_{ltu}= T \Omega T $, $ L = \log(T) = -\log(I-B) $ and $ D = \log(\Omega) $. We are primarily interested in situations where $ L $ and $ D $ are sparse. The guarantee that is typically sought for high-dimensional covariance estimators is consistency in the spectral norm under a notional double-asymptotic regime in dimension $p=p(n)$ and sample size $n$. Different approaches and asymptotic regimes might be considered, giving for instance, faster rates of convergence with slower permissible scaling of $p$ with $n$, or vice versa. Here we show one possible estimator and derive its convergence rates in spectral norm, under the scaling $ \log p / n \rightarrow 0 $. The theoretical properties are detailed in section \ref{sec:theory_estim} and proved in section \ref{secProofsEst}.
	
	The broad scheme involves constructing pilot estimators of the relevant quantities which have an elementwise consistency property, before exploiting sparsity on the transformed scale to obtain guarantees in the stronger norm. Suppose that $ \tilde{T} $ and $ \tilde{\Omega} $ are estimators of $ T $ and $ \Omega $ that have exploited sparsity on the transformed scale, and have been shown to be consistent in spectral norm. A natural estimator of $ \Sigma $ is then $ \tilde{\Sigma} = \tilde{T} \tilde{\Omega} \tilde{T}^{\T} $, which is also consistent in spectral norm. 
	
	In order to construct $ \tilde{T} $ and $ \tilde{\Omega} $, pilot estimators $ \hat{T} = (I-\hat{B})^{-1} $ and $ \hat{\Omega} $ are needed that are consistent in an elementwise sense. From these, let $ \hat{L} = -\log(I - \hat{B}) $ and $ \hat{D} = \log(\hat{\Omega})$. The simplest way to exploit sparsity of $ L $ and $ D $ is to use a thresholding operator \citep{BickelCovarThresh}, which sets entries of $ \hat{L} $ and $ \hat{D} $  to zero if their absolute values are below a specified threshold. The spectral-norm consistent estimators $ \tilde{T} $ and $ \tilde{\Omega} $ are then obtained by defining $ \tilde{T} :=  \exp(\tilde{L}) $ and $ \tilde{\Omega} := \exp(\tilde{D}) $, where $\tilde{L}$ and  $\tilde{D}$ are the thresholded versions of $\hat{L}$ and $\hat D$.
	
	To construct elementwise-consistent estimators $ \hat{B} $ and $ \hat{\Omega} $, note that \eqref{eq:AMPDrton} from the main text implies,
	\begin{align*}
		(I-B) X \sim N(0, \Omega).
	\end{align*}
	For a chain component $ c $, let $ \text{pa}(c) $ denote the set of parent nodes of $ c $.
	Then,
	\begin{equation}	\label{eq:factorisation}
		X_{c}|X_{\text{pa}(c)} = N(B_{c}X_{\text{pa}(c)}, \Omega_{c}).
	\end{equation}
	The factorization of joint density implies factorization of the parameter space \citep[see also][]{DrtonEichlerMLE}. As a result, we can estimate $ B_{c} $ and $ \Omega_{c} $ separately for each chain component.  From now on we omit the subscript $ c $ to simplify the notation.
	Equation \eqref{eq:factorisation} suggests estimating $ B $ by regressing each node on its parents, which yields an elementwise-consistent estimator  $ \hat{B} $. The estimator $ \hat{\Omega} $ can be obtained as a sample covariance matrix of residuals for all nodes in a given chain component (see equation \eqref{eq:factorisation}). For this, we can use estimates of regression coefficients $ \hat{B} $ or alternatively, a version $ \tilde{B} $ that exploits any sparsity on the transformed scale. Both result in an elementwise-consistent estimator of $ \Omega $, although $ \tilde{B} $ offers some advantage in high-dimensional settings.

	\subsection{Theoretical guarantees}
	\label{sec:theory_estim}
	
	For an $ m \times m $ matrix $ M $, let $ \| M \|_{\max} = \max_{i,j} | M_{i,j} | $, where $ M_{i,j} $ denotes an entry $ (i, j) $ of $ M $, and $ \| M \|_{2} = \sup_{\| w \|_{2} = 1} \| M w \|_{2} $. The largest eigenvalue of $ M $ is denoted by $ \lambda_{\max}(M) $. The size of a random vector $ X $ is denoted by $ | X | $. The set of parent nodes of node $ i $ and chain component $ c $ are denoted, respectively, by $ \text{pa}(i) $ and $ \text{pa}(c) $. Let $ \hat{\Sigma} $ be a sample covariance matrix of $ X $. For two sets of indices, $ s_1$,  $ s_2 \subseteq [p] $, let $ \hat{\Sigma}_{s_1, s_2} $ be the matrix obtained by selecting rows $s_1$ and columns $s_2$ of $ \hat{\Sigma} $.
	
	The results presented in this section are valid under a weaker assumption of sub-Gaussian rather than Gaussian distributions.
	\begin{condition}
		\label{ass:normal}
		For every chain component $ c $, 
		$
		X_{c}|X_{\text{pa}(c)}  $ is sub-Gaussian with a variance proxy $ \sigma_{\varepsilon}^{2} $.
	\end{condition}
	In addition, we assume that the covariance matrix $ \Sigma $ of $ X $ satisfies conditions \ref{ass:max_entry} and  \ref{ass:min_eigval}.

	\begin{condition}
		\label{ass:max_entry}
		The quantities $ \| \Sigma \|_{\max} $, $ \| \Sigma^{-1} \|_{\max} $, $ \| L \|_{2} $ and $ \| \Omega \|_{2} $ are bounded as $ n, p \rightarrow \infty $.
	\end{condition}
	
	\begin{condition}
		\label{ass:min_eigval}
		The sequence of smallest eigenvalues of $ 
		\hat{\Sigma} $ is bounded away from zero as $ p \rightarrow \infty $.
	\end{condition}
	
	Equation \ref{eq:factorisation} suggests estimating the $ i$th row of $ B $, denoted by $ \beta^{i} $, by regressing $ X_{i} $ on $ X_{\text{pa}(i)} $. Lemma \ref{th:beta_consistency} establishes elementwise consistency of the resulting estimator, $ \hat{\beta}^{i} $, which implies the consistency of $ \hat{B} = (\hat{\beta}^{1}, \ldots, \hat{\beta}^{p})^{\T} $. 
	\begin{lemma}
		\label{th:beta_consistency}
		Let $ X_{j} = X_{\text{pa}(j)}\beta + \varepsilon $, where $ X_{j} \in \mathbb{R}^n $, $ X_{\text{pa}(j)} \in \mathbb{R}^{n \times |\text{pa}(j)|} $, $ \beta \in \mathbb{R}^{|\text{pa}(j)|} $ and $ \varepsilon = (\varepsilon_1, \ldots, \varepsilon_{n}) $ is sub-Gaussian with zero mean and  variance proxy $ \sigma_{\varepsilon}^2 $. Then,
		\begin{align*}
			\max_{j \in [p]} \| \hat{\beta}^{j} - \beta^j \|_{\max} = O_p ( (\log p / n)^{1/2} ).
		\end{align*}
	\end{lemma}

	We now seek an estimator $ \hat{L}$ of $ L $ that inherits the elementwise consistency of $ \hat{B} $. As discussed in Section 5, the element $ (i,j) $ of the matrix logarithm of $ T $ corresponds to the effects of node $ i $ on node $ j $ along all directed paths connecting the two nodes. It is often sensible to assume that there is some length, say $l^{*}$, such that effects along paths of longer length are negligible when weighted inversely by the path length. This leads to Condition \ref{ass:decaying_paths:maintext} in the main text.
	
	Since $ L = -\log(I - B) $,
	a natural way of exploiting Condition \ref{ass:decaying_paths:maintext} is to approximate the matrix logarithm by a truncated power expansion of order $ l^{*} $. Specifically, for a matrix $ A $, define a truncated matrix logarithm of the $l$th order as
	\begin{align*}
		\log_{|l}(A) := \sum_{k = 1}^{l} (-1)^{k+1} \frac{(A-I)^{k}}{k}
	\end{align*}
	and let $ \hat{L} = - \log_{|l^{*}}(I- \hat{B}) $. Lemma \ref{th:lmax} establishes elementwise consistency of $ \hat{L} $ under Condition \ref{ass:decaying_paths:maintext}.

	\begin{lemma}
		\label{th:lmax}
		Let $ \hat{B} $ be an estimator of $ B $ such that  $ \| \hat{B} - B \|_{\max} = O_{p}((n^{-1} \log p )^{1/2})$.  Assume that Condition \ref{ass:decaying_paths:maintext} holds. 
		Then,
		\begin{equation}\label{eq:max_bound_L}
			\| \hat{L} - L \|_{\max} = O_{p} \bigl( ( n^{-1} \log p)^{1/2} \bigr).
		\end{equation}
	\end{lemma}
	
	Simulations presented in Section \ref{sec:th_simuls} suggest that the assumptions of Lemma \ref{th:lmax} are not necessary for bound \eqref{eq:max_bound_L} to hold. In particular, the rate of convergence in \eqref{eq:max_bound_L} is valid also for $ \hat{L} = -\log (I- \hat{B})$ and in the absence of Condition \ref{ass:decaying_paths:maintext}.
	
	We now use the elementwise-consistent estimator $ \hat{L} $ to construct an estimator $\tilde{L}$ of $ L $ consistent in the spectral norm. The key assumption to achieve consistency in a high-dimensional regime is that of sparsity. Specifically, assume that matrix $ L $ belongs to a sparse class of matrices, as stated in Condition \ref{ass:app_2:maintext}, which generalizes the notion of sparsity used in most of the main paper by allowing approximate zeros. Thus,
	\begin{align*}
		L = \log(T) \in  \biggl\{ L \in \lts(p): \max_{i} \sum_{j=1}^{p} |L_{ij}|^{q_{l}} = s_{l}(p)  \biggr\},  
	\end{align*}
	where $ 0 \leq q_{l} \leq 1 $ and $ s_{l}(p)/p \rightarrow 0 $.
	The estimator $ \tilde{L} $ is obtained by elementwise thresholding of $ \hat{L} $. In particular, for a $ p \times p $ matrix $ A $, an elementwise thresholding operator $ \mathcal{T}(A) $, introduced by \cite{BickelCovarThresh}, has the form,
	\begin{equation}
		\mathcal{T}(A)_{ij} = \mathcal{T}(A_{ij}) = A_{ij} \ind \{|A_{ij}| > \tau \}.
		\label{eq:thresholding_operator}
	\end{equation}
	Thus, $ \tilde{L} $ has the form,
	\begin{align}
		\label{eq:thresholdingL}
		\tilde{L} &= \mathcal{T}(\hat{L}), \quad \mathcal{T}(\hat{L})_{ij} =  \hat{L}_{ij} \: \ind \{|\hat{L}_{ij}| > \tau_l \}. 
	\end{align}
	Under Condition \ref{ass:app_2:maintext}, the following result follows from Theorem 1 in \cite{BickelCovarThresh}.
	
	\begin{corollary}
		\label{th:consost_threshold}
		Suppose that $ L \in  \mathcal{U}(q_{l}, s_{l}(p)) $ and $ \| \hat{L}  - L \|_{\max} = O_{p}( r_{n,p} ) $. Let $ \tau_l \asymp r_{n,p} $ in \eqref{eq:thresholdingL}. Then, 
		$
		\| \tilde{L} - L \|_{2} = O_{p} \bigl(s_{l}(p) r_{n,p}^{1-q_{l}} \bigr) 
		$ as $ n $, $ p \rightarrow \infty $.
	\end{corollary}
	
	The consistency of $ \tilde{L} $ is sufficient to obtain an spectral-norm consistent estimator of $ T $, as shown in Lemma \ref{th:est_of_b}. 
	\begin{lemma}
		\label{th:est_of_b}
		Let $ \tilde{B} = I - \exp(-\tilde{L}^{\T}) $ and $ \tilde{T}  = \exp(\tilde{L}) $. Then,
		\begin{align*}
			\| B - \tilde{B} \|_{2} &\leq \exp(\lambda_{\max}(L^{\T}L)) \|  L - \tilde{L} \|_{2} \exp(\| L - \tilde{L} \|_{2}), \\  
			\| T - \tilde{T} \|_{2} &\leq \exp(\lambda_{\max}(L^{\T}L)) \|  L - \tilde{L} \|_{2} \exp(\| L - \tilde{L} \|_{2}).  
		\end{align*}
	\end{lemma}
	
	A direct consequence of Lemma \ref{th:est_of_b} is that thresholding in the transformed domain yields an  $ \ell_{2} $-norm consistent estimator of regression coefficients $ \beta $, which constitute the rows of $ B $. 
	\begin{corollary}
		% \label{th:est_of_b}
		Let $ \tilde{\beta}^{i}$ and $ \beta^{i} $ denote the $i$th row of $ \tilde{B} $ and $ B $ respectively. Then,
		\begin{align*}
			\| \beta^{i} - \tilde{\beta}^{i} \|_{2} &\leq \exp(\lambda_{\max}(L^{\T}L)) \|  L - \tilde{L} \|_{2} \exp(\| L - \tilde{L} \|_{2}).
		\end{align*}
	\end{corollary}
	
	Lemma \ref{th:est_of_b}, together with Corollary \ref{th:consost_threshold} and Lemma \ref{th:lmax} imply that
	\begin{equation}
		\| T - \tilde{T} \|_{2} = O_{p} \biggl( s_{l}(p)  \left( n^{-1} \log p \right)^{(1-q_{l})/2} \biggr).
		\label{eq:T_convergnece_rate}
	\end{equation}
	Since $ \Omega $ is positive-definite, a spectral-norm consistent estimator $ \tilde{\Omega} $ can be obtained using the approaches of \cite{Battey2019} or \cite{Zwiernik2025}.
	The former requires an elementwise-consistent estimator of $ \Omega $. Given an estimator $ \bar{B} $ of $ B $, let  $ \hat{\Omega} $ denote a sample covariance matrix of residuals for a chain component $ c $.
	Lemma \ref{th:max_consistency_omega} establishes elementwise consistency of $ \hat{\Omega} $.
	
	\begin{lemma}
		\label{th:max_consistency_omega}
		Assume Condition \ref{ass:max_entry} holds. Let $ \bar{B} $ denote an estimator of $ B $. 
		\begin{enumerate}
			\item If $ \| B - \bar{B} \|_{2} = O_p(r_{\beta}(n, p)) $ then,
			\begin{equation}
				\| \hat{\Omega} - \Omega \|_{\max}
				= 
				O_{p} \left( r^{2}_{\beta}(n,p) \sqrt{\frac{\log \rho_{pa}}{n}} + \sqrt{\frac{\log \rho}{n}} \right), 
				\label{eq:omega_max_1}
			\end{equation} 
			\item If $ \| B - \bar{B} \|_{\max} = O_p(r_{\beta}(n, p)) $ then,
			\begin{equation}
				\| \hat{\Omega} - \Omega \|_{\max} = O_{p} \biggl( \rho_{pa}^{2}  r_{\beta}^{2}(n, \rho) \sqrt{\frac{\log \rho_{pa} }{n}}   \biggr),
				\label{eq:omega_max_2}
			\end{equation}
		\end{enumerate}
		where $ \rho_{pa} = \max_{i \in c}  |X_{\text{pa}(i)}| $, $ \rho = |X_{c} | $ and $ c $ denotes a chain component.
	\end{lemma}
	
	If the size of chain components grows at the same rate as $ p$, $ \rho \asymp p $ and $ \rho_{\text{pa}} \asymp p $, under conditions of Lemma \ref{th:beta_consistency}, the rate of convergence in \eqref{eq:omega_max_1} is more advantageous than in \eqref{eq:omega_max_2}. This suggests using $ \tilde{B} $ rather than $ \hat{B} $ to obtain the residual covariance matrix $ \hat{\Omega} $, which yields,
	\begin{align*}
		\| \hat{\Omega} - \Omega \|_{\max}
		% &= 
		% O_{p} \left( r_{\beta}(n,p)^{2} \right) \\
		&= O_{p} \left( s_{l}(p)^{2} (n^{-1} \log p)^{3/2-q_l} + (n^{-1} \log p )^{1/2} \right) \\
		&= O_{p} \left( s_{l}(p)^{2} (n^{-1} \log p)^{3/2 - q_l} \right),
	\end{align*} 
	where we have assumed that the first term dominates the convergence rate. Let 
	\begin{align}
		\label{eq:thresholdingO}
		\tilde{\Omega} = \exp(\mathcal{T}(\log(\hat{\Omega}))), \quad
		\mathcal{T}(\hat{\Omega})_{ij} =  \hat{\Omega}_{ij} \: \ind \{ | \hat{ \Omega}_{ij}| > \tau_{ \omega} \}. 
	\end{align}
	Then, under Condition \ref{ass:app_2:maintext}, for $ \tau_{\omega} \asymp s_{l}(p)^{2} (n^{-1} \log p)^{1/2}  $, Theorem 2 of \cite{Battey2019} implies,
	\begin{equation}
		\| \Omega - \tilde{\Omega} \|_{2} = O_{p}\biggl( s_{\omega}(p) s_{l}(p)^{2-2 q_{\omega}} (n^{-1} \log p)^{(3/2-q_{l})(1-q_{\omega})} \biggr),
		\label{eq:omega_convergence}
	\end{equation}
	where $ s_{\omega}(p)/p \rightarrow 0 $ and $ 0 \leq q_{\omega} \leq 1 $.
	
	Given estimators $ \tilde{T} $ and $ \tilde{\Omega} $, the estimator of the covariance matrix $ \Sigma $ can be obtained by $ \tilde{\Sigma} = \tilde{T} \tilde{\Omega} \tilde{T}^{\T} $. Proposition \ref{th:sigma_consistent:maintext} in the paper establishes the spectral-norm consistency of $ \tilde{\Sigma} $ as $ p $, $n \rightarrow \infty $, provided that $ \log p / n \rightarrow 0 $.
	
	In the absence of a causal ordering of variables, a natural pilot estimator of $ T $ is a triangular matrix obtained by the LDL decomposition of $ \hat{\Sigma} $. Since $ T = (I - B)^{-1} $, we can use the proof strategy used to establish the elementwise consistency of the matrix logarithm above. Specifically, under
	Condition \ref{ass:decaying_paths_inv:maintext}, an elementwise consistency of a truncated matrix inverse, defined for a matrix $ A $ as
	$
	A^{-1}_{|l} := \sum_{k = 1}^{l} (-1)^{k+1} A^{k}
	$
	is established by Lemma \ref{th:tinv_max} below.

	\begin{condition}
		There exists $ l^{*} \in \mathbb{N} $, such that for any pair of nodes $ (i,j) $, $ j < i  $, $ \sum_{l = l^{*}+1}^{p} \delta_{i|j}(l) = C (\log p /n)^{\varphi_2 / (2(\varphi_2+1))} $ for $ \varphi_2 > 0 $, where $ \delta_{i|j}(l) $ denotes the sum of effects of node $ j $ on node $ i $ along all paths of length $ l $.
		\label{ass:decaying_paths_inv:maintext}
	\end{condition}
	\begin{lemma}
		\label{th:tinv_max}
		Let $ \hat{B} $ be an estimator of $ B $ such that  $ \| \hat{B} - B \|_{\max} = O_{p}((n^{-1} \log p)^{1/2})$.  Assume that Condition \ref{ass:decaying_paths_inv:maintext} holds and let $ \bar{T} = (I - \hat{B})_{|l}^{-1} $. 
		Then,
		\begin{equation}\label{eq:max_bound_inv}
			\| \bar{T} - T \|_{\max} = O_{p} \bigl( ( n^{-1} \log p)^{1/2} \bigr).
		\end{equation}
	\end{lemma}
	Simulations presented in Figure 3 suggest that Condition \ref{th:tinv_max} and the restriction to a truncated inverse are not necessary for Lemma \ref{th:tinv_max} to hold. This suggests that the rate $ (\log p / n)^{1/2} $ is also valid for a pilot estimator $ \hat{T} = (I- \hat{B})^{-1} $ , which corresponds to the triangular matrix obtained by the LDL decomposition of the sample covariance matrix.

	\subsection{Simulations}
	\label{sec:th_simuls}
	
	Lemma \ref{th:lmax} shows that Condition \ref{ass:decaying_paths:maintext} is sufficient to establish elementwise convergence of $ \hat{L} $, where $ \hat{L} = -\log_{|l^{*}}(I - \hat{B})$. Using simulations, we now compare the rate of convergence of $ \| \bar{L} - L \|_{\max} $, $ \bar{L} = -\log(I - \hat{B})$, and $ \| \hat{B} - B \|_{\max} $ in the absence of Condition \ref{ass:decaying_paths:maintext}. The results, presented in Figure \ref{fig:convergence} (a), suggest that Condition \ref{ass:decaying_paths:maintext} is not necessary for the equation \eqref{eq:max_bound_L} to hold. In addition, Lemma \ref{th:lmax} holds when $ \hat{L} $ is replaced by $ \bar{L} $. An analogous analysis is performed to assess the necessity of Condition \ref{ass:decaying_paths_inv:maintext} for the validity of Lemma \ref{th:tinv_max} in Figure \ref{fig:convergence} (b), which compares the rate of convergence of $ \| \bar{T} - T \|_{\max} $, $ \bar{T} = (I - \hat{B})^{-1} $, and $ \| \hat{B} - B \|_{\max} $.
	
	For each simulation, $ \Sigma = O \Lambda O^{\T} $, where $ O $ is an orthogonal matrix obtained by a QR decomposition of a $ p \times p $ matrix with iid standard normal entries, and elements of $ \Lambda $ are drawn from a gamma distribution with a shape parameter $ k $ and a scale parameter $ v $. 
	
	\begin{figure}[h!] 
		%\vspace{-0.2cm}
		\centering
		\begin{subfigure}[b]{0.49\textwidth}
			\centering
			\includegraphics[trim={0cm 0cm 0cm 0cm},clip,width = 1.0\textwidth]{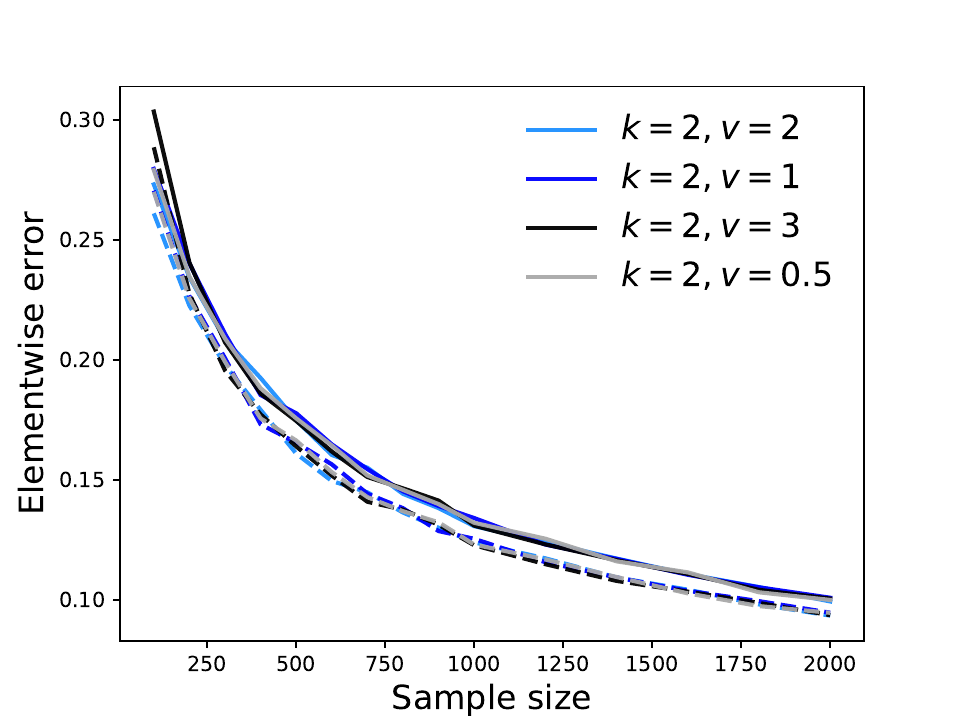} 
			\caption{$ \| \bar{L} - L \|_{\max} $}
		\end{subfigure}
		%\hspace{-2cm}
		\begin{subfigure}[b]{0.49\textwidth}
			\centering
			\includegraphics[trim={0cm 0cm 0cm 0cm},clip,width = 1.0\textwidth]{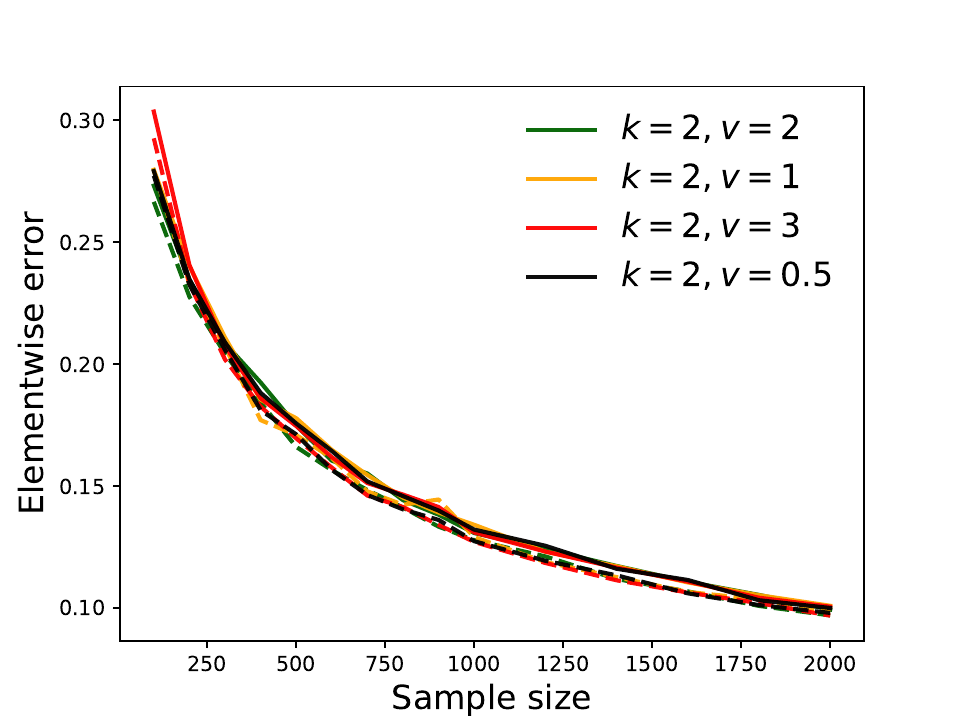}
			\caption{$ \| \bar{T} - T \|_{\max} $}
		\end{subfigure} 
		\caption{Average elementwise errors $ \| \hat{B} - B \|_{\max} $ (solid lines in both plots), $ \| \bar{L} - L \|_{\max} $ (dotted lines, left plot) and $ \| \bar{T} - T \|_{\max} $ (dotted lines, right plot) for $ 100 $ simulations for each $ n $, with $ p = n/10 $.
		}
		\label{fig:convergence}
	\end{figure}
	
	\subsection{Proofs of results in Appendix \ref{sec:theory_estim}}\label{secProofsEst}

	\medskip
	\noindent\textbf{Proof of Lemma \ref{th:beta_consistency}}.	    The estimation error for the $i$th element of $ \hat{\beta}^{j} $ has the following form,
	\begin{align*}    
		\Delta_i^j \equiv \hat{\beta}^j_i - \beta^j_i = e_i^T  (X_{\text{pa}(j)}^T X_{\text{pa}(j)})^{-1} X_{\text{pa}(j)}^T \varepsilon.
	\end{align*}
	Since $ \varepsilon $ is sub-Gaussian, $ \Delta_i^j $ is a linear combination of sub-Gaussian random variables.	Thus, $ \Delta_i^j $ is sub-Gaussian with a variance proxy $ \|e_i^T  (X_{\text{pa}(j)}^T X_{\text{pa}(j)})^{-1} X_{\text{pa}(j)}^T  \|_2^2 \sigma_{\varepsilon} $. Now,
	\begin{align*}
		\|e_i^T  (X_{\text{pa}(j)}^T X_{\text{pa}(j)})^{-1} X_{\text{pa}(j)}^T   \|_2^2 = e_i^T  (X_{\text{pa}(j)}^T X_{\text{pa}(j)})^{-1} e_{i} = \widehat{\text{var}}(\text{pa}(j))^{-1}_{ii} /  (n-1).
	\end{align*} 
	Under Condition \ref{ass:min_eigval} the maximum eigenvalue of $ \widehat{\text{var}}(\text{pa}(j))^{-1} $ is upper-bounded. By the definition of the operator norm, for any column $ \nu $ of $ \widehat{\text{var}}(\text{pa}(j))^{-1} $ we have $ \| \nu \|_{2} \leq M $. By the Cauchy-Schwartz inequality $ | \nu_k | = | \langle e_k, \nu \rangle | \leq  \| \nu \|_{2} \leq M$, where $ \nu_k $ is the $k$th entry of $ \nu $ and $ e_k $ is a canonical basis vector with $ k$th element equal to one. Thus, $ \max_{i, j \in [p] } \widehat{\text{var}}(\text{pa}(j))^{-1}_{ii} \leq M  $ for some constant $ M $. Hence,
	$ \mathbb{E} \exp(\Delta^j_i ) \leq \exp( \widehat{\text{var}}(\text{pa}(j))^{-1}_{ii} \sigma_{\varepsilon}^{2} / 2 (n-1)  ) \leq \exp(M \sigma_{\varepsilon}^{2} / 2n  ) $ for every $ i$, $ j \in [p] $.
	Then, 
	\begin{align*}
		\mathbb{P}(|\Delta^j_i| \geq t) \leq 2 \exp \left( - \frac{n t^{2}}{2 M \sigma_{\varepsilon}^{2}} \right).
	\end{align*}
	On setting $ t = \sqrt{\frac{ 2 M \sigma^{2}_{\varepsilon} \log(2 p^{2}/\delta) }{n}} $ we obtain
	\begin{align*}
		\mathbb{P} \left( |\Delta_i| \geq \sqrt{\frac{ 2 M \sigma_{\varepsilon}^{2} \log(2 p^2/\delta) }{n}} \right) \leq 2 \exp \left( - \frac{n t^{2}}{2 M \sigma_{\varepsilon}^{2}} \right) = \frac{\delta}{p^2}.
	\end{align*}  
	On applying the union bound,
	\begin{align*}
		\mathbb{P} ( \max_{j \in [p]} \| \hat{\beta}^{j} - \beta^j \|_{\infty} \geq t) \leq \sum_{j=1}^{p} \sum_{i=1}^{p} \mathbb{P}(|\Delta_i^j| \geq t) \leq \delta.
	\end{align*}
	Hence, for any $ \delta \in (0, 1) $,
	\begin{align*}
		\mathbb{P} \left( \max_{j \in [p]} \| \hat{\beta}^j - \beta^j \|_{\infty} \geq \sqrt{\frac{ 2 M \sigma_{\varepsilon}^{2} \log(2 p^2/\delta) }{n}} \right) \leq 1 - \delta.
	\end{align*}
	\qed

	\medskip
	\noindent\textbf{Proof of Lemma \ref{th:lmax}}. Let $ r_{\beta}(n,p) \triangleq (n^{-1} \log p)^{\frac{\kappa}{2(\kappa + 1)}} $. Recall that $ \delta_{i|j}(l) $ denotes the total effect of node $ j $ on node $ i $ along all paths of length $ l $. Specifically, let $ S(j,i,l) $ be a set of subsets of indices $ \{j+1, \ldots, i-1 \}$ of length $ l - 2 $. For each $ s \in S(j,i,l) $, denote the corresponding indices by $ s_{1} < s_{2} < \ldots < s_{l-2} $. Then, for each $ l $, $ \delta_{k|i}(l) $ has the form,
	\begin{align*}
		\delta_{k|i}(l) = \sum_{s \in S(j,i,l)} \gamma_{i| s_{l-1}} \gamma_{j|s_{1}} \prod_{v=1}^{l-2} \gamma_{s_{v+1}|s_{v}},
	\end{align*}
	where $ \gamma_{t,v} $ denotes a regression coefficient of $ X_{v} $ from regression of $ X_{t} $ on $ X_{1}, \ldots, X_{v} $. 
	
	Now consider,
	\begin{align*}
		\Delta(j,i,s) = \bigl| \gamma_{i, s_{l}} \gamma_{j|s_{1}} \prod_{v=1}^{l-2} \gamma_{s_{r+1}|s_{r}} - \hat{\gamma}_{i, s_{l}} \hat{\gamma}_{j|s_{1}} \prod_{v=1}^{l-2} \hat{\gamma}_{s_{v+1}|s_{v}} \bigr|.
	\end{align*}
	This expression has the form
	$
	\bigl| \prod_{v=1}^{l} a_{v} - \prod_{t = 1}^{l} c_{l} \bigr|
	$
	with $ | a_{v} - c_{v} | = O_{p}(r_{\beta}(p, n))$. Let $ A_{v} = \prod_{i=1}^{v} a_{v} $ and $ C_{v} = \prod_{i=1}^{v} c_{v} $. Then, by the triangle inequality,
	\begin{align*}
		|A_{v+1} - C_{v+1}| &= |a_{v+1} A_v - c_{v+1} C_{v} | \\
		&\leq | A_{v} | |a_{v+1} - c_{v+1} | + | c_{v+1} | | A_{v} - C_{v} |. 
	\end{align*}
	Applying the inequality recursively we obtain $ |A_{v+1} - C_{v+1}| = O_{p}( (v+1) r_{\beta}(n,p) ). $ Thus, $ \Delta(j,i,s) = O_{p}(l r_{\beta}(n,p)) $, which represents an estimation error for a single path of length $ l $ connecting nodes $ i $ and $ j $. By the binomial theorem there are $ 2^{i-j-2} $ directed paths between $ i $ and $ j $.
	Then,
	\begin{align*}
		\| \log_{|l^{*}}(I - B) - \log_{|l^{*}}(I - \hat{B}) \|_{\max} = O_{p}(2^{l^{*}-2} r_{\beta}(n,p)).
	\end{align*}
	As a result,
	\begin{align*}
		\| \log(I - B) - \log_{|l^{*}}(I - \hat{B}) \|_{\max} = O_{p} \bigl(  r_{\beta}(n,p) \bigr) + \max_{i,j} \sum_{l=l^{*}+1}^{p} \frac{1}{l} | \Delta(j, i,s) |.
	\end{align*}
	By Condition \ref{ass:decaying_paths:maintext},
	\begin{align*}
		\max_{i,j} \sum_{l=l^{*}+1}^{p} \frac{1}{l} | \Delta(j,i,s) | = O_{p} \bigl( (n^{-1} \log p)^{\varphi} \bigr).
	\end{align*}
	The result follows since $ \varphi \geq 1/2. $
	\qed
	
	\medskip
	\noindent\textbf{Proof of Corollary \ref{th:consost_threshold}}.	Except for the change in the object being thresholded, the proof is that of Theorem 1 in \cite{BickelCovarThresh}. \qed
	
	\medskip
	\noindent\textbf{Proof of Lemma \ref{th:est_of_b}}.	Consider $ B - \tilde{B} = \exp(-L) - \exp(-\mathcal{T}(\hat{L})) $. By Corollary 6.2.32 in \cite{HornJohsonTMA},
	\begin{align*}
		\| B - \tilde{B} \|_{2} = \| \exp(-L) - \exp(-\mathcal{T}(\hat{L})) \|_{2} \leq \|  L - \mathcal{T}(\hat{L}) \|_{2} \exp(\| L \|_{2}) \exp(\| L - \mathcal{T}(\hat{L}) \|_{2}).
	\end{align*}
	By the definition of the spectral norm, $ \| L \|_{2} = \lambda_{\max}(L^{\T} L) $.	Similarly, since $ T = \exp(L) $,
	\begin{align*}
		\| T - \tilde{T} \|_{2} &= \| \exp(L) - \exp(\mathcal{T}(\hat{L})) \|_{2} \\
		&\leq \| L - \mathcal{T}(\hat{L}) \|_{2} \exp(\| L \|_{2}) \exp( \| L - \mathcal{T}(\hat{L}) \|_{2} ),
	\end{align*}
	where the inequality follows from Corollary 6.2.32 in \cite{HornJohsonTMA}. \qed

	\medskip
	\noindent \textbf{Proof of Lemma \ref{th:max_consistency_omega}}.	For a chain component $ c $, let $ \mathcal{E} = X_{c} - B X_{\text{pa}(c)} $ and  $ \hat{\mathcal{E}} = X_{c} - \hat{B} X_{\text{pa}(c)} $. Then, $ \hat{\Omega} = \widehat{\text{var}}(\mathcal{E}) $ and
	\begin{align*}
		\text{var}(\mathcal{E}) = \text{var}(X^{c}) - \text{var}(B X_{\text{pa}(c)}).
	\end{align*}
	Using the triangle inequality and the equality above,
	\begin{align}
		&\| \widehat{\text{var}}(\mathcal{E}) - \text{var}(\mathcal{E}) \|_{\text{max}} 
		= \| (\widehat{\text{var}}(X_c) - \widehat{\text{var}}(\hat{B} X_{\text{pa}(c)})) - (\text{var}(X_c) - \text{var}(B X_{\text{pa}(c)})) \|_{\text{max}} \\
		&=  \| \widehat{\text{var}}(X_c) - \text{var}(X_{c}) \|_{\text{max}} + \| \widehat{\text{var}}((\hat{B} - B) X_{\text{pa}(c)})  \|_{\text{max}} + 
		\| \widehat{\text{var}}(B X_{\text{pa}(c)}) - \text{var}({B} X_{\text{pa}(c)})  \|_{\text{max}}. 
		\label{eq:var_three_terms}
	\end{align}
	By Lemma A.3 in \cite{BickelLevinaRegEst}, $ \| \widehat{\text{var}}(X_c) - \text{var}(X_{c}) \|_{\text{max}} = O_{p}(\sqrt{\log \rho_{c} / n})$. Now consider the second term in \eqref{eq:var_three_terms} and let $ X_{\text{pa}(c)}^{i} $ denote the $i$th sample of $ X_{\text{pa}(c)} $. 
	\begin{align}
		\widehat{\text{var}}((\hat{B} - B) X_{\text{pa}(c)})  &= \frac{1}{n-1} \sum_{i=1}^{n} (B - \hat{B}) X^{i}_{\text{pa}(c)} X^{i}_{\text{pa}(c)} (B - \hat{B})^{\T} \\
		&= (B - \hat{B})  \left( \frac{1}{n-1} \sum_{i=1}^{n}  X^{i}_{\text{pa}(c)} X^{i}_{\text{pa}(c)} \right) (B - \hat{B})^{\T} \\
		&= (B - \hat{B}) \widehat{\text{var}}(X_{\text{pa}(c)}) (B - \hat{B})^{\T}.
		\label{eq:var_eq2}
	\end{align}
	Let $ [A]_{ij} $ denote an element $ (i,j) $ of matrix $ A $. Then, from equation \eqref{eq:var_eq2},
	\begin{align*}
		| [\widehat{\text{var}}(\hat{B} X_{\text{pa}(c)} - B X_{\text{pa}(c)})]_{ij} | &= | [(B - \hat{B}) \widehat{\text{var}}(X_{\text{pa}(c)}) (B - \hat{B})^{\T}]_{ij} | \\
		&= \biggl| \sum_{l=1}^{|pa|} \sum_{k=1}^{|pa|} (B - \hat{B})_{il} \widehat{\text{var}}(X_{\text{pa}(c)})_{lk}(B - \hat{B})_{jk} \biggr| \\
		&\leq \| \widehat{\text{var}}(X_{\text{pa}(c)}) \|_{\max}  \biggl| \sum_{l=1}^{|pa|} (B - \hat{B})_{il} \biggr| \biggl| \sum_{k=1}^{|pa|} (B - \hat{B})_{jk}  \biggr| \\
		&\leq \| \widehat{\text{var}}(X_{\text{pa}(c)}) \|_{\max} \Bigl( \sum_{l=1}^{|pa|} (B - \hat{B})_{il}^{2} \Bigr)^{1/2} \Bigl( \sum_{k=1}^{|pa|} (B - \hat{B})_{jk}^{2} \Bigr)^{1/2} \\
		&= \| \widehat{\text{var}}(X_{\text{pa}(c)}) \|_{\max} \| \hat{\Delta}_{i} \|_{2} \| \hat{\Delta}_{j} \|_{2} 
	\end{align*}
	where $ \hat{\Delta}_{i} $ denotes an $ i $th row of $ B - \hat{B} $.
	The elementwise consistency of the covariance matrix, together with Condition \ref{ass:max_entry}, the spectral-norm consistency of $ \hat{B} $ and the triangle inequality imply, 
	\begin{align*}
		\| \widehat{\text{var}}((\hat{B} - B) X_{\text{pa}(c)})  \|_{\text{max}} = O_{p}\biggl(r_{\beta}(n,p)^{2} +  r_{\beta}(n,p)^{2} \sqrt{\log \rho_{pa} / n} \biggr).
	\end{align*}
	
	To upper-bound the third term in equation \eqref{eq:var_three_terms} note that
	\begin{align*}
		\| \widehat{\text{var}}(B X_{\text{pa}(c)}) - \text{var}({B} X_{\text{pa}(c)})  \|_{\text{max}} = 
		\max_{ij} \bigl| [\widehat{\text{var}}(B X_{\text{pa}(c)}) - \text{var}({B} X_{\text{pa}(c)})]_{ij} \bigr|
	\end{align*}
	where $ B X_{pa(c)} $ is sub-Gaussian with zero mean. Thus, we can upper-bound this term using the elementwise consistency of the covariance matrix estimator, which yields
	\begin{align*}
		\| \widehat{\text{var}}(B X_{\text{pa}(c)}) - \text{var}({B} X_{\text{pa}(c)})  \|_{\text{max}} = O_{p}(\sqrt{\log \rho_{c} /n}).
	\end{align*}
	Overall, we obtain,
	\begin{align*}
		\| \widehat{\text{var}}(\mathcal{E}) - \text{var}(\mathcal{E}) \|_{\text{max}} = O_{p} \left( r_{\beta}(n,p)^{2} +  r_{\beta}(n,p)^{2} \sqrt{\log \rho_{pa} / n} + \sqrt{\log \rho_{c} / n} \right),
	\end{align*}
	which establishes the first claim in Lemma \ref{th:max_consistency_omega}. The proof for the second claim is identical, except for the upper bound for the second term in \eqref{eq:var_three_terms}, which we address now. The proof is similar to that of \cite{BickelLevinaRegEst}.
	By Cauchy-Schwartz inequality and the fact that $ \widehat{\text{var}}(X_{\text{pa}(c)})_{lk} \leq \widehat{\text{var}}(X_{\text{pa}(c)})_{ll} \widehat{\text{var}}(X_{\text{pa}(c)})_{kk} $,
	\begin{align*}
		| [\widehat{\text{var}}((\hat{B} - B) X_{\text{pa}(c)})]_{ij} | &= | [(B - \hat{B}) \widehat{\text{var}}(X_{\text{pa}(c)}) (B - \hat{B})^{\T}]_{ij} | \\
		&= | \sum_{l=1}^{|pa|} \sum_{k=1}^{|pa|} (B - \hat{B})_{il} \widehat{\text{var}}(X_{\text{pa}(c)})_{lk} (B - \hat{B})_{jk} | \\
		&\leq  \biggl( \sum_{l=1}^{|pa|} |(B - \hat{B})_{il} |\biggr)^{2} \widehat{\text{var}}(X_{\text{pa}(c)})_{ll} \\
		&\leq  \rho_{pa}^{2} \| \widehat{\text{var}}(X_{\text{pa}(c)}) \|_{\max}   \Bigl( \max_{il} | B_{il} - \hat{B}_{il} | \Bigr)^{2} 
	\end{align*}
	\qed
	
	\subsection{Proof of Proposition \ref{th:sigma_consistent:maintext}}
	
	Recall the inequality,
	\begin{align*}
		&\| A_1 A_2 A_3 - C_1 C_2 C_ 3 \|_{2} \\
		&\leq \sum_{j=1}^{3} \| A_j - C_j \|_{2} \prod_{k \neq j } \| C_{k} \|_{2} 
		+ \sum_{j=1}^{3} \| C_{j} \|_{2} \prod_{k \neq j } \| A_k - C_k \|_{2} + \prod_{j=1}^{3} \| A_j - C_j \|_{2}.
	\end{align*}
	Let $ A_1 = A_3^\T =\tilde{T} $, $ C_1 = C_3^{\T} = T $, $ A_{2} = \hat{\Omega} $ and $ C_{2} = \Omega^{-1} $. Let $ r_{t} $ and $ r_{\omega} $ denote the convergence rates of $ \tilde{T} $ and $ \hat{\Omega} $ respectively. In particular, $ \| \tilde{T} - T \|_{2} = O_{p}(r_{t}) $ and $ \| \hat{\Omega} - \Omega \|_{2} = O_{p}(r_{\omega}) $. Then,
	\[
	\| \hat{\Sigma} - \Sigma \|_{2} \leq 2  r_{t}  \| T \|_{2} \| \Omega \|_{2} +  r_{\omega} \| T \|_{2}^{2} + 2  r_{t} r_{\omega}  \| T \|_{2} +  r_{t}^{2} \| \Omega_{c} \|_{2}^{2} + r_{t}^{2} r_{\omega}.
	\]
	The result follows from equations \eqref{eq:T_convergnece_rate} and \eqref{eq:omega_convergence}.
	\qed

	\section{Simulation results for \S \ref{secApproxSparsity}}\label{appApproxSparsity}
	
	A notion of approximate sparsity that allows for slight departures from zero is, for any matrix $A$,
	\begin{equation}\label{eqSparsityMeas}
		s_\tau(A) = \sum_{i,j<i}\ind \, (|A_{ij}|>\tau).
	\end{equation}
	This replaces elements by 1 and 0 according to their values relative to $\tau$, and thus is more suitable than \eqref{eqSparseClass} for comparison across scales.
	
	For each of the four parametrizations of \eqref{eqRepara}, we explore the extent to which $L$ is sparser than $\Sigma^{-1}$ according to equation \eqref{eqSparsityMeas}, and the implications for estimation. For tables \ref{tabPD}--\ref{tabLTU}, random matrices $L$ of dimension $p=60$ were generated using the appropriate basis in equation \eqref{eqBasisExp} by randomly drawing $s^*/2$ entries of $\alpha$ from a uniform distribution on $[-4, -2]$, $s^*/2$ entries from a uniform distribution on $[2,4]$ and $m-s^*$ entries from a uniform distribution on $[-0.01, 0.01]$, where $m$ is the number of elements in the basis. The resulting matrix $L$ was converted to the relevant matrix space $\pd(p)$, $\so(p)$, $\lt(p)$ or $\ltu(p)$ by taking the matrix exponential. The positive diagonal entries needed to complete the specification for the $\Sigma_o$ and $\Sigma_{ltu}$ parametrizations were drawn from an exponential distribution of rate $\rho$. In producing the simulations of this section, we have used R functions to implement the LDL, Cholesky, and LU decompositions, mainly to avoid the complications arising from pivoting operations used in the corresponding implementations in Matlab. 
	
	The estimation error in non-trivial matrix norms is most relevant when the matrix object is a nuisance parameter, and the numerical results presented here are motivated by that setting. Since it is usually the precision matrix that is the nuisance parameter in procedures of multivariate analysis, rather than the covariance matrix, we focus on estimation of $\Sigma^{-1}$. 
	
	For each of 200 simulation replicates, $n=200$ $p$-dimensional random vectors were generated from a mean-zero normal distribution with covariance matrix as specified above. Three estimators of the precision matrix were compared in terms of their average estimation  errors in the spectral and Frobenius norms. 
	
	The simplest type of estimator exploiting sparsity sets entries of a preliminary estimate to zero if they are below a threshold $\tau$. For the four parametrizations of equation \eqref{eqRepara}, the simplest preliminary estimate is the matrix logarithm of the relevant sample quantity, constructed from the eigen-, Cholesky, or LDL decomposition of the sample covariance matrix. The matrix logarithm was computed using the algorithm of \citet{Higham2012}, whose implementation is part of Matlab's standard distribution and R's \texttt{expm} package. Let $\hat L_{\tau}$ denote the thresholded estimator on the logarithmic scale, so that an estimator of $\Sigma^{-1}$ under the $\Sigma_{pd}$ parametrization is $\exp(-\hat L_\tau)$ and the analogous quantities for the other three parametrizations are $\hat O_\tau = \exp(\hat L_\tau)\in\so(p)$, $\hat V_\tau = \exp(\hat L_\tau)\in\lt(p)$ and $\hat U_\tau = \exp(\hat L_\tau)\in\ltu(p)$, from which an estimator of $\Sigma^{-1}$ is constructed in the obvious way. A comparable estimator based on an assumption of sparsity directly on the inverse scale is $\hat\Sigma_{\tau}^{-1}$, the inverse sample covariance matrix thresholded at $\tau$. In all cases, the threshold $\tau=1$ was used as the level below which entries were set to zero, implying that $s_\tau(L)$ from equation \eqref{eqSparsityMeas} is $s^*$ by the simulation design. 
	The estimator $\hat \Sigma^{-1}_\tau$ typically violates positive definiteness, which may or may not be problematic, depending on context. The results for the three parametrizations are reported in Tables \ref{tabPD}--\ref{tabLTU}.
	
	For the $\Sigma_o$ parametrization, an additional step checked whether the matrix of orthonormal eigenvectors $\hat O$ of the sample covariance matrix was special orthogonal, and if not, converted it to special orthogonal by multiplying the first row of $\hat O$ by minus one. This step ensures that the matrix logarithm is skew-symmetric and real-valued. 
	
	Thresholding on the logarithmic scale was justified by \citet{Battey2019} under the $\Sigma_{pd}$ parametrization, and in Proposition \ref{th:sigma_consistent:maintext} under the $\Sigma_{ltu}$ parametrization.  We have not in these simulations attempted to optimize tuning constants, and it is likely that the results could be improved through a data-adaptive tuning, nevertheless, several of the results suggest a benefit from exploiting sparsity on the logarithmic scale as opposed to on the inverse scale. 
	
	\begin{table}
		\begin{center}
			\begin{tabular}{cccccc}
				&        &                                &  \multicolumn{2}{c}{Estimator $E$} \\
				\cline{4-5} 
				$s^*$  & $\frac{\|E-\Sigma^{-1}\|_\bullet}{\|\Sigma^{-1}\|_\bullet}$ & $s_\tau(\Sigma^{-1})$  & \;\;\; $\hat \Sigma^{-1}_{\tau}$\;\;\;   &  $\exp(-\hat L_\tau)$ \\
				\hline
				6   &  $\bullet = 2$ & 45.6 &  0.561  & 0.203 \\
				6   &  $\bullet = \text{F}$ & 45.6 & 0.503  & 0.187 \\
				10 & $\bullet = 2$ & 54.7 &  0.551  & 0.214 \\
				10  &  $\bullet = \text{F}$ & 54.7 & 0.504  & 0.197 \\
				20  &$\bullet = 2$ & 94.3 & 0.540  & 0.215 \\
				20  & $\bullet = \text{F}$ & 94.3 & 0.505   & 0.202 \\
				40  &  $\bullet = 2$ & 436 & 0.496  & 0.231 \\
				40 &  $\bullet = \text{F}$ & 436 & 0.477   & 0.221 \\
				\hline
				\multicolumn{2}{c}{Largest std.~err.} & 130 & 0.189  & 0.076 \\
			\end{tabular}
		\end{center}
		\caption{Simulation averages of $s_\tau(\Sigma^{-1})$ and the relative estimation errors for estimators exploiting an assumption of sparsity on the inverse and logarithmic scales under the $\Sigma_{pd}$ parametrization.\label{tabPD}}
	\end{table}
	
	\begin{table}
		\begin{center}
			\begin{tabular}{ccccccc}
				&             &       &                    &  \multicolumn{2}{c}{Estimator $E$} \\
				\cline{5-6} 
				$s^*$  & $\rho$ & $\frac{\|E-\Sigma^{-1}\|_\bullet}{\|\Sigma^{-1}\|_\bullet}$ & $s_\tau(\Sigma^{-1})$  & \;\;\; $\hat \Sigma^{-1}_{\tau}$\;\;\;  &  $\hat O_\tau \hat \Lambda^{-1} \hat O_\tau^\T$ \\
				\hline 
				6  & 2 &$\bullet = 2$ & 91.3 &  0.561  & 1.264 \\
				6  & 2  &$\bullet = \text{F}$ & 91.3 & 0.540  & 1.451 \\
				6  & 4 &$\bullet = 2$ & 132&  0.565  & 1.281 \\
				6  & 4 &$\bullet = \text{F}$ & 132 &  0.551  & 1.469 \\
				10 & 2 &$\bullet = 2$  & 91.7 & 0.599  & 1.305 \\
				10 & 2  &$\bullet = \text{F}$  & 91.7 & 0.561  & 1.467 \\
				10 & 4 &$\bullet = 2$ & 131 & 0.604  & 1.308 \\
				10 & 4 &$\bullet = \text{F}$ & 131 & 0.571  & 1.480 \\
				20 & 2 &$\bullet = 2$  & 110 &  0.597  & 1.340 \\
				20 & 2  &$\bullet = \text{F}$ & 110 & 0.564  & 1.521 \\
				20 & 4 &$\bullet = 2$ & 155 & 0.602  & 1.355 \\
				20 & 4 &$\bullet = \text{F}$ & 155 & 0.576  & 1.539 \\
				\hline
				\multicolumn{3}{c}{Largest standard error} & 96.2 & 0.205 & 0.328 \\
			\end{tabular}
		\end{center}
		\caption{Simulation averages of $s_\tau(\Sigma^{-1})$ and the relative estimation errors for estimators exploiting an assumption of sparsity on the inverse and logarithmic scales under the $\Sigma_{o}$ parametrization.\label{tabO}}
	\end{table}
	
	The performance of $\hat O_\tau \hat \Lambda^{-1} \hat O_\tau^\T$, as reported in Table \ref{tabO}, is relatively poor, suggesting that the thresholding approach is too simplistic for this case. One issue concerns the constraints on $\alpha$ needed to make the $\Sigma_o$ parametrization injective (see Proposition \ref{propInjectivity}), which are not naturally accommodated by the thresholding estimator. Another aspect is the distortion of the distribution of matrix entries by the matrix logarithm, and its possible effect on the estimation error, which has not been formally studied for the $\Sigma_o$ parametrization. \citet{RybakBattey2021}  noted a different estimator that does not involve taking matrix logarithms of sample quantities and that accommodates constraints on $\alpha$. The formal implementation and theoretical justification of that approach requires major work not taken up here.

	\begin{table}
		\begin{center}
			\begin{tabular}{ccccccc}
				&             &                           &  \multicolumn{2}{c}{Estimator $E$} \\
				\cline{4-5} 
				$s^*$ & $\frac{\|E-\Sigma^{-1}\|_\bullet}{\|\Sigma^{-1}\|_\bullet}$  & $s_\tau(\Sigma^{-1})$  & \;\;\; $\hat \Sigma^{-1}_{\tau}$\;\;\;  & $(\hat V_\tau \hat V_\tau^\T)^{-1}$ \\
				\hline
				6   &  $\bullet = 2$ & 42.6 &  0.590   & 0.148 \\
				6   &  $\bullet = \text{F}$ & 42.6 & 0.508  & 0.129 \\
				10 & $\bullet = 2$ & 50.4 &  0.581  & 0.147 \\
				10  &  $\bullet = \text{F}$ & 50.4 & 0.514   & 0.131 \\
				20  &$\bullet = 2$ & 74.5 &  0.551  & 0.164 \\
				20  & $\bullet = \text{F}$ & 74.5 & 0.516  & 0.153 \\
				40  &  $\bullet = 2$ & 162 & 0.503  & 0.203 \\
				40 &  $\bullet = \text{F}$ & 162 & 0.481  & 0.191 \\
				\hline
				\multicolumn{2}{c}{Largest std.~err.} & 47.7 & 0.198  & 0.112 \\
			\end{tabular}
		\end{center}
		\caption{Simulation averages of $s_\tau(\Sigma^{-1})$ and the relative estimation errors for estimators exploiting an assumption of sparsity on the inverse and logarithmic scales under the $\Sigma_{lt}$ parametrization.\label{tabLT}}
	\end{table}

	\begin{table}
		\begin{center}
			\begin{tabular}{ccccccc}
				&             &       &                    &  \multicolumn{2}{c}{Estimator $E$} \\
				\cline{5-6} 
				$s^*$  & $\rho$ & $\frac{\|E-\Sigma^{-1}\|_\bullet}{\|\Sigma^{-1}\|_\bullet}$ & $s_\tau(\Sigma^{-1})$  & \;\;\; $\hat \Sigma^{-1}_{\tau}$\;\;\;  & $(\hat U_\tau \hat D \hat U_\tau^\T)^{-1}$ \\
				\hline
				6  & 2 &$\bullet = 2$ & 90.7 &  0.526  & 0.405 \\
				6  & 2  &$\bullet = \text{F}$ & 90.7 &  0.503   & 0.372 \\
				6  & 4 &$\bullet = 2$ & 130 &  0.527   & 0.405 \\
				6  & 4 &$\bullet = \text{F}$ & 130 & 0.506  & 0.373 \\
				10 & 2 &$\bullet = 2$  & 110 & 0.547  & 0.431 \\
				10 & 2  &$\bullet = \text{F}$  & 110 & 0.520  & 0.396 \\
				10 & 4 &$\bullet = 2$ & 157 & 0.548  & 0.425 \\
				10 & 4 &$\bullet = \text{F}$ & 157 & 0.522   & 0.392 \\
				20 & 2 &$\bullet = 2$  & 163 &  0.548  & 0.482 \\
				20 & 2  &$\bullet = \text{F}$ & 163 &  0.514   & 0.448 \\
				20 & 4 &$\bullet = 2$ & 235 & 0.548  & 0.486 \\
				20 & 4 &$\bullet = \text{F}$ & 235 & 0.514  & 0.453 \\
				\hline
				\multicolumn{3}{c}{Largest standard error} & 83.2 & 0.189 & 0.310 \\
			\end{tabular}
		\end{center}
		\caption{Simulation averages of $s_\tau(\Sigma^{-1})$ and the relative estimation errors for estimators exploiting an assumption of sparsity on the inverse and logarithmic scales under the $\Sigma_{ltu}$ parametrization.\label{tabLTU}}
	\end{table}

	\section{Application to leukemia data}\label{appReal}
	
	The data \citep[][\S 19.1]{EH} consist of 3571 features for 72 patients. Of these, 47 have acute lymphoblastic leukaemia and 25 have acute myeloid leukemia. We used linear discriminant analysis with the sample covariance matrix replaced by a thresholded estimator on each of the scales considered in the paper, in order to assess the ultimate classification performance. Since the estimator $ \exp(\hat{L}_{\tau}) $ requires the sample covariance matrix to be positive definite, which fails to hold if $ n < p$, we replace $ \hat{\Sigma}  $ by $ \hat{\Sigma} + \delta_{p, n} \text{diag}(\hat{\Sigma}) $, where $ \delta_{p,n} = (\log (p)/ n)^{1/2} $; this choice was justified by \citet{Battey2019}. For estimators $ \hat{\Sigma}_{\tau} $, $ \hat{O}_{\tau} \hat{\Lambda}^{-1} \hat{O}^{\T}_{\tau} $ and $ \hat{U}_{\tau} \hat{D} \hat{U}_{\tau}^{\T} $ we considered both $ \hat{\Sigma} $ and $ \hat{\Sigma} + \delta_{p, n} \text{diag}(\hat{\Sigma}) $ as pilot estimators, and report the lower of the two error rates in Table \ref{tab:leukemia}.
	
	The misclassification rates were obtained by randomly splitting the sample into two sets, consisting of 50 and 22 patients respectively, the smaller subset serving as a hold-out for testing classification performance on the basis of the larger training set. To select a threshold, the larger subset is itself split into a training (80\%) and a validation set (20\%) ten times. For each method, we select a threshold that minimizes validation error over the ten splits. The final classifier is estimated using all $ 50 $ patients and its out-of-sample performance is calculated using the hold-out sample. The procedure is repeated $ 50 $ times, which results in a set of $ 50 $ out-of-sample misclassification rates for each method. Results are reported in Table \ref{tab:leukemia}.

	\begin{table}[h!]
		\begin{center}
			\begin{tabular}{l|cccc} 
				Test error &  $ \hat{\Sigma}_{\tau} $ & $  \hat{O}_{\tau} \hat{\Lambda}^{-1} \hat{O}^{\T}_{\tau} $ & $\exp(\hat{L}_{\tau}) $ & $ \hat{U}_{\tau} \hat{D} \hat{U}_{\tau}^{\T} $
				\\
				\hline
				Mean & 6\% & 4\% & 4\% & 4\% \\
				Median & 5\% & 5\% & 5\% & 2\% \\
				s.~e. & 6\%& 5\% & 4\% & 5\% \\
				\hline
			\end{tabular}
			\caption{Average, median and standard error of accuracy scores on a hold-out dataset. Calculated over $ 50 $ randomly chosen test sets. \label{tab:leukemia}}
		\end{center}
	\end{table}

\end{appendix}

\end{document}